\documentclass[times,10pt]{elsarticle}

\usepackage[utf8]{inputenc}

\usepackage[english]{babel}

\usepackage{amsmath}
\usepackage{amsfonts}
\usepackage{amssymb}
\usepackage{geometry}

\usepackage[noline]{algorithm2e}

\usepackage{graphicx}
\usepackage{epstopdf}
\usepackage{grffile} %

\usepackage{pdflscape}

\usepackage[numbers]{natbib}
\usepackage{hyperref} %
\usepackage{caption}
\usepackage{subcaption}
\usepackage{float}
\usepackage{xspace}

\usepackage[amsmath,thmmarks]{ntheorem}
\numberwithin{equation}{section}

\newtheorem{theorem}{Theorem}[section]

\newtheorem{remark}[theorem]{Remark}

\newtheorem{corollary}[theorem]{Corollary}

\theoremstyle{nonumberplain}
\theorembodyfont{\normalfont}
\theoremsymbol{\ensuremath{\blacksquare}}
\newtheorem{proof}{Proof}

\usepackage{siunitx}

\usepackage[nameinlink,capitalise]{cleveref}
\usepackage{booktabs}

\usepackage{dcolumn}
\newcolumntype{d}{D{.}{.}{-1}}

\usepackage{color}
\definecolor{newgreen}{RGB}{0,85,68}
\definecolor{aggrored}{RGB}{255,47,47}
\definecolor{steelblue3}{RGB}{33,29,132}
\definecolor{greenish}{RGB}{0, 153, 51}
\definecolor{blueish}{RGB}{51, 102, 204}
\definecolor{orangeish}{RGB}{255, 153, 51}
\definecolor{pinkish}{RGB}{205,147,204}
\definecolor{lightgrey}{RGB}{211, 211, 211}
\definecolor{darkgrey}{RGB}{111, 111, 111}

\usepackage{mathtools}
\DeclarePairedDelimiter\abs{\lvert}{\rvert}

\DeclarePairedDelimiterX\set[2]{\{}{\}}{\,#1 \;\delimsize\vert\; #2\,}
\newcommand{\sign}{\operatorname{sign}}

\usepackage{aas_macros}
\numberwithin{equation}{section}

\newcommand{\ud}{\mathrm{d}}

\renewcommand{\epsilon}{\ensuremath\varepsilon}

\newcommand   {\qpd}  [2][]{\partial_{#1}#2}

\newcommand   {\fd}   [2][]{\ensuremath{\frac{d}{d#1}{#2}}}
\newcommand   {\grad} [2][]{\ensuremath{\nabla_{#1}#2}}
\renewcommand {\div}  [2][]{\ensuremath{\nabla_{#1}\cdot #2}}

\newcommand   {\Be}{\epsilon}

\newcommand   {\sndxtime}{\tau_{\mathrm{char}}}
\newcommand   {\tend}{t_f}
\newcommand   {\stencil}{\Sigma}

\renewcommand{\vec}[1]{\ensuremath{\boldsymbol{#1}}}

\newcommand{\bbf}{\vec{f}}
\newcommand{\bF}{\vec{F}}
\newcommand{\bg}{\vec{g}}
\newcommand{\bG}{\vec{G}}

\newcommand{\bbs}{\vec{s}}
\newcommand{\bS}{\vec{S}}
\newcommand{\bu}{\vec{u}}
\newcommand{\bU}{\vec{U}}
\newcommand{\bv}{\vec{v}}

\newcommand{\bw}{\vec{w}}
\newcommand{\bW}{\vec{W}}

\newcommand{\cF}{\ensuremath{\mathcal{F}}}

\newcommand{\cL}{\ensuremath{\mathcal{L}}}

\newcommand{\cR}{\ensuremath{\mathcal{R}}}

\newcommand{\cW}{\ensuremath{\mathcal{W}}}

\newcommand{\algo}[1]{\textsc{#1}\xspace}

\newcommand{\domain}{\ensuremath{\Omega}\xspace}
\newcommand{\cell}{I}
\newcommand{\scell}{V}

\newcommand{\minmod}{\algo{MinMod}}

\newcommand{\half}{\ensuremath{1/2}}

\newcommand{\adidx}{\ensuremath{\gamma}}

\newcommand{\eos}{EoS\xspace}

\newcommand{\const}{\text{const}}

\newcommand{\err}{\operatorname{err}}

\newcommand{\etal}{et al.}

\begin{document}

\begin{frontmatter}

\title{Well-balanced finite volume schemes for nearly steady adiabatic flows}

\author[auth1]{L.~Grosheintz-Laval \corref{cor1}}
\cortext[cor1]{Corresponding author}
\ead{luc.grosheintz@sam.math.ethz.ch}
\author[auth1]{R.~K\"appeli}

\address[auth1]{
Seminar for Applied Mathematics (SAM),
Department of Mathematics,
ETH Z\"urich,
CH-8092 Z\"urich,
Switzerland}

\begin{keyword}
Numerical methods \sep
Hydrodynamics \sep
Source terms \sep
Well-balanced schemes

\end{keyword}

\begin{abstract}
  We present well-balanced finite volume schemes designed to approximate the Euler
  equations with gravitation.
  They are based on a novel local steady state reconstruction.
  The schemes preserve a discrete equivalent of steady adiabatic
  flow, which includes non-hydrostatic equilibria.
  The proposed method works in Cartesian, cylindrical and spherical coordinates.
  The scheme is not tied to any specific numerical flux and can be combined with
  any consistent numerical flux for the Euler equations, which provides great
  flexibility and simplifies the integration into any standard finite volume
  algorithm.
  Furthermore, the schemes can cope with general convex equations of state, which
  is particularly important in astrophysical applications.
  Both first- and second-order accurate versions of the schemes and their
  extension to several space dimensions are presented.
  The superior performance of the well-balanced schemes compared to standard
  schemes is demonstrated in a variety of numerical experiments.
  The chosen numerical experiments include simple one-dimensional problems in
  both Cartesian and spherical geometry, as well as two-dimensional simulations
  of stellar accretion in cylindrical geometry with a complex multi-physics
  equation of state.

\end{abstract}

\end{frontmatter}

\section{Introduction}
\label{sec:intro}
A great variety of physical phenomena can be modeled by the Euler equations
with gravitational forces.
Applications extend from the study of atmospheric phenomena, such as
numerical weather prediction and climate modeling, to the numerical
simulation of the climate of exoplanets, convection in stars, accretion
processes, and stellar explosions.
The Euler equations with gravity source terms express the conservation of mass,
momentum and energy as
\begin{align}\label{eq:euler}
  \begin{split}
  \qpd[t]{\rho} + \div{\rho \bv} &= 0 , \\
  \qpd[t]{\rho \bv} + \div{\left(\rho \bv \otimes \bv\right)} + \grad{p}
    &= - \rho \grad{\phi} , \\
  \qpd[t]{E} + \div{\left[ (E + p) \bv \right]} &= - \rho \bv \cdot \grad{\phi}
  ,
  \end{split}
\end{align}
where $\rho$ is the mass density and $\bv$ the velocity.
The total fluid energy $E = \rho e + \frac{1}{2} \rho v^2$ is the sum of internal
and kinetic energy densities.
An equation of state closes the system and describes the relation between
density, specific internal energy $e$ and the pressure $p=p(\rho,e)$.

The source terms model the influence of gravity onto the fluid through the
gravitational potential $\phi$.
The latter may either be a fixed function or, in the case of self-gravity, be
determined by the fluid's mass distribution through Poisson's equation
\begin{equation}
\label{eq:intro0020}
  \nabla^2 \phi = 4 \pi G \rho,
\end{equation}
where $G$ is the gravitational constant.

The Euler equations with gravitation \eqref{eq:euler} are a typical system of
balance laws
\begin{align}
\label{eq:intro0030}
  \qpd[t]{}\bu + \div[]{}\bbf(\bu) = \bbs(\bu)
  ,
\end{align}
where $\bu$, $\bbf = \bbf(\bu)$ and $\bbs = \bbs(\bu)$ are the conserved
variables, fluxes and source terms, respectively.
A distinctive feature of systems of balance laws is the presence of non-trivial
steady states
\begin{equation}
\label{eq:intro0040}
  \div[]{}\bbf(\bu) = \bbs(\bu)
  ,
\end{equation}
which are characterized by a subtle flux-source balance.

A rich class of steady states for the Euler equations \eqref{eq:euler} is
the hydrostatic equilibrium
\begin{equation}
\label{eq:intro0050}
  \grad{p} = - \rho \grad{\phi}
  ,
\end{equation}
where gravity forces are balanced by pressure forces.
As a matter of fact, \cref{eq:intro0050} specifies only a mechanical
equilibrium.
To fully characterize the equilibrium, a thermal stratification needs to be
supplemented.
For isentropic conditions, \cref{eq:intro0050} can easily be integrated to
\begin{equation}
\label{eq:intro0080}
  h + \phi = \const
  ,
\end{equation}
where $h$ is the specific enthalpy.
For different thermal stratifications, such as isothermal or generally for
barotropic fluids (in which density is a function of pressure only),
\cref{eq:intro0050} can be integrated into similar forms
(see e.g. \cite{LandauLifshitz1987}).
In many applications, the dynamics of interest are taking place near such a
steady state.
This is for example the case in numerical weather prediction and climate
modeling \cite{Botta2004539}, the simulation of waves in stellar atmospheres
\cite{Fuchs20104033,GundlachLeVeque2011}, and the simulation of convection in
stars \cite{PopovEtAl2019}.

Another class of steady states is provided by steady adiabatic flow for which
Bernoulli's equation
\begin{equation}
\label{eq:intro0060}
  \frac{v^{2}}{2} + h + \phi = \const
\end{equation}
holds along each streamline, but may differ from streamline to streamline
(see e.g. \cite{LandauLifshitz1987}).
In many astrophysical applications, the dynamics of interest are realized near
steady flow such as in accretion or wind phenomena
\cite{HolzerAxford1970,FrankKingRaine2002,PringleKing2007}.

Solutions to systems of balance laws can often only be approximated with the
help of numerical methods.
There exist several types of accurate and robust discretization methods such
as finite difference, finite volume, and discontinuous Galerkin (DG) methods.
However, standard numerical methods have in general difficulties near steady
states as they do not necessarily satisfy a discrete equivalent of the
flux-source balance \cref{eq:intro0040}.
Hence such states are not preserved exactly but are only approximated with an
error proportional to the truncation error of the scheme.
If the interest relies in the simulation of small perturbations on top of a
steady state, the numerical resolution has to be increased to the point where
the truncation errors do not obscure these small perturbations.
This may result in prohibitively high computational costs, especially in
multiple dimensions.

To overcome the difficulties of standard discretization methods, the
well-balanced design principle was introduced by Cargo \& LeRoux and
Greenberg \& LeRoux \cite{CargoLeRoux1994,Greenberg1996}.
In well-balanced schemes, a discrete equivalent of the steady state of interest
is exactly preserved.
Many such schemes have been developed, especially for the shallow water
equations with non-trivial bottom topography, see e.g.
\cite{LeVeque1998,Gosse2000,Audusse2004,Noelle2006,NoelleEtAl2007,CastroEtAl2007} and
references therein.
An extensive review on well-balanced and related schemes for many applications
is also given in the book by Gosse \cite{Gosse2013}.

Well-balanced schemes for the Euler equations with gravitation have received
much attention in the literature recently.
Most of the schemes focus on hydrostatic case \cref{eq:intro0050}.
Pioneering schemes have been developed by Cargo \& LeRoux
\cite{CargoLeRoux1994}, and LeVeque \& Bale \cite{LeVequeEtAl1998,LeVeque1999}.
The latter apply the quasi-steady wave-propagation algorithm of LeVeque
\cite{LeVeque1998} to the Euler equations with gravity.
Botta et al. \cite{Botta2004539} designed a well-balanced finite volume scheme
for numerical weather prediction applications.
More recently, several well-balanced finite volume
\cite{LeVeque2011,Kaeppeli2014,Desveaux2014,Chandrashekar2015,Kaeppeli2016,ToumaKoleyKlingenberg2016,LiXing2016,Kaeppeli2017,ChertockCuiKurganovEtAl2018,Gaburro2018,BerberichEtAl2019a,BerberichEtAl2019,KlingenbergEtAl2019,Krause2019,CastroPares2020,BerberichEtAl2020},
finite difference \cite{Xing2013,LiXing2018a}
and discontinuous Galerkin \cite{Li2015,ChandrashekarZenk2017,LiXing2018}
schemes have been presented.
Magneto-hydrostatic steady state preserving well-balanced finite volume schemes
were devised in \cite{Fuchs20104033,FuchsEtAl2010a,FuchsEtAl2011}.

Well-balanced schemes for steady adiabatic flow have received much less
attention in the literature.
LeVeque \& Bale \cite{LeVeque1999} adapted the quasi-steady wave propagation
algorithm to handle steady states with non-zero velocity.
More recently, Bouchut \& de Luna \cite{BouchutLuna2010} have constructed a
well-balanced scheme for subsonic states of the Euler-Poisson system.

In this paper, we extend the well-balanced second-order finite volume schemes
of K\"{a}ppeli \& Mishra \cite{Kaeppeli2014} to steady adiabatic flow.
The schemes possess the following novel features:
\begin{itemize}
\item They are well-balanced for steady adiabatic flow by using the Bernoulli
      equation \cref{eq:intro0060} for the local equilibrium preserving
      reconstruction and gravitational source terms discretization.
      Subsonic, supersonic and transonic steady states are captured.
\item They are well-balanced for any consistent numerical flux.
      This allows a straightforward implementation within any standard finite
      volume algorithm.
      For numerical fluxes capable of recognizing stationary shock waves
      exactly, the schemes are able to preserve steady flow with shocks
      located at cell interfaces.
\item They are applicable to general convex equations of states, which is
      particularly important for astrophysical applications.
\item They are designed for Cartesian, cylindrical and spherical coordinate
      systems, which are for instance often employed in astrophysical
      applications.
\item They are extended to several space dimensions and are well-balanced for
      steady adiabatic flow with grid-aligned streamlines.
\end{itemize}
The rest of the paper is organized as follows: the well-balanced finite volume
schemes are presented in \cref{sec:nm}.
Numerical results are presented in \cref{sec:numex} and a summary of the paper
is provided in \cref{sec:conc}.

\section{Numerical Method}
\label{sec:nm}
\subsection{Numerical method in one dimension}
\label{sec:nm_1d}
We consider the one-dimensional Euler equations with gravitation
\cref{eq:euler} in the following compact form of a balance law
\begin{align}\label{eq:balance_short}
    \qpd[t]{}\bu + \qpd[x]{} \bbf = \bbs
  ,
\end{align}
where
\begin{equation}
  \label{eq:euler_short}
  \bu = \begin{bmatrix}
          \rho     \\
          \rho v \\
          E
        \end{bmatrix}
  \; , \quad
  \bbf = \begin{bmatrix}
           \rho v       \\
           \rho v^2 + p \\
           (E + p) v
         \end{bmatrix}
  \; , \quad
  \mathrm{and} \quad
  \bbs = - \begin{bmatrix}
             0        \\
             \rho     \\
             \rho v
           \end{bmatrix} \qpd[x]{\phi}
\end{equation}
denote the vector of conserved variables, fluxes and source terms,
respectively.
The primitive variables will be denoted by $\bw = [\rho,v,p]^T$.
Furthermore, we introduce the notation $f^{(\rho)}$, $f^{(\rho
v)}$ and $f^{(E)}$ for the mass, momentum and energy flux, respectively.
We will use the same superscript notation to indicate specific components of
the source term, e.g. $s^{(\rho v)}$ denotes the momentum source term.

Next, we briefly outline a standard first- and second-order finite volume
discretization of the above equations to fix the notation.
Subsequently, we describe our novel well-balanced schemes in detail.

\subsubsection{Standard finite-volume discretization}\label{sec:standard_fvm}
The spatial domain of interest is discretized into a number of finite
volumes or cells $\cell_i = [x_{i-\half}, x_{i+\half}]$.
For the $i$-th cell $\cell_{i}$, the $x_{i\pm\half}$ denote the left/right
cell interfaces and the $x_{i} = (x_{i-\half}+x_{i+\half})/2$ the cell
centers.
For simplicity, we assume uniform cell sizes
$\Delta x = x_{i+\half}-x_{i-\half}$.
However, varying cell size can easily be accommodated for.

A one-dimensional semi-discrete finite volume method for
\cref{eq:euler_short} then reads
\begin{align}
\label{eq:nm_1d_fv_0010}
    \fd[t]{}\bU_i
  = \cL(\bU)
  =  - \frac{1}{\Delta x}\left(\bF_{i+\half} - \bF_{i-\half}\right)
     + \bS_i
  .
\end{align}
Here $\bU_i = \bU_{i}(t)$ denotes the approximate average of the solution
$\bu(x,t)$ over cell $\cell_{i}$,
\begin{equation}
  \label{eq:nm_1d_fv_0020}
  \bU_i(t) \approx \frac{1}{\Delta x} \int_{\cell_{i}} \bu(x,t) ~ \ud x
  ,
\end{equation}
the $\bF_{i\pm\half}$ the numerical fluxes through the left/right cell
interface and $\bS_{i}$ the approximate cell average of the source term.
Moreover, the shorthand $\cL(\bU)$ is introduced for the spatial discretization
operator.

\paragraph{Numerical Flux}
The numerical fluxes are obtained by the (approximate) solution of a Riemann
problem at each cell interface
\begin{equation}
  \label{eq:nm_1d_fv_0030}
  \bF_{i+1/2} = \cF ( \bW_{i+1/2-} ,\bW_{i+1/2+} ),
\end{equation}
where $\cF$ denotes a consistent, i.e. $\cF(\bW,\bW) = \bbf(\bW)$, and
Lipschitz continuous numerical flux function.
In the numerical experiments, we will use the
HLL(E)~\cite{HartenEtAl1983,Einfeldt1988} and HLLC~\cite{Toro1994} solvers with
carefully chosen waves speeds allowing the resolution of isolated flow
discontinuities (see e.g.~\cite{Batten1997,Toro2009}).

\paragraph{Reconstruction}
Input to the numerical flux function are the traces of the primitive
variables $\bW_{i+1/2\pm}$ at the cell interface.
These are obtained by some non-oscillatory reconstruction procedure $\cR$
\begin{equation}
  \label{eq:nm_1d_fv_0040}
  \bW_{i}(x) = \cR(x; \{\bW_{k}\}_{k \in \stencil_{i}}),
\end{equation}
where $\stencil_{i}$ is the stencil of the reconstruction procedure for cell
$\cell_{i}$.
The left/right cell interface traces of the primitive variables are then
simply obtained by evaluating the reconstruction in cell
$\cell_{i}/\cell_{i+1}$ at cell interface $x_{i+1/2}$
\begin{equation}
  \label{eq:nm_1d_fv_0050}
  \bW_{i+1/2-} = \bW_{i  }(x_{i+1/2})
               = \cR(x_{i+1/2}; \{\bW_{k}\}_{k \in \stencil_{i  }})
  \quad \text{and} \quad
  \bW_{i+1/2+} = \bW_{i+1}(x_{i+1/2})
               = \cR(x_{i+1/2}; \{\bW_{k}\}_{k \in \stencil_{i+1}})
  .
\end{equation}
Many such reconstruction procedures have been elaborated in the literature
and an incomplete list includes the Total Variation Diminishing (TVD) and
the Monotonic Upwind Scheme for Conservation Laws (MUSCL) methods
(see e.g.~\cite{Leer1979,HartenEtAl1983,Sweby1984,Laney1998,LeVeque2002,Toro2009}),
the Piecewise Parabolic Method (PPM)~\cite{ColellaWoodward1984}, the
Essentially Non-Oscillatory (ENO) (see e.g.~\cite{HartenEtAl1987}),
Weighted ENO (WENO) (see e.g.~\cite{Shu2009} and references therein) and
Central WENO (CWENO) methods (see e.g.~\cite{CraveroEtAl2017}).

In the schemes developed below, we will use spatially first- and second-order
accurate TVD/MUSCL type reconstructions.
Up to this spatial accuracy, point values and cell averages agree and the
cell-averaged primitive variables are simply obtained from the cell-averaged
conserved variables $\bW_{i} = \bw(\bU_{i})$.
Then, a spatially first-order accurate piecewise constant reconstruction is
simply given by
\begin{equation}
  \label{eq:nm_1d_fv_0060}
  \bW_{i}(x) = \cR(x; \{\bW_{i}\}) = \bW_{i}
  .
\end{equation}
A spatially second-order accurate piecewise linear reconstruction is
\begin{equation}
  \label{eq:nm_1d_fv_0061}
  \bW_{i}(x) = \cR(x; \{\bW_{i-1}, \bW_{i}, \bW_{i+1}\})
             = \bW_{i} + D\bW_{i} ~ (x - x_{i})
  ,
\end{equation}
where $D\bW_{i}$ are some appropriately limited slopes.
Below we will make use of the so-called generalized \minmod slope limiter
family
\begin{equation}
  \label{eq:nm_1d_fv_0062}
  D\bW_{i} = \minmod \left(
                       \theta \frac{\bW_{i  } - \bW_{i-1}}{  \Delta x} ,
                              \frac{\bW_{i+1} - \bW_{i-1}}{2 \Delta x} ,
                       \theta \frac{\bW_{i+1} - \bW_{i  }}{  \Delta x}
                     \right)
  ,
\end{equation}
where $\theta \in [1,2]$ is a parameter and the \minmod function is defined
by
\begin{equation}
  \label{eq:nm_1d_fv_0063}
    \minmod(a_1,a_2,...)
  = \begin{cases}
      \min_j \left\{ a_j \right\} & \mathrm{if} \; a_j > 0 \; \forall \; j , \\
      \max_j \left\{ a_j \right\} & \mathrm{if} \; a_j < 0 \; \forall \; j , \\
      0                           & \mathrm{otherwise}
      .
    \end{cases}
\end{equation}
\cref{eq:nm_1d_fv_0062} has to be understood component-wise.
For $\theta = 1$ ($\theta = 2$), \cref{eq:nm_1d_fv_0062} reproduces the
traditional \minmod (monotonized centered) limiter
(see e.g.~\cite{NessyahuTadmor1990,JiangTadmor1998}).

\paragraph{Source terms}
The standard second-order discretization of the cell-averaged source term is
simply the physical source term evaluated at the cell center
\begin{equation}\label{eq:ub_src}
  \bS_{i} = - \begin{bmatrix}
                0        \\
                \rho_{i} \\
                \rho v_{i}
              \end{bmatrix} \qpd[x]{}{\phi}(x_i)
  \approx \frac{1}{\Delta x} \int_{\cell_{i}} \bbs(\bu(x,t)) ~ \ud x
  .
\end{equation}
Here the gravitational acceleration may either be calculated analytically
or with finite differences
\begin{align}\label{eq:ub_src_fd}
  \qpd[x]{}{\phi}(x_i)
  \approx \frac{\phi_{i+\half} - \phi_{i-\half}}{\Delta x}
  ,
\end{align}
where $\phi_{i\pm\half} \approx \phi(x_{i\pm\half})$ is an approximation of
the gravitational potential at cell interfaces.

\paragraph{Temporal discretization}
The temporal domain of interest $[0,\tend]$ is discretized into time steps
$\Delta t^{n} = t^{n+1} - t^{n}$, where the superscript labels the respective
time levels.
The system of ordinary differential equations \cref{eq:nm_1d_fv_0010} can be
integrated in time with the strong stability-preserving Runge-Kutta methods
(see \cite{GottliebShuTadmor2001} and references therein).
In particular, we will use the temporally first-order accurate Euler method
\begin{equation}
\label{eq:SSP-RK1}
  \bU_{i}^{n+1} = \bU_{i}^{n} + \Delta t^n \cL_i(\bU^{n})
\end{equation}
and second-order accurate Heun method (SSP-RK2)
\begin{align}
  \label{eq:SSP-RK2}
  \begin{split}
    \bU_{i}^{(1)} & = \bU_{i}^{n} + \Delta t^n \cL_i(\bU^{n}), \\
    \bU_{i}^{(2)} & = \bU_{i}^{(1)} + \Delta t^n \cL_i(\bU^{(1)}), \\
    \bU_{i}^{n+1} & = \frac{1}{2} \left(\bU_{i}^{n} + \bU_{i}^{(2)}\right)
    .
  \end{split}
\end{align}
Since the above choices are explicit in time, the time step $\Delta t^{n}$ is
required to fulfill the CFL condition for finite volume methods with a CFL
number specified in the numerical experiments.
However, because the presented methods are only concerned with the spatial
reconstruction procedure and source term discretization, the derived
techniques are, in principle, not restricted to explicit time integrators.

This concludes the description of a standard spatially and temporally
first/second-order accurate finite volume method for the one-dimensional
balance law \cref{eq:balance_short}.
We refer to the excellent textbooks available in the literature for further
details and derivations, e.g.
\cite{GodlewskiRaviart1996,Laney1998,LeVeque2002,Hirsch2007,Toro2009}.
However, it turns out that standard schemes, as just outlined, are in general
not capable of preserving a discrete equivalent of steady adiabatic flow
\cref{eq:intro0060}.
Next, we describe the components allowing the exact (up to machine precision)
discrete preservation of such steady states.

\subsection{Well-balanced finite volume discretization}
\label{subsec:nm_1d_wb}
In this section, we describe the modifications required to well-balance the
standard finite volume scheme from \cref{sec:standard_fvm}.
One-dimensional steady adiabatic flow is given by
\begin{equation}
\label{eq:nm_1d_wb_0010}
  s = \const
  , \quad
  \rho v = \const
  , \quad
  \frac{v^{2}}{2} + h + \phi = \const.
\end{equation}
The first constant expresses the fact that the flow proceeds adiabatically,
i.e. the flow is isentropic.
The second and third constants are a consequence of mass and energy
conservation, respectively.
In order to construct a well-balanced scheme, one requires the usual three
components: (i) a local equilibrium profile $\bW_{eq,i}(x)$ within each cell
$I_{i}$, (ii) a well-balanced equilibrium preserving reconstruction and (iii) a
well-balanced source term discretization.

For clarity of presentation, we begin with a detailed description of the
well-balanced equilibrium preserving reconstruction in \cref{sec:wb_rc}
followed by the well-balanced source term discretization in
\cref{sec:wb_src_cart}.
In both sections, we assume that the local equilibrium profile fulfilling
\cref{eq:nm_1d_wb_0010} in each cell is given by
\begin{equation}
\label{eq:nm_1d_wb_0020}
  \bW_{eq,i}(x) = \begin{bmatrix}
                    \rho_{eq,i}(x) \\
                    v_{eq,i}(x)    \\
                    p_{eq,i}(x)
                  \end{bmatrix}
  .
\end{equation}
The constants in \cref{eq:nm_1d_wb_0010} are fixed at the cell center, meaning
that the local equilibrium profile is anchored at the cell center, i.e.\
\begin{equation}
\label{eq:nm_1d_wb_0030}
  \bW_{eq,i}(x_{i}) = \bW_{i}
  .
\end{equation}
The determination of the local equilibrium profile is subsequently presented in
great detail in \cref{sec:local_equilibrium}.

\subsubsection{Well-balanced reconstruction}\label{sec:wb_rc}
In the following, we present the necessary modifications to the standard
reconstruction procedure $\mathcal{R}$ in \cref{sec:standard_fvm}.
This will result in a well-balanced equilibrium preserving reconstruction
procedure we shall denote by $\mathcal{W}$.

Given the local equilibrium profile, the modification of the first-order
accurate reconstruction \cref{eq:nm_1d_fv_0060} is simply the replacement of
the piecewise constant representation by the local equilibrium profile
\begin{equation}
\label{eq:wb_rc_0010}
  \bW_{i}(x) = \mathcal{W}(x; \left\{ \bW_{i} \right\})
             = \bW_{eq,i}(x)
  .
\end{equation}

For the second-order accurate reconstruction \cref{eq:nm_1d_fv_0061}, the
well-balanced equilibrium reconstruction is decomposed into two additive terms,
one for the equilibrium and another for a (possibly large) perturbation therefrom.
The equilibrium term is simply given by the local equilibrium profile
$\bW_{eq,i}(x)$.
The equilibrium perturbation reconstruction is obtained by applying a standard
piecewise linear reconstruction \cref{eq:nm_1d_fv_0061} to the equilibrium
perturbation.
The data for this reconstruction is obtained by extrapolating the local
equilibrium profile $\bW_{eq,i}(x)$ of the $i$-th cell to the neighboring
cells $I_{i-1}$ and $I_{i+1}$:
\begin{equation}
\label{eq:wb_rc_0020}
  \delta \bW_{i-1} = \bW_{i-1} - \bW_{eq,i}(x_{i-1})
  \quad \text{and} \quad
  \delta \bW_{i+1} = \bW_{i+1} - \bW_{eq,i}(x_{i+1})
  .
\end{equation}
Note that $\delta \bW_{i} = \bW_{i} - \bW_{eq,i}(x_{i}) = 0$ holds by
construction, since the equilibrium profile within cell $I_{i}$ is anchored at
cell center \cref{eq:nm_1d_wb_0030}.
Thereby, we obtain
\begin{equation}
\label{eq:wb_rc_0030}
  \bW_{i}(x) = \mathcal{W}(x; \left\{ \bW_{i-1},\bW_{i},\bW_{i+1} \right\})
             = \bW_{eq,i}(x)
             + \mathcal{R}(x;
                           \left\{ \delta\bW_{i-1},0,\delta \bW_{i+1} \right\})
  .
\end{equation}
Moreover, it is clear that this reconstruction will preserve any equilibrium by
construction, because $\delta\bW_{i-1}$ and $\delta \bW_{i+1}$ both vanish
under this condition. Finally, we introduce the notation
\begin{align}
  \delta \bW_i(x) = \cW_i(x) - \bW_{eq,i}(x) = \cR(x; \{\delta \bW_{i-1}, 0, \delta \bW_{i+1}\})
\end{align}
and point out that $\delta \bW_i(x_i) = 0$ for any choice of $\delta \bW_{i-1}$ and
$\delta \bW_{i+1}$.

In astrophysically relevant simulations, some additional clipping of density
and pressure may be required. We propose clipping the density as follows  
\begin{align}
  \bar{\rho}_i(x_{i+1/2}) = \max(\check{\rho}, \min{\rho_{i}(x_{i+1/2}), \hat{\rho})}.
\end{align}
with $\check{\rho} = \min(\rho_i, \rho_{i+1})$ and $\hat{\rho} = \max(\rho_{i},
\rho_{i+1})$ and to proceed analogously for the pressure. The velocity can
remain unmodified. Note that if $(\phi(x_i), \phi(x_{i+1/2}), \phi(x_{i+1}))$ is
a monotone sequence, then $(\rho_{eq}(x_i), \rho_{eq}(x_{i+1/2}),
\rho_{eq}(x_{i+1}))$ and analogously the equilibrium pressure is also monotone.
Therefore, the clipping does not affect the well-balanced property of the
overall scheme.

\begin{remark}
  Unfortunately, it is not always possible to find a $\bW_{eq,i}$ which
  satisfies $\bW_{eq,i}(x_i) = \bW_i$. Whenever this happens that cell will
  default to the standard reconstruction and source term discretization. Note
  that if the initial conditions are the point values of an equilibrium, then
  clearly $\bW_{eq,i}$ exists. Therefore, this does not affect the well-balanced
  property.
\end{remark}

\subsubsection{Well-balanced source term discretization}
\label{sec:wb_src_cart}
Next, we detail the necessary modifications to the source term discretization.
The idea is to decompose the source term into an equilibrium and a perturbation
part followed by an appropriate integration to obtain the source term cell
average.
As a result, the equilibrium part can then readily be written in a flux
difference form, guaranteeing the exact balance at the equilibrium.

For the momentum source term, we obtain the following decomposition in cell
$I_{i}$
\begin{align}\label{eq:wb_src_00}
  s^{(\rho v)}(\bW_{i}(x)) = - \rho_{i}(x) \qpd[x]{} \phi(x)
                           = - \left(\rho_{eq}(x) + \delta \rho_{i}(x)\right)
                               \qpd[x]{} \phi(x)
                           = - \rho_{eq,i}(x) \qpd[x]{} \phi(x)
                             - \delta \rho_i(x) \qpd[x]{} \phi(x)
  .
\end{align}
Direct numerical integration of the above will not result in a well-balanced
scheme.
Instead, we use the fact that we have the following correspondence for the
equilibrium part
\begin{equation}
\label{eq:wb_src_01}
  \qpd[x]{}{f^{(\rho v)}(\bW_{eq,i}(x))} = - \rho_{eq,i}(x) \qpd[x]{} \phi(x)
\end{equation}
by construction.
Hence, the equilibrium part of the source term can be trivially integrated.
Subsequently, we apply the second-order accurate midpoint rule to the
perturbation part to obtain the following expression for the cell-averaged
momentum source term
\begin{align}\label{eq:wb_src_02}
  S_i^{(\rho v)} = \left. 
                     \frac{1}{\Delta x} f^{(\rho v)}(\bW_{eq,i}(x))
                   \; \right|_{x_{i-1/2}}^{x_{i+1/2}}
                 - \delta \rho_i(x_i) \qpd[x]{} \phi(x_i)
                 = \left. 
                     \frac{1}{\Delta x} f^{(\rho v)}(\bW_{eq,i}(x))
                   \; \right|_{x_{i-1/2}}^{x_{i+1/2}}
  .
\end{align}
In the last equality, we used the fact that the perturbation $\delta \rho_i(x)$
vanishes at the cell center, as described in \cref{sec:wb_rc}.

Along the exact same lines, one obtains the following well-balanced
second-order accurate discretization of the energy equation source term
\begin{equation}
\label{eq:wb_src_03}
  S_i^{(E)} = \left. 
                \frac{1}{\Delta x} f^{(E)}(\bW_{eq,i}(x))
              \; \right|_{x_{i-1/2}}^{x_{i+1/2}}
  .
\end{equation}

By combining the above, we obtain the following well-balanced second-order
discretization of the cell-averaged source term
\begin{equation}
\label{eq:wb_src_04}
  \bS_{i} = \begin{bmatrix}
              0              \\
              S_i^{(\rho v)} \\
              S_i^{(E)}
            \end{bmatrix}
  .
\end{equation}

\subsection{Local equilibrium determination}
\label{sec:local_equilibrium}
Finally, we detail the remaining component: the determination of the local
equilibrium profile \cref{eq:nm_1d_wb_0020}.
The latter fulfills the steady adiabatic flow \cref{eq:nm_1d_wb_0010}
\begin{equation}
  \label{eq:nm_1d_localeq_0020}
  \begin{aligned}
    s_{eq,i}(x) &= s_{i}                                              \\
    \rho_{eq,i}(x) ~ v_{eq,i}(x) & = \rho_{i} v_{i} = m_{i}           \\
    \frac{v^{2}_{eq,i}(x)}{2} + h(p_{eq,i}(x), s_{eq,i}(x)) + \phi(x)
    & = \frac{v^{2}_{i}}{2} + h_{i} + \phi(x_{i}) = \Be_{i}
    ,
  \end{aligned}
\end{equation}
where the equilibrium mass flux $m_{i}$, Bernoulli constant $\Be_{i}$ and
specific entropy $s_{i}$ are fixed by their values at cell center $x_{i}$.
This formal definition is highly implicit in nature and it is not obvious if
such an equilibrium is unique or exists at all. Therefore, we first discuss
existence and uniqueness.

We note that a form of the Bernoulli equation also appears in so-called moving
steady states of the shallow water equations.
There, similar issues arise and we refer to , e.g., Noelle et al.
\cite{NoelleEtAl2007} for a thorough discussion.

The above equations can be combined into a single equation for the equilibrium
density reconstruction $\rho_{eq,i}(x)$ as
\begin{equation}
  \label{eq:nm_1d_localeq_0030}
    \frac{m^{2}_{i}}{2 \rho^{2}_{eq,i}(x)}
  + h(p(\rho_{eq,i}(x), s_{i}), s_{i})
  + \phi(x)
  = \Be_{i}
  .
\end{equation}
If a suitable $\rho_{eq,i}(x)$ is found, the equilibrium velocity and pressure
are simply given by
\begin{equation}
  \label{eq:nm_1d_localeq_0040}
  v_{eq,i}(x) = \frac{m_{i}}{\rho_{eq,i}(x)}
  \quad \text{and} \quad
  p_{eq,i}(x) = p(\rho_{eq,i}(x), s_{i})
  .
\end{equation}

In order to simplify the notation, let us rewrite
\cref{eq:nm_1d_localeq_0030} as
\begin{equation}
  \label{eq:nm_1d_localeq_0050}
  \frac{m^{2}_{0}}{2 \rho^2} + h(p(\rho,s_{0}),s_{0}) + \phi = \Be_{0}
  ,
\end{equation}
where we have suppressed any references to the spatial dependence as well as
the cell under consideration, i.e. $m_{0}$, $\Be_{0}$ and $s_{0}$ denote
some generic mass flux, Bernoulli and entropy constants at a location $x_{0}$,
respectively.
The above may be separated into specific fluid and gravitational energy parts
by introducing
\begin{equation}
  \label{eq:nm_1d_localeq_0051}
  e(\rho) = \frac{m^{2}_{0}}{2 \rho^{2}} + h(\rho)
  ,
\end{equation}
which represents the fluid part (we also suppressed the dependence of the
enthalpy on the entropy).
Then the task to find an equilibrium density $\rho = \rho(x)$ at location $x$
is as follows: Given the equilibrium constants $m_{0}$, $\Be_{0}$,
$s_{0}$, and a gravitational potential value
$\phi = \phi(x)$, determine a suitable solution $\rho$ of
\begin{equation}
  \label{eq:nm_1d_localeq_0060}
  e(\rho) = \Be_{0} - \phi
  ,
\end{equation}
if it exists at all.
With help of the following fundamental thermodynamic relation for the specific
enthalpy
\begin{equation}
  \label{eq:nm_1d_localeq_0070}
  \ud h = T \ud s + \frac{\ud p}{\rho}
\end{equation}
we may express the derivative of $e(\rho)$ as
\begin{equation}
  \label{eq:nm_1d_localeq_0080}
  e'(\rho) = \frac{c^{2}(\rho)}{\rho} - \frac{m^{2}_{0}}{\rho^{3}}
  ,
\end{equation}
where the definition of the sound speed
$c^{2} = (\partial p/\partial \rho)_{s}$ was used.

Assuming that the \eos is convex \cite{MenikoffPlohr1989}, i.e.
$(\partial^2 p/\partial \rho^2)_{s} > 0$, we conclude from
\cref{eq:nm_1d_localeq_0080} that $e(\rho)$ has a unique minimum
$e_{\ast} = e(\rho_{\ast})$ at $\rho_{\ast}$ where
$v(\rho_{\ast}) = v_{\ast} = c_{\ast}$.
It corresponds to the density where the fluid velocity is equal to the sound
velocity for the given equilibrium constants.
This density/velocity is also called the critical density/velocity
\cite{LandauLifshitz1987}.

Therefore, we are now in position to answer the question of existence of an
equilibrium at a certain location $x$, i.e. a value of the gravitational
potential $\phi = \phi(x)$, based on the equilibrium constants as follows:
\begin{equation}
  \label{eq:nm_1d_localeq_0090}
  \Be_{0} - \phi ~ \begin{cases}
                          < e_{\ast} & \text{no  equilibrium,} \\
                          = e_{\ast} & \text{one equilibrium,} \\
                          > e_{\ast} & \text{two equilibria.}
                        \end{cases}
\end{equation}
In the case of two possible equilibria, there is one subsonic
($\rho > \rho_{\ast}$) and one supersonic ($\rho < \rho_{\ast}$) equilibrium.
The situation is sketched in Fig. \ref{fig:Be}

\begin{figure}[hbt]
  \begin{center}
    \includegraphics[width=0.5\textwidth]{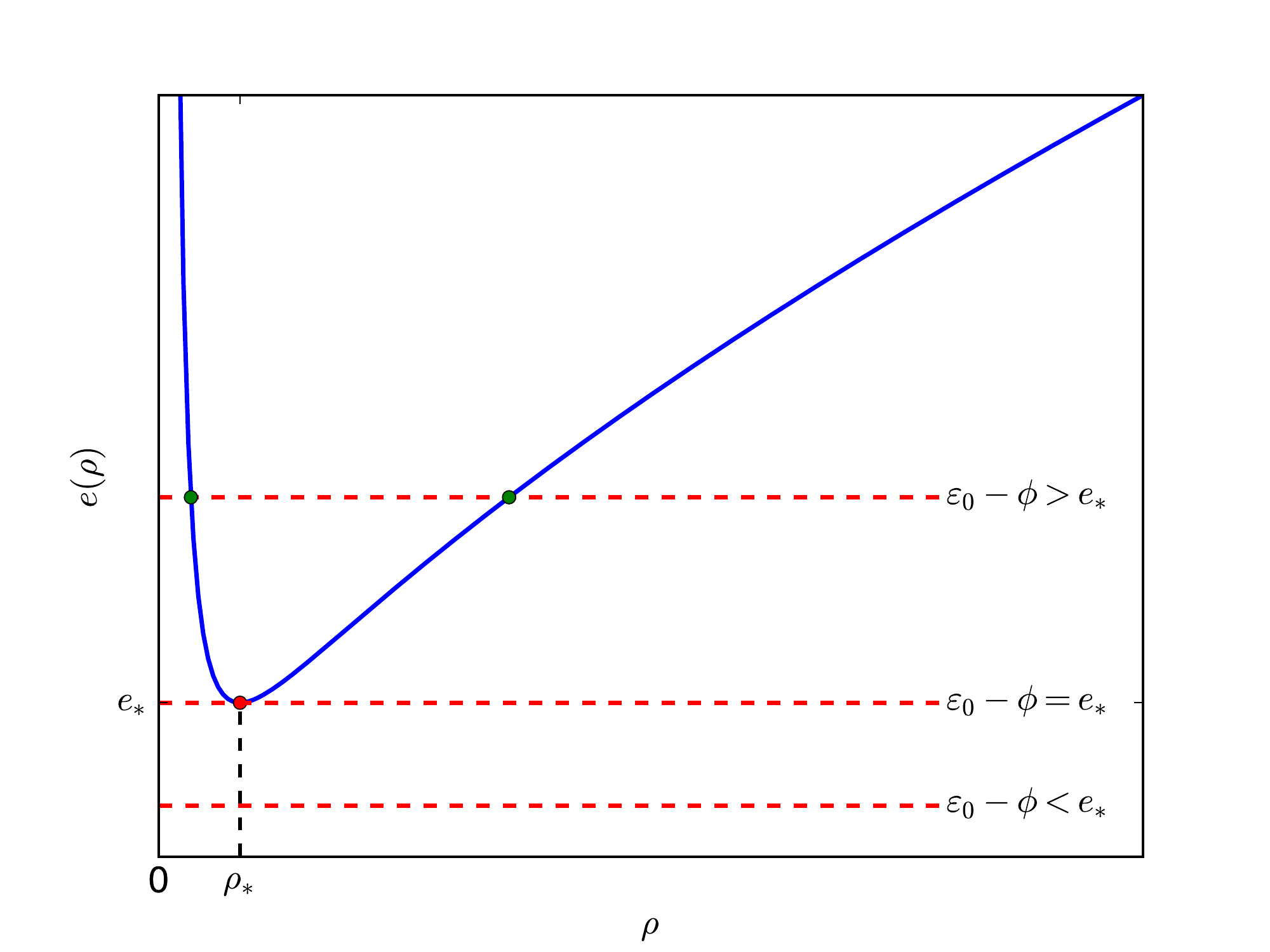}
  \end{center}
  \caption{Sketch of the function $e(\rho)$ (\cref{eq:nm_1d_localeq_0060}).
           Also shown are the possibilities of no, one (red dot) and two
           solutions (green dots).
           In the case of two solutions, the left (right) solution with
           $\rho < \rho_{\ast}$ ($\rho > \rho_{\ast}$) is supersonic
           (subsonic).}
  \label{fig:Be}
\end{figure}

So far, the discussion is applicable to any convex \eos.
For ease of presentation, we assume an ideal gas law equation of state
\begin{equation}
  p = (\gamma - 1) \rho e,
\end{equation}
where $\gamma$ is the ratio of specific heats.
In that case, the equilibrium determination is simplified because $e_{\ast}$
can be computed explicitly.
However, we stress that our well-balanced method is not restricted to this
particular \eos and numerical experiments with a general convex \eos are shown
in \cref{subsec:numex_stella}.

Thereby, we write the ideal gas law in the form
\begin{equation}
  \label{eq:nm_1d_localeq_0100}
  p = p(\rho, K) = K \rho^{\gamma}
  ,
\end{equation}
where $K$ is a function of entropy alone, i.e. $K = K(s)$.
Then we have $K_{0} = p_{0}/\rho_{0}^{\gamma}$ and consequently the critical
values can be computed explicitly as
\begin{equation}
  \label{eq:nm_1d_localeq_0101}
  \rho_{\ast} = \left(
                  \frac{m^{2}_{0}}{\gamma K_{0}}
                \right)^\frac{1}{\gamma + 1}, \;
  p_{\ast} = K_{0} \rho_{\ast}^{\gamma}, \;
  v^{2}_{\ast} = c^{2}_{\ast} = \frac{\gamma p_{\ast}}{\rho_{\ast}}, \;
  h_{\ast} = \frac{\gamma}{\gamma - 1}\frac{p_{\ast}}{\rho_{\ast}}
           = \frac{c^{2}_{\ast}}{\gamma - 1}
  \quad \text{and} \quad
  e_{\ast} = \frac{m^{2}_{0}}{2 \rho_{\ast}} + h_{\ast}
  .
\end{equation}
In case two equilibrium solutions exist, we chose the solution with
sub/super-sonic velocity if the equilibrium constants correspond to a
sub/super-sonic state.
In practice, the equilibrium is found by a hybrid Newton method combining
the quadratic convergence of the Newton method with a form of the robust
bisection method, see e.g. \cite{DennisSchnabel1996}.
The detailed algorithm is outlined in \cref{algo:equilibrium_ideal}.

\begin{algorithm}[hbt]
  Initial guess $\rho^{(0)} = \rho_{0}$\;
  \For{k = 0, 1, 2, ...}{
    \lIf{$\abs{e(\rho^{(k)}) + \phi - \Be_{0}} < tol ~ e(\rho_0)$}{Stop}
    $\rho^{(trial)} = \rho^{(k)}
                    - \frac{e(\rho^{(k)}) + \phi - \Be_{0}}
                           {e'(\rho^{(k)})}$\;

    \lIf{$v_{0} < c_{0}$ \text{\bf and } $\rho^{(trial)} < \rho_{\ast}$} {
        $\rho^{(trial)} = \frac{1}{2}(\rho_{\ast} + \rho^{(k)})$
    }
    \lIf{$v_{0} > c_{0}$ \text{\bf and } $\rho^{(trial)} > \rho_{\ast}$ }{
        $\rho^{(trial)} = \frac{1}{2}(\rho^{(k)} + \rho_{\ast})$
    }
    \lIf{$\rho^{(trial)} < 0$} {
      $\rho^{(trial)} = \frac{1}{2}\rho_{(k)}$
    }
    $\rho^{(k+1)} = \rho^{(trial)}$\;
  }
  \caption{Local equilibrium determination}
  \label{algo:equilibrium_ideal}
\end{algorithm}

The case of a general convex \eos is treated in the appendix, see
\cref{algo:equilibrium_general}. This concludes
the elaboration of the well-balanced scheme for one-dimensional steady adiabatic
flow and we summarize it in the following:
\begin{theorem}
\label{thm:nm_1d}
Consider the scheme \eqref{eq:nm_1d_fv_0010} with a consistent and Lipschitz
continuous numerical flux $\mathcal{F}$, the spatially first/second-order
reconstruction $\mathcal{W}$ \eqref{eq:wb_rc_0010}/\eqref{eq:wb_rc_0030} and
source term discretization \eqref{eq:wb_src_04}.

The scheme has the following properties:
\begin{enumerate}[(i)]
\item The scheme is consistent with \eqref{eq:balance_short} and it is
      formally first/second-order accurate in space (for smooth solutions).
\item The scheme is well-balanced and preserves a discrete steady adiabatic
      flow given by \eqref{eq:nm_1d_wb_0010} exactly.
\end{enumerate}
\end{theorem}
\begin{proof} %
(i) The consistency and formal order of accuracy of the scheme is
straightforward.

(ii) Let data $\bW_{i} = [\rho_{i}, v_{i}, p_{i}]^T$ in steady adiabatic flow
state \eqref{eq:nm_1d_wb_0010} be given.
Then both first- and second-order accurate reconstructions $\mathcal{W}$
\eqref{eq:wb_rc_0010}/\eqref{eq:wb_rc_0030} will yield the same equilibrium
fulfilling state for the left and right cell interface traces
\begin{equation*}
  \bW_{i+1/2-} = \bW_{i+1/2+} = \bW_{i+1/2}
  .
\end{equation*}
Plugging this into a consistent numerical flux gives
\begin{equation*}
  \bF_{i+1/2} = \mathcal{F}(\bW_{i+1/2},\bW_{i+1/2})
              = \bbf(\bW_{i+1/2})
  .
\end{equation*}
Similarly, we may evaluate the source term by \cref{eq:wb_src_04}
\begin{equation*}
  \bS_{i} = \frac{1}{\Delta x}
            \begin{bmatrix}
              0                                                           \\
              f^{(\rho v)}(\bW_{i+1/2}) - f^{(\rho v)}(\bW_{i-1/2}) \\
              f^{(E)}(\bW_{i+1/2}) - f^{(E)}(\bW_{i-1/2})
            \end{bmatrix}
  .
\end{equation*}
By combining both above expressions in \eqref{eq:nm_1d_fv_0010}, we immediately
obtain
\begin{equation*}
  \fd[t]{}\bU_i = \mathcal{L}(\bU_{i})
                = - \frac{1}{\Delta x}
                    \left( \bF_{i+1/2} - \bF_{i-1/2}\right)
                  + \bS_{i}
                = 0
  .
\end{equation*}
This shows the well-balanced property of the scheme.
\end{proof}

\subsection{Extension to cylindrical and spherical symmetry}
\label{subsec:nm_cylsph}
The Euler equations with gravity in cylindrical and spherical symmetry can be
written in the following compact form
\begin{align} \label{eq:euler_compact_sphe}
    \qpd[t]{}\bu + r^{-\alpha}\ \qpd[r]{}(r^\alpha\ \bbf) = \bbs_{geo} + \bbs_{gra} = \bbs
\end{align}
where the conserved variables $\bu$, the fluxes $\bbf$ and the gravity source
term $\bbs_{gra}$ are as in \cref{eq:euler_short}.
The radial coordinate is denoted by $r$ and $\alpha$ specifies whether the
symmetry is cylindrical ($\alpha = 1$) or spherical ($\alpha = 2$).
The geometric source term reads
\begin{equation}
\label{eq:nm_cylsph_0010}
  \bbs_{geo} = \begin{bmatrix}
                 0                       \\
                 \alpha r^{\alpha - 1} p \\
                 0
               \end{bmatrix}
  .
\end{equation}
In these geometries, steady adiabatic flow is governed by
\begin{align}
  \label{eq:eq_sphe}
  s = \const, \quad
  r^{\alpha} \rho v = \const
  \quad \text{and} \quad
  \frac{v^2}{2} + h + \phi = \const.
\end{align}

\subsubsection{Standard finite-volume discretization}
\label{sec:standard_fvm_sphe}
A semi-discrete finite volume method for \cref{eq:euler_compact_sphe} is given
by
\begin{align}
\label{eq:standard_fvm_sphe_0010}
  \fd[t]{}\bU_i =   \mathcal{L}(\bU)
                = - \frac{1}{\abs{V_{i}}}
                    \left(  A_{i+1/2} ~ \bF_{i+1/2}
                          - A_{i-1/2}\bF_{i-1/2}
                    \right)
                  + \bS_{geo,i}
                  + \bS_{gra,i}
  .
\end{align}
where $V_i = [r_{i-\half}, r_{i+\half}]$ denotes the $i$-th cell ranging over
the left/right cell interface $r_{i\pm1/2} = r_{i} \pm \Delta r/2$ with cell
center $r_{i}$ and cell size $\Delta r$.
Explicit expressions for the cell volume $\abs{V_{i}}$ and interface areas
$A_{i\pm1/2} = A(r_{i\pm 1/2})$ are given by
\begin{equation}
\label{eq:standard_fvm_sphe_0020}
  \abs{V_{i}} = 2 \pi \left( r_{i+1/2}^{2} - r_{i-1/2}^{2} \right)
  , \quad
  A(r) = 2 \pi r
\end{equation}
for cylindrical symmetry and
\begin{equation}
\label{eq:standard_fvm_sphe_0030}
  \abs{V_{i}} = \frac{4 \pi}{3} \left( r_{i+1/2}^{3} - r_{i-1/2}^{3} \right)
  , \quad
  A(r) = 4 \pi r^{2}
\end{equation}
for spherical symmetry.

For the numerical flux, reconstruction and gravity source term discretization,
the same standard components as in the Cartesian case \cref{sec:standard_fvm}
can be employed.
However, we note that specialized reconstruction procedures for curvilinear
coordinates have been designed in the literature (see \cite{Mignone2014}
and references therein).

The momentum component of the geometric source term $\bbs_{geo}$ can be
discretized as
\begin{equation}
\label{eq:standard_fvm_sphe_0040}
  S^{(\rho v)}_{geo,i} = \frac{1}{\abs{V_{i}}}
                         \left( A_{i+1/2} - A_{i-1/2}\right) p_{i}
  ,
\end{equation}
where $p_{i}$ is the pressure at cell center, or more precisely, computed
simply from the cell-averaged conserved variables $\bW_{i} = \bw(\bU_{i})$.
Note that this discretization has the desirable property that resting uniform
conditions ($\rho = \const$, $p = \const$ and $v = 0$) are exactly preserved.

In general, the just outlined standard finite volume scheme for
cylindrical/spherical symmetry has difficulties in resolving the steady
adiabatic equilibrium \cref{eq:eq_sphe}.
Next, we describe the necessary modifications enabling the scheme to exactly
preserve such steady states, thereby extending the approach of Cartesian
geometry from \cref{subsec:nm_1d_wb}.

\subsubsection{Well-balanced finite-volume discretization}
\label{subsubsec:nm_cylsph_wb}

We follow the structure in the section of well-balanced finite volume
discretization in Cartesian coordinates and will first describe the finite volume
method in terms of an abstract equilibrium reconstruction, i.e. for each cell
let $\bW_{eq,i} = (\rho_{eq,i}, v_{eq,i}, p_{eq,i})$ be a stationary solution of
the Euler equation in cylindrical and spherical symmetry satisfying
\cref{eq:eq_sphe} such that $\bW_{eq,i} = \bW_i$. In a second step we will
describe how to evaluate $\bW_{eq,i}$.

\paragraph{Reconstruction}
The well-balanced reconstruction \cref{sec:wb_rc}, based on equilibrium profiles which satisfy
\cref{eq:eq_sphe} can also be used in cylindrical and spherical coordinates.

\paragraph{Momentum source term}
The derivation of the momentum source term in cylindrical and spherical coordinates closely
follows \cref{sec:wb_src_cart}. However, in curvilinear coordinates we need to additionally consider
the geometric source term, i.e.\
\begin{align} \label{eq:nm_cyl_sphe_0100}
  \frac{1}{r^\alpha}\qpd[r]{} r^\alpha \bbf^{(\rho v)}(\bW_{eq,i})
  = \bbs_{geo}(\bW_{eq,i}) + \bbs_{gra}(\bW_{eq,i})
  = \alpha r^{\alpha-1}p_{eq,i} - \rho_{eq,i} \qpd[r]{}{\phi}.
\end{align}
In a first step, we split the pressure and density into an
equilibrium and perturbation term as follows
\begin{align}
  \bbs_{geo}^{(\rho v)}(r, \bW_i(r)) + s_{gra}^{(\rho v)}(r, \bW_i(r))
  = \alpha r^{\alpha -1} (p_{eq,i}(r) + \delta p_i(r)) - (\rho_{eq,i}(r) + \delta \rho_i(r))\ \qpd[r]{} \phi.
\end{align}

By the same steps as in the Cartesian case, the well-balanced cell-averaged
source term in cell $\scell_i$ is determined to be
\begin{align}\label{eq:wb_src_30}
  S_i^{(\rho v)}
  &= \frac{1}{\abs{\scell_i}}
    \left.
    A(r)\left(\rho v^2_{eq,i}(r) + p_{eq,i}(r)\right)
    \right|^{r_{i+\half}}_{r_{i-\half}}.
\end{align}

\paragraph{Energy source term discretization}
The derivation of the energy source term also closely follows its Cartesian
counter part. Following those steps, we derive that up to second order
\begin{align}
  s^{(E)}(r, \bW_i(r)) &= \frac{1}{r^\alpha} \qpd[r]{}(r^\alpha f^{(E)}(\bW_{eq,i})).
\end{align}
The well-balanced energy source term is obtained in a similar fashion as the
momentum source term and reads
\begin{align} \label{eq:wb_src_31}
  S_i^{(E)}
  = \frac{1}{\abs{\scell_i}}\left.A(r) f^{(E)}(\bW_{eq,i})\right|_{r_{i-\half}}^{r_{i+\half}}.
\end{align}

\subsubsection{Local equilibrium reconstruction in cylindrical/spherical
               symmetry}
\label{sec:local_equilibrium_sphe}
In analogy to the equilibrium reconstruction in Cartesian coordinates, we
formally define an equilibrium profile $\bW_{eq,i}(r)$ which satisfies the equations
of steady adiabatic flow in cylindrical and spherical coordinates.
The only difference is the mass flux which accounts for the geometry
\begin{align}
  r^\alpha \rho_{eq,i}(r) v_{eq,i}(r) = r^\alpha_i \rho_i v_i = m_i
  ,
\end{align}
where we again fix the constant by their values at cell center $r_{i}$.
Following precisely the steps outlined in \cref{sec:local_equilibrium},
we obtain one scalar equation for the equilibrium density, namely
\begin{equation} \label{eq:eq_sphe_01}
  \frac{m_i^{2}}{2 r^{2\alpha} \rho^2_{eq,i}(r)} + h(p_{eq,i}(r), s_{eq,i}(r)) + \phi(r) = \Be_{i}.
\end{equation}
Let us again simplify notation by rewriting \cref{eq:eq_sphe_01} as
\begin{align}\label{eq:eq_sphe_02}
  \frac{m_0^{2}}{2 r^{2\alpha} \rho^2} + h(p(\rho, s_0), s_0) + \phi(r) = \Be_{0}.
\end{align}
where $m_0$, $s_0$ and $\Be_{0}$ denote the values of $m$, $s$ and
$\Be$ at a reference point $r_0$.
The corresponding thermodynamic part and its derivative are
\begin{equation}
  \label{eq:eq_sphe_03}
\begin{aligned}
  e(\rho, r) = \frac{m_0^2}{2 r^{2\alpha} \rho^2} + h(\rho), \quad
  e^\prime(\rho, r) = -\frac{m_0^2}{2 r^{2\alpha} \rho^3} + \frac{c^2(\rho)}{\rho}.
\end{aligned}
\end{equation}
Both now depend on the spatial coordinate $r$. Let $\rho_\ast$ be the critical
density such that $e^\prime(\rho, r) = 0$. Clearly, both the $\rho_\ast$ and
$e_\ast$ are functions of $r$. Keeping this fact in mind, we determine the
number of solutions to \cref{eq:eq_sphe_01} as in the Cartesian case.

The same hybrid Newton's method proposed for the Cartesian case
(\cref{algo:equilibrium_ideal}), can be used to solve \cref{eq:eq_sphe_03}.
However, since $\rho_\ast$ now depends on $r$, it could happen that the
reference state is on the supersonic branch $\rho_0 < \rho_\ast(r_0)$, but
the initial guess given to the algorithm $\rho^{(0)} = \rho_0$ is on the
subsonic branch $\rho_0 > \rho_\ast(r)$, or vice versa.
Therefore, we choose the initial guess as follows
\begin{align}
  \rho^{(0)} = \frac{\rho_0}{\rho_\ast(r_0)} \rho_{\ast}(r)
  ,
\end{align}
which ensures that the initial guess is in the same sub-/super-sonic regime as
the reference state.

This concludes the elaboration of the well-balanced scheme for steady adiabatic
flow in cylindrical and spherical coordinates.
We summarize the schemes' properties in the following:
\begin{corollary}
\label{cor:nm_cylsph}
Consider the scheme \eqref{eq:standard_fvm_sphe_0010} with a consistent and
Lipschitz continuous numerical flux function $\mathcal{F}$, the spatially
first/second-order accurate reconstruction $\mathcal{W}$
\cref{eq:wb_rc_0010}/\cref{eq:wb_rc_0030} and source term discretization
\cref{eq:wb_src_30} and \cref{eq:wb_src_31}.

The scheme has the following properties:
\begin{enumerate}[(i)]
\item The scheme is consistent with \eqref{eq:euler_compact_sphe} and it is
      formally first/second-order in space (for smooth solutions).
\item The scheme is well-balanced and preserves a discrete steady adiabatic
      flow given by \eqref{eq:eq_sphe} exactly.
\end{enumerate}
\end{corollary}
\begin{proof}
The proof parallels mostly the one of \cref{thm:nm_1d} and is straightforward.
\end{proof}

\subsection{Extension to several space dimensions}
\label{subsec:nm_md}
We briefly outline the straightforward extension of the above schemes to
several space dimensions.
However, the extension will in general only be truly well-balanced if the
streamlines of the steady adiabatic flow of interest are aligned with a
computational grid axis.

For the sake of simplicity, we treat the two-dimensional Cartesian case
explicitly since the extension to other geometries and three dimensions is
analogous.
The two-dimensional Euler equations with gravity in Cartesian coordinates are
given by
\begin{equation}
  \label{eq:nm_md_0010}
  \qpd[t]{}\bu + \qpd[x]{} \bbf + \qpd[y]{} \bg = \bbs
\end{equation}
with
\begin{equation}
  \label{eq:nm_md_0020}
  \bu = \begin{bmatrix}
          \rho     \\
          \rho v_x \\
          \rho v_y \\
          E
        \end{bmatrix}
  , \quad
  \bbf = \begin{bmatrix}
           \rho v_x       \\
           \rho v_x^2 + p \\
           \rho v_y v_x   \\
           (E + p) v_x
         \end{bmatrix}
  , \quad
  \bg = \begin{bmatrix}
          \rho v_y       \\
          \rho v_x v_y   \\
          \rho v_y^2 + p \\
          (E + p) v_y
        \end{bmatrix}
  \quad \mathrm{and} \quad
  \bbs = \bbs^{x} + \bbs^{y}
       = - \begin{bmatrix}
             0        \\
             \rho     \\
             0        \\
             \rho v_x
           \end{bmatrix} \qpd[x]{} \phi
         - \begin{bmatrix}
             0        \\
             0        \\
             \rho     \\
             \rho v_y
           \end{bmatrix} \qpd[y]{}{\phi}
  ,
\end{equation}
where $\bu$ is the vector of conserved variables, $\bbf$ and $\bg$ the fluxes
in $x$- and $y$-direction, and $\bbs$ the gravitational source terms.
The primitive variables are given by $\bw=[\rho, v_{x}, v_{y}, p]^{T}$.

We consider a rectangular spatial domain
$\Omega = [x_{\min},x_{\max}] \times [y_{\min},y_{\max}]$ discretized
uniformly (for ease of presentation) by $N_{x}$ and $N_{y}$ cells or finite
volumes in $x$- and $y$-direction, respectively.
The cells are labeled by $I_{i,j} = I_{i} \times I_{j}
                                  =        [x_{i-1/2},x_{i+1/2}]
                                    \times [y_{j-1/2},y_{j+1/2}]$
and the constant cell sizes by $\Delta x = x_{i+1/2} - x_{i-1/2}$ and
$\Delta y = y_{j+1/2} - y_{j-1/2}$.
We denote the cell centers by $x_{i} = (x_{i-1/2} + x_{i+1/2})/2$ and
$y_{j} = (y_{j-1/2} + y_{j+1/2})/2$.

A semi-discrete finite volume scheme for the numerical approximation of
\eqref{eq:nm_md_0010} then takes the following form
\begin{equation}
  \label{eq:nm_md_0030}
  \frac{\mathrm{d}}{\mathrm{d} t} \bU_{i,j}
  = \cL(\bU)
  = - \frac{1}{\Delta x}
      \left( \bF_{i+1/2,j} - \bF_{i-1/2,j} \right)
    - \frac{1}{\Delta y}
      \left( \bG_{i,j+1/2} - \bG_{i,j-1/2} \right)
    + \bS_{i,j}
  ,
\end{equation}
where $\bU_{i,j}$ denotes the approximate cell averages of the conserved
variables,
\begin{equation}
  \label{eq:nm_md_0040}
  \bF_{i\pm1/2,j} = \mathcal{F}(\bW_{i\pm1/2-,j},\bW_{i\pm1/2+,j})
  \quad \text{and} \quad
  \bG_{i,j\pm1/2} = \mathcal{G}(\bW_{i,j\pm1/2-},\bW_{i,j\pm1/2+})
\end{equation}
the numerical fluxes through the respective cell face, and $\bS_{i,j}$ the cell
averages of the source term.
The $\bW_{i\pm1/2\mp,j}$ and $\bW_{i,j\pm1/2\mp}$ denote the traces of the
primitive variables at the center of the cell face, in the respective direction.

The equilibrium preserving reconstruction $\mathcal{W}$ of \cref{sec:wb_rc} is
trivially applied in $x$- and $y$-direction independently.
In $x$-direction, let
\begin{align}
  \bW^{x}_{eq,i,j}(x) = \left(\rho^x_{eq,i,j}(x), v^x_{x,eq,i,j}(x), 0, p^x_{eq,i,j}(x)\right)
\end{align} 
be the solution of
\begin{align}
  s = \const
  , \quad
  \rho v_x = \const
  , \quad
  \frac{v_x^{2}}{2} + h + \phi = \const,
\end{align}
as described in \cref{sec:local_equilibrium}. The well-balanced reconstruction
in $x$-direction is then
\begin{equation}
\label{eq:nm_md_0050}
  \bW_{i,j}(x,y_{j}) = \mathcal{W}^{x}\left(
                                        x,y_{j};
                                        \left\{\bW_{k,j}\right\}_{k \in S_{i}}
                                      \right)
                     = \bW^{x}_{eq,i,j}(x)
                     + \delta^{x}\bW_{i,j}(x,y_{j}).
\end{equation}
Note that since the transverse component of the velocity of $\bW^x_{eq,i,j}$ is
zero, the reconstruction of that component of the velocity is in fact
simply the standard reconstruction. The reconstruction in
$y$-direction is obtained by simply reversing the roles of $v_x$ and $v_y$.

The gravity source terms are discretized as
\begin{equation}
\label{eq:nm_md_0070}
  \bS_{i,j} = \bS^{x}_{i,j} + \bS^{y}_{i,j}
  , \quad
  \bS^{x}_{i,j} = \begin{bmatrix}
                    0                        \\
                    S^{x,(\rho v_{x})}_{i,j} \\
                    0                        \\
                    S^{x,(E)}_{i,j}
                  \end{bmatrix}
  \quad \text{and} \quad
  \bS^{y}_{i,j} = \begin{bmatrix}
                    0                        \\
                    0                        \\
                    S^{y,(\rho v_{y})}_{i,j} \\
                    S^{y,(E)}_{i,j}
                  \end{bmatrix}
  .
\end{equation}
The momentum source term is computed by \cref{eq:wb_src_02} on a
dimension-by-dimension basis as
\begin{equation}
\label{eq:nm_md_0080}
  S^{x,(\rho v_{x})}_{i,j} = \left. 
                               \frac{1}{\Delta x}
                               f^{(\rho v_{x})}
                               \left( \bW^{x}_{eq,i,j}(x,y_{j})\right)
                             \right|_{x_{i-1/2}}^{x_{i+1/2}}
  \quad \text{and} \quad
  S^{y,(\rho v_{y})}_{i,j} = \left. 
                               \frac{1}{\Delta y}
                               g^{(\rho v_{y})}
                               \left( \bW^{y}_{eq,i,j}(x_{i},y)\right)
                             \right|_{y_{j-1/2}}^{y_{j+1/2}}
  .
\end{equation}
Likewise, the energy source term is computed by \cref{eq:wb_src_03} on a
dimension-by-dimension basis as
\begin{equation}
\label{eq:nm_md_0090}
  S^{x,(E)}_{i,j} = \left. 
                      \frac{1}{\Delta x}
                      f^{(E)} \left( \bW^{x}_{eq,i,j}(x,y_{j})\right)
                      \right|_{x_{i-1/2}}^{x_{i+1/2}}
  \quad \text{and} \quad
  S^{y,(E)}_{i,j} = \left. 
                      \frac{1}{\Delta y}
                      g^{(E)} \left( \bW^{y}_{eq,i,j}(x_{i},y)\right)
                      \right|_{y_{j-1/2}}^{y_{j+1/2}}
  .
\end{equation}

This concludes the outline of the extension of the well-balanced schemes to
multiple dimensions.
It is clear that it is well-balanced if the streamlines of the considered
steady state are aligned with the $x$- or $y$-axis.

\section{Numerical Experiments}\label{sec:numex}
In this section, we test our well-balanced schemes on a series of numerical
experiments and compare their performance with a standard (unbalanced) base
scheme. Furthermore, we also compare to the hydrostatically well-balanced scheme
\cite{Kaeppeli2014}.
For the sake of conciseness, we only present the results for the practically
relevant second-order schemes.

To characterize a time scale on which a model reacts to perturbations of its
equilibrium, we define a characteristic crossing time
\begin{equation}
\label{eq:numex_0010}
  \sndxtime = \int_{x_{0}}^{x_{1}} \frac{\ud x}{\abs{v} + c}
  ,
\end{equation}
where $v$ is the fluid velocity and $c$ the sound speed.
It measures the time it takes a wave traveling at the fastest characteristic
speed to traverse the steady state of interest.

We quantify the accuracy of the schemes by computing the absolute errors
\begin{equation}
\label{eq:numex_0020}
  err_{1}(q) = \|q - q_{ref}\|_{1}
  ,
\end{equation}
where $\|.\|_{1}$ denotes the 1-norm, $q$ some quantity of interest (e.g.
pressure, velocity, ...) and $q_{ref}$ a reference solution.
The reference solution may be the steady state to be maintained discretely or
the result of an appropriately averaged high-resolution simulation.
While the comparison with a numerically obtained reference solution does not
provide a rigorous evidence of convergence, it nevertheless indicates a
meaningful measure of the errors.
We also introduce the following relative error measure
\begin{equation}
\label{eq:numex_0020}
  relerr_{1}(q) = \frac{\|q - q_{ref}\|_{1}}{\|q_{ref}\|_{1}}
  .
\end{equation}

We begin in \cref{numex:gaussian_bump} by several simple one-dimensional
numerical experiments with Cartesian geometry followed in
\cref{numex:bondi_sec} by several simple one-dimensional numerical
experiments with spherical geometry.
Both setups employ an ideal gas \eos.
The interested reader may readily reproduce these experiments in order to check
his or her implementation.
Finally, we demonstrate in \cref{subsec:numex_stella} the performance of the
scheme on a two-dimensional stellar accretion problem in cylindrical
coordinates involving a complex multi-physics \eos.

\subsection{One-dimensional steady adiabatic flow} \label{numex:gaussian_bump}
The first and simplest test we perform is a steady state solution of
\cref{eq:euler} in Cartesian coordinates on the domain $\domain = [0, 2]$
with and without a perturbation. The gravity is given by a linear gravitational
potential, i.e. $\phi(x) = g x$, with $g = 1$. The ratio of specific heats is
$\adidx = 5/3$. We enforce boundary conditions by keeping the values in the
ghost-cells constant and equal to the initial conditions. All numerical solvers
in this subsection use the HLLC numerical flux and the monotonized centered
limiter. The tolerance in the root finding procedure
\cref{algo:equilibrium_ideal} is $tol = 10^{-13}$. The CFL number is $c_{CFL} = 0.45$.

The initial conditions are
\begin{align} \label{eq:numex_ad_0000}
  \left(\rho^0, v^0, p^0\right)(x)
 = \left(\rho_{eq}(x), v_{eq}(x), p_{eq}(x) + A \exp(-(x - \bar{x})^2/\sigma^2)\right)
\end{align}
where $\sigma = 0.1$, $\bar{x}$ is the center of the bump and
$(\rho_{eq}, v_{eq}, p_{eq})$ is the equilibrium defined by the
points values 
\begin{align}
  (\rho_0, v_{0}, p_{0}) = (1, -M c_{s,0}, 1)
\end{align}
at $x_0 = 0$. Here $c_{s,0} = \gamma^{1/2}$ denotes the speed of sound at $x_0$.
The discrete initial conditions are obtained by applying the midpoint rule to
\cref{eq:numex_ad_0000} which results in
\begin{align}
  \left(\rho_i^0, v_i^0, p_i^0\right) = \left(\rho^0(x_i), v^0(x_i), p^0(x_i)\right).
\end{align}

We perform the experiment for a hydrostatic ($M = 0$), a subsonic ($M = 0.01$) and
a supersonic ($M = 2.5$) equilibrium and different sizes of the perturbation
(specified later) will be investigated. The location of the perturbation is
\begin{align}
  \bar{x} = \begin{cases}
    1.0, & \text{for } M = 0 \\
    1.1, & \text{for } M = 0.01 \\
    1.5, & \text{for } M = 2.5.
  \end{cases}
\end{align}

All convergence studies in this subsection are run with the unbalanced,
hydrostatically well-balanced and adiabatically well-balanced schemes on $N = 32, 64,
128, \dots, \num{2048}$ cells. A reference solution is computed by the adiabatically
well-balanced method on $N = 8192$ cells.

\subsubsection{Well-balanced property}
In this first experiment, we set the amplitude of the perturbation to zero and
check that the proposed scheme is well-balanced for all three values of $M$. The
simulation is run until $\tend = 4$ ($M < 1$) or $\tend = 1$ ($M > 1$)
which corresponds to roughly two characteristic crossing times $\sndxtime$.

The results are shown in \cref{tab:gaussian_bump_wb}. For all values of $M$ the
adiabatically well-balanced method preserves the discrete stationary state
down to machine precision.
As expected, the hydrostatically well-balanced method only preserves the
hydrostatic case where $M = 0$.
Finally, the unbalanced scheme produces large errors and is unable to maintain
the steady state accurately.

\begin{figure}[htbp]
  \begin{center}
    \includegraphics[width=0.32\columnwidth]{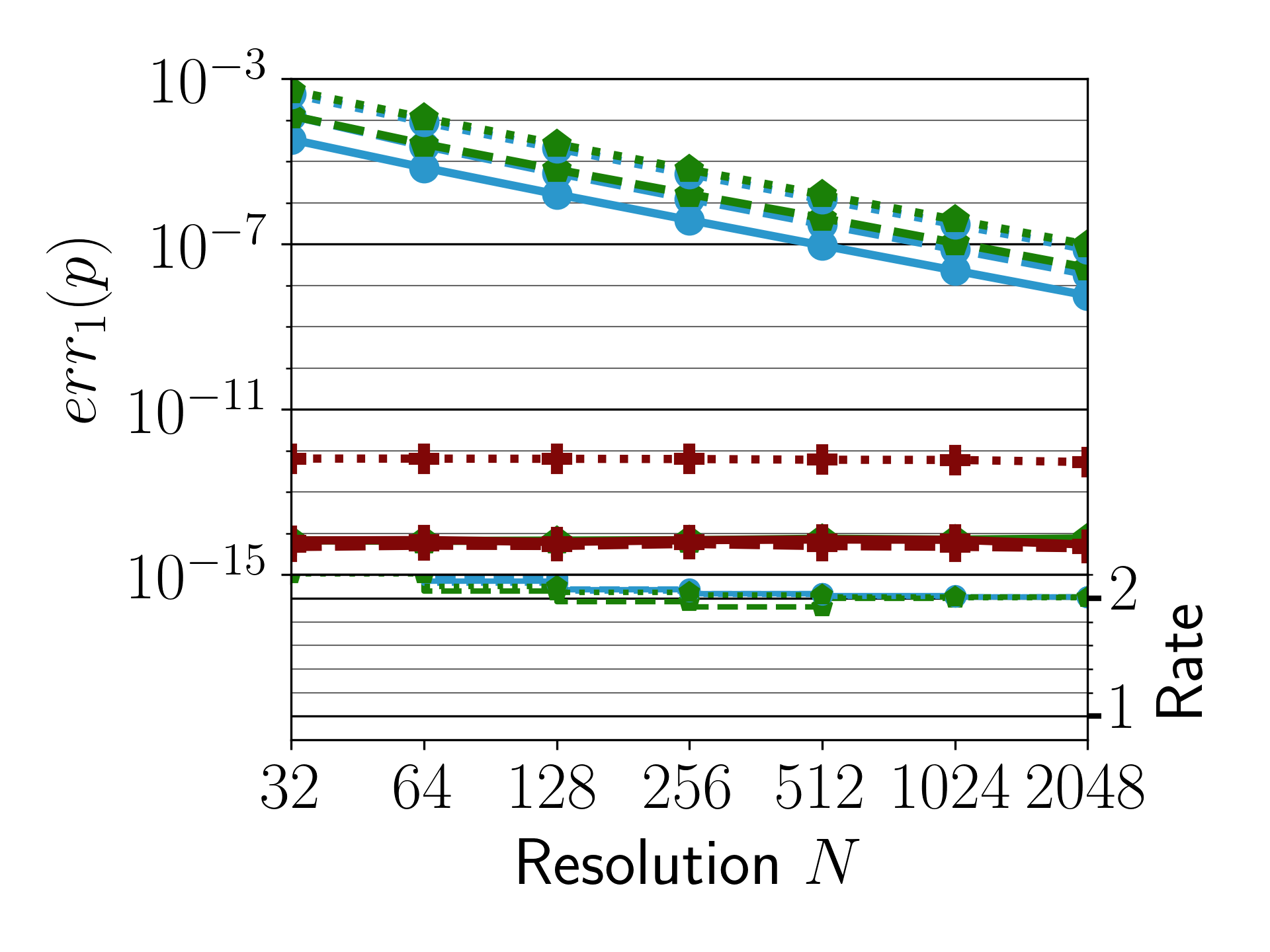}
    \includegraphics[width=0.32\columnwidth]{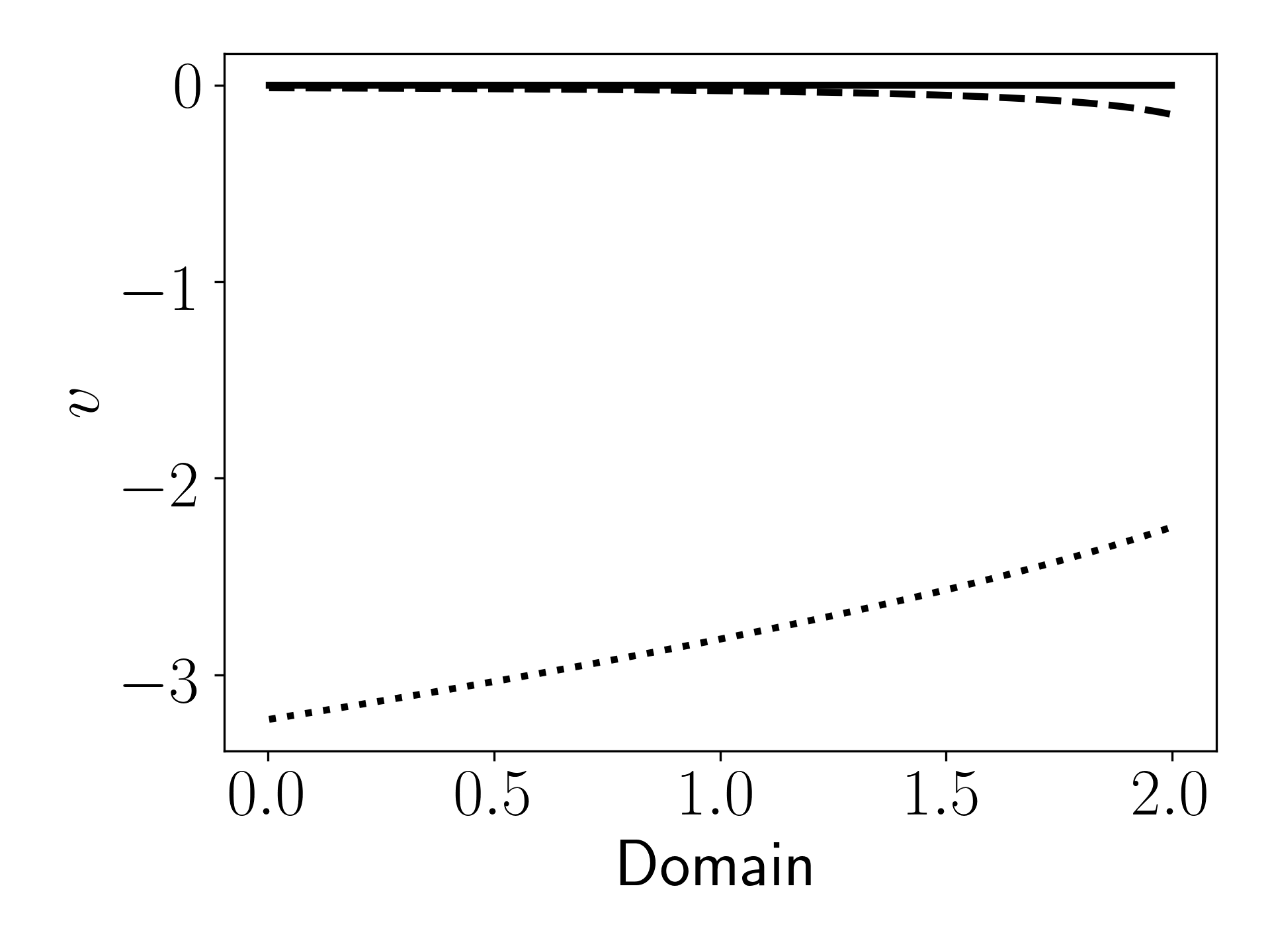}
    \includegraphics[width=0.32\columnwidth]{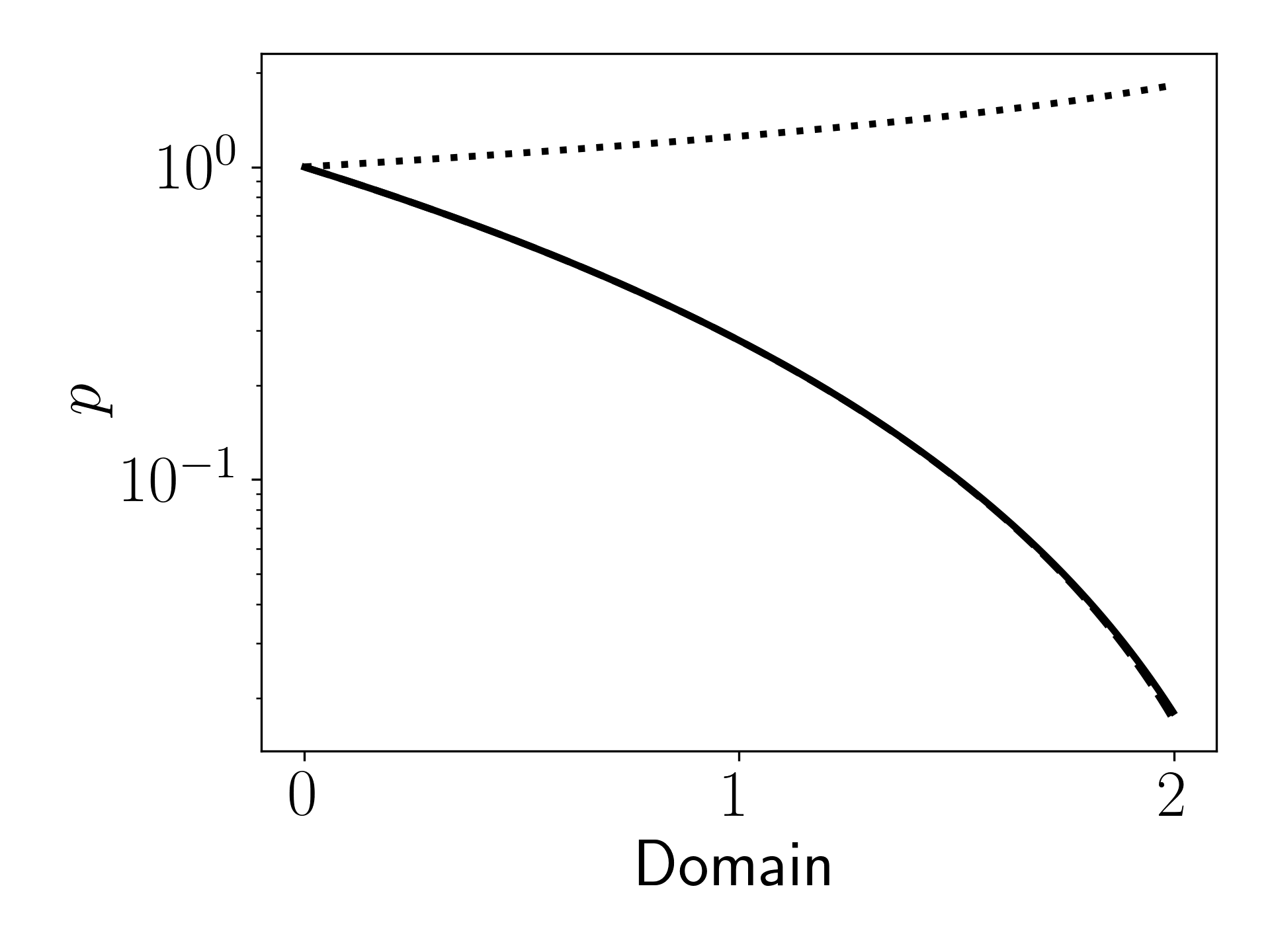}
  \end{center}
  \vspace{-0.2in}
  \caption{The left most column shows the well-balanced property of the methods
    on the problem described in \cref{numex:bondi} (with $A = 0$). The upper part of
    the subplot shows the $L^1$-error of the pressure $p$. The
    lower subplot shows the convergence rate between two consecutive levels of
    refinement. The unbalanced, hydrostatically well-balanced and adiabatically
    well-balanced scheme is shown in blue, green and red, respectively. The
    solid black line is the reference solution. The middle column shows the velocity
    and the right column shows the pressure, both at $\tend = 0$. Solid, dash and
    dotted lines represent $M = 0$, $M = 0.01$ and $M = 2.5$ cases, respectively.
    Note that the pressure for the hydrostatic and subsonic cases nearly coincide.}
  \label{fig:bondi_wb_a}
\end{figure}

\subsubsection{Smooth wave propagation}
We now compare the ability of the schemes to propagate small perturbations on
top of the equilibrium.
The size of the perturbation,
$A = 10^{-6}$, is chosen such that the wave remains smooth for the duration of
the numerical experiment.

The results at $\tend = 0.45$ ($M < 1$) or $\tend = 0.25$ ($M > 1$) are shown in
\cref{fig:gaussian_bump_small_a,fig:gaussian_bump_small_b,fig:gaussian_bump_small_c} and \cref{tab:gaussian_bump_small} is the
corresponding convergence table. All three solvers attain their formal
second-order accuracy. The adiabatically well-balanced method is the only
scheme capable of evolving the pressure perturbation accurately for all Mach
numbers and for $M > 0$ the error is smaller by a factor of at least $100$
compared to the hydrostatically well-balanced method. The hydrostatically
well-balanced method coincides with the adiabatically well-balanced method for
$M = 0$. At $M = 0.01$ the hydrostatically well-balanced method is only slightly
better than the unbalanced solver. For $M = 2.5$ the hydrostatically
well-balanced method has lost its advantage over the unbalanced solver. We
highlight that the errors of the adiabatically well-balanced scheme at the
lowest resolutions $N=32, 64$ are comparable to the errors of the unbalanced
scheme at the highest resolution $N=2048$ for all velocities.

\begin{figure}[htbp]
  \begin{center}
    \includegraphics[width=0.32\columnwidth]{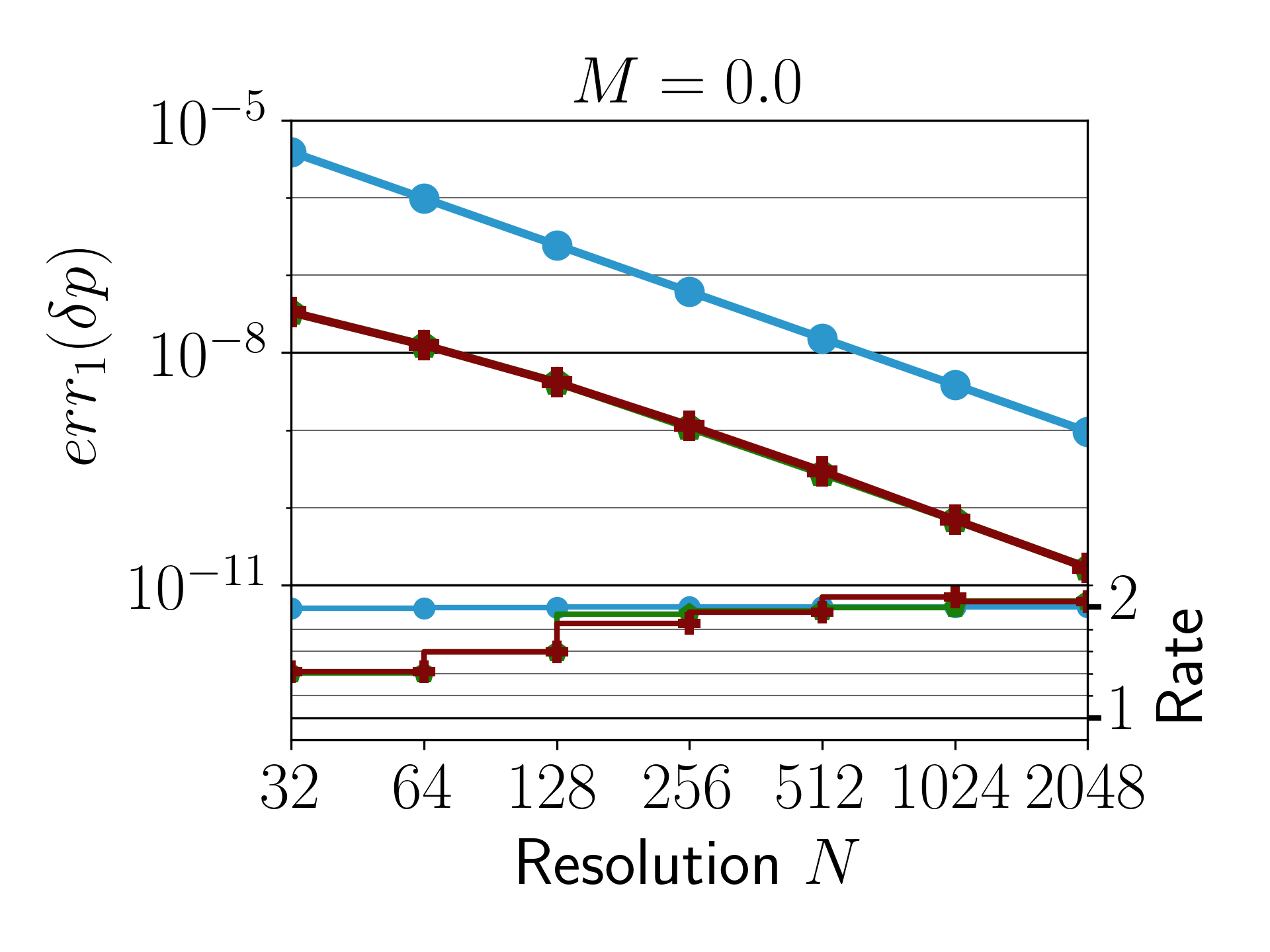}
    \includegraphics[width=0.32\columnwidth]{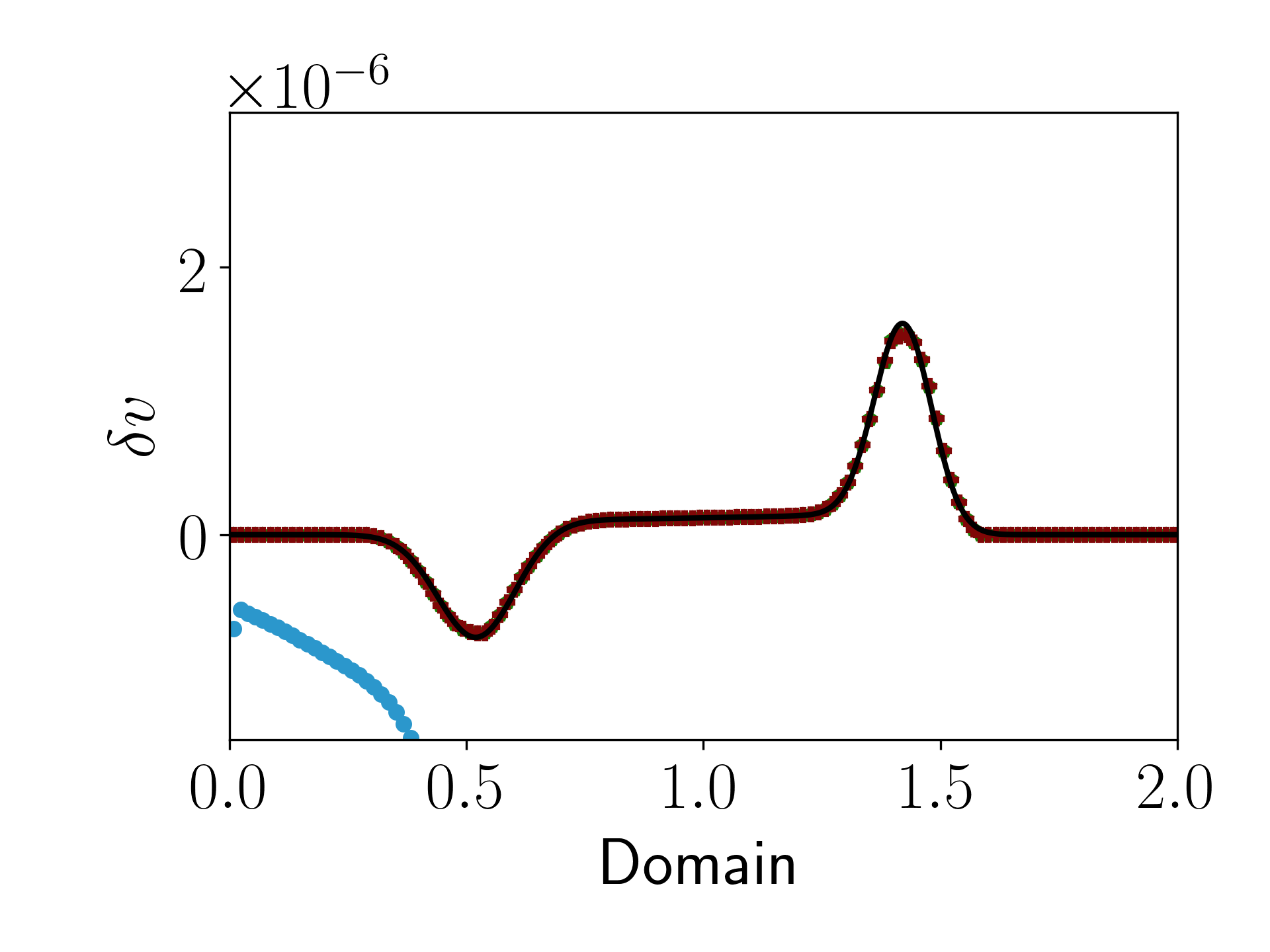}
    \includegraphics[width=0.32\columnwidth]{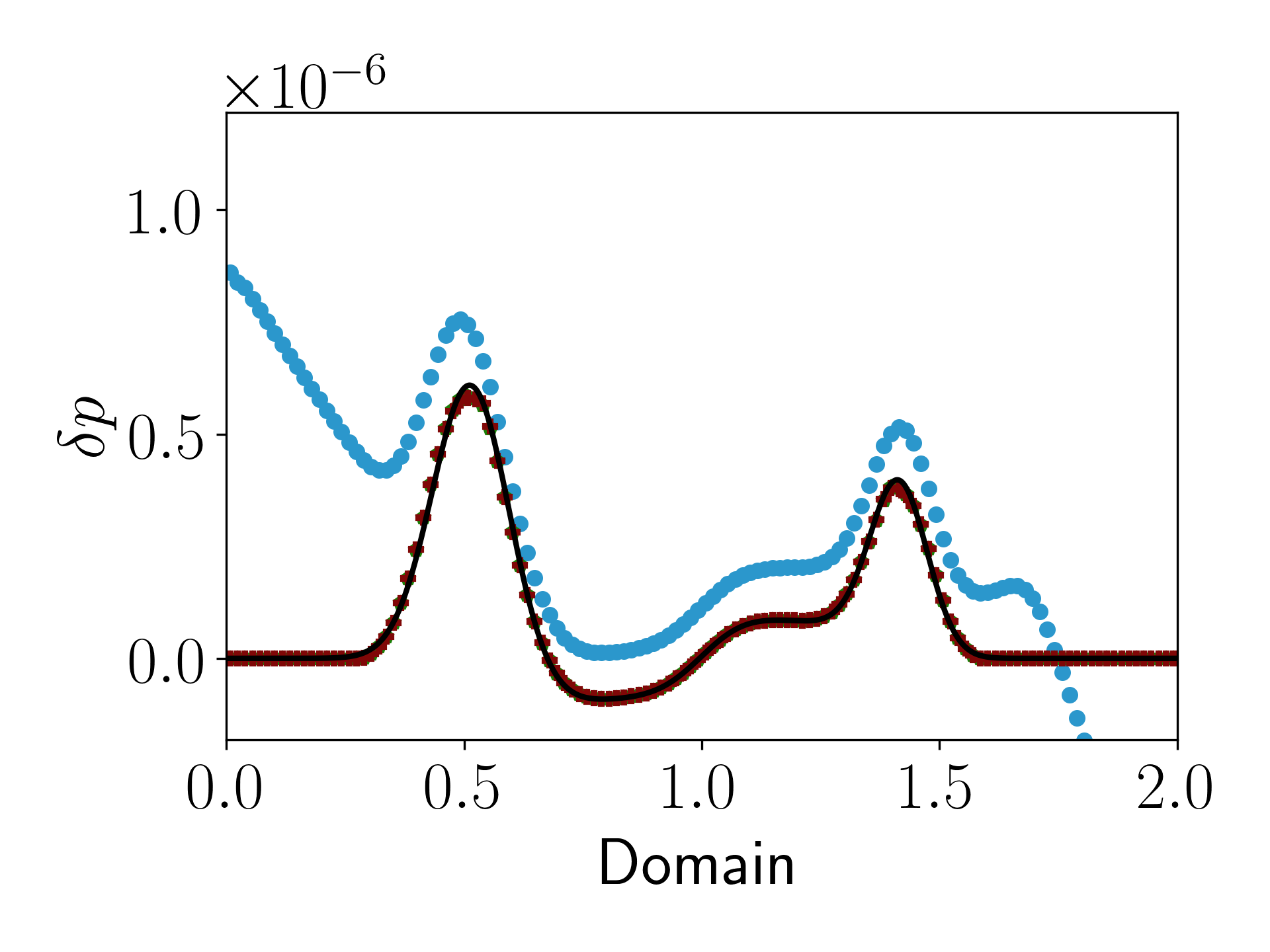}
  \end{center}
  \vspace{-0.2in}
  \caption{The left most column shows the convergence of the methods on the problem
    described in \cref{numex:gaussian_bump} with $A = 10^{-6}$. The upper part
    of the subplot shows the $L^1$-error of the pressure perturbation $\delta
    p$. The lower subplot shows the convergence rate between two consecutive
    levels of refinement. The middle column shows the velocity perturbation
    $\delta v$ and the right column shows the pressure perturbation $\delta p$. The
    Mach number at the reference point is $M = 0$. The scatter plots show the
    approximation with $N = 128$ cells at the final time described in the text. The
    unbalanced, hydrostatically well-balanced and adiabatically well-balanced scheme
    is shown in blue, green and red, respectively. The solid black line is the
    reference solution.}
  \label{fig:gaussian_bump_small_a}
\end{figure}

\begin{figure}[htbp]
  \begin{center}
    \includegraphics[width=0.32\columnwidth]{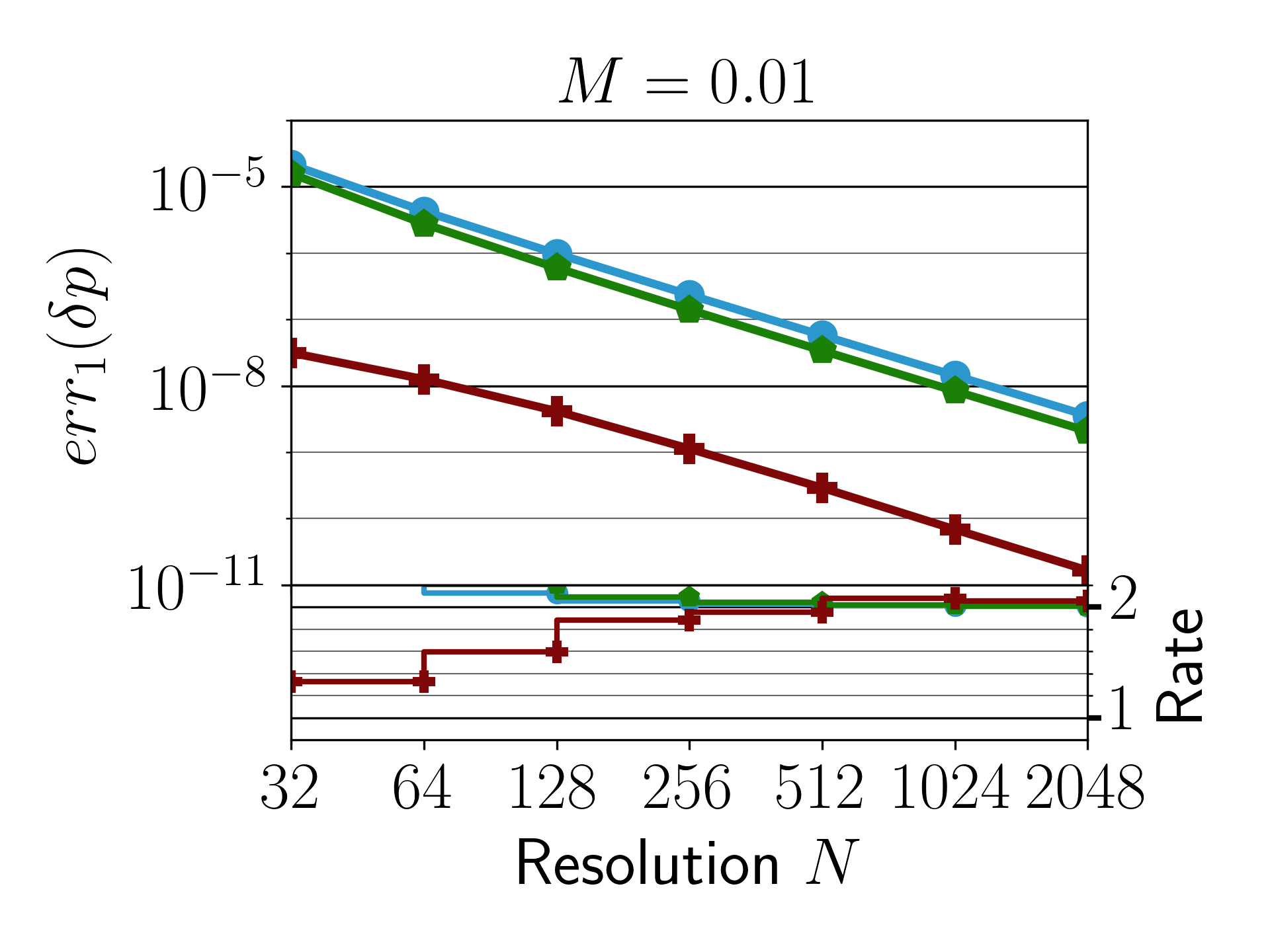}
    \includegraphics[width=0.32\columnwidth]{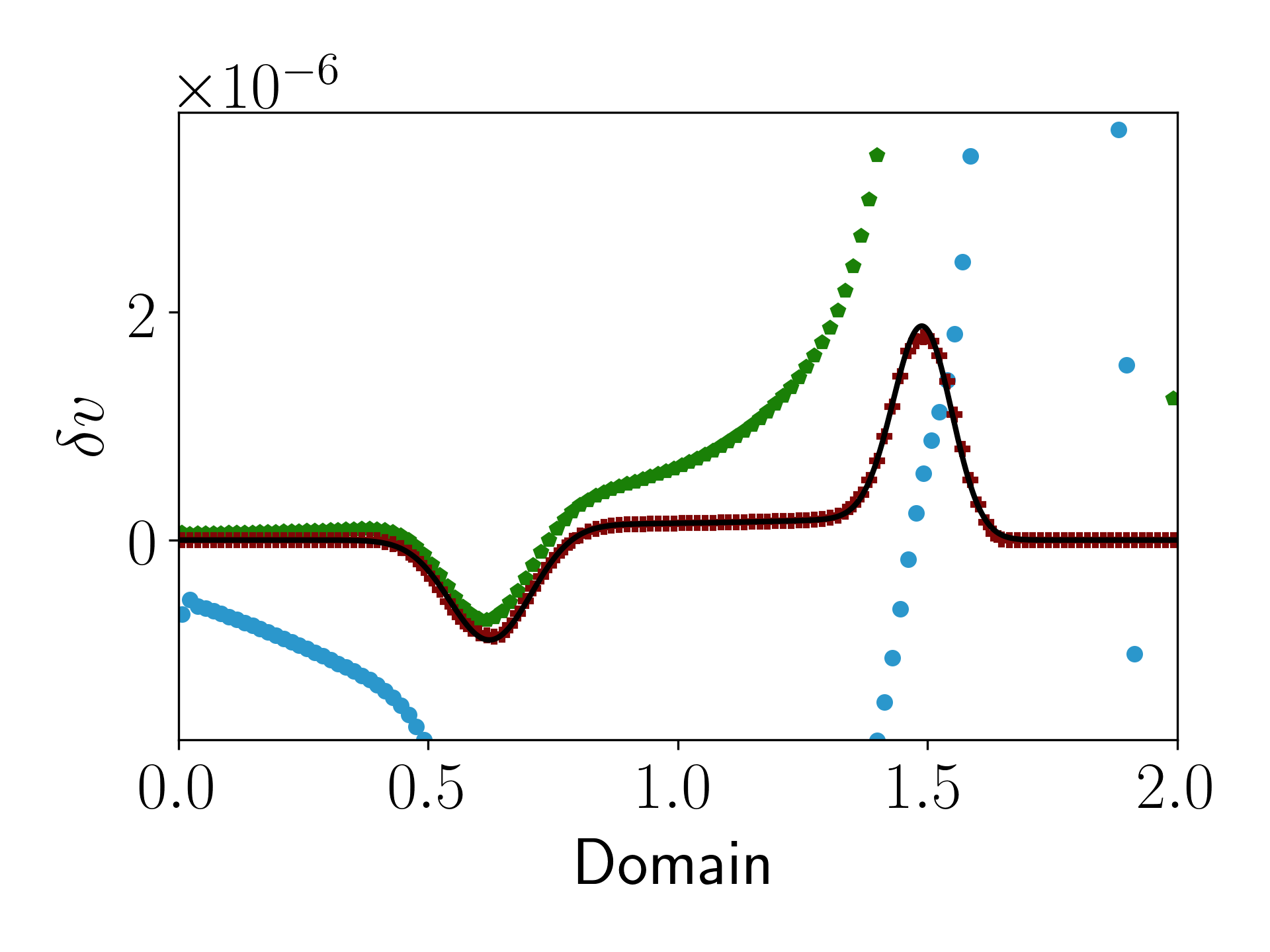}
    \includegraphics[width=0.32\columnwidth]{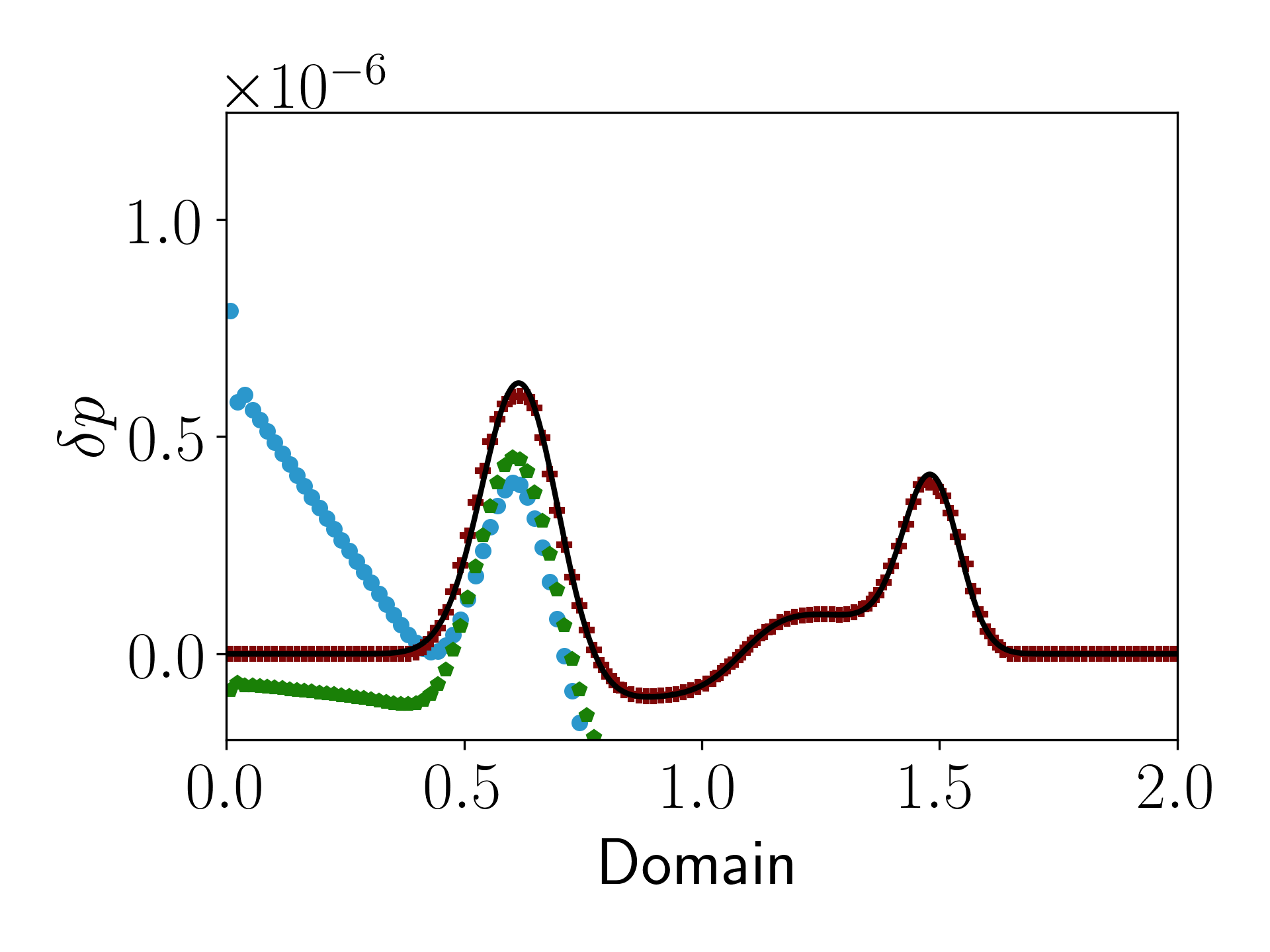}
  \end{center}
  \vspace{-0.2in}
  \caption{This figure shows the results for \cref{numex:gaussian_bump} with
    $A = 10^{-6}$ and $M = 0.01$. Please refer to the caption of
    \cref{fig:gaussian_bump_small_a} for further details.}
  \label{fig:gaussian_bump_small_b}
\end{figure}

\begin{figure}[htbp]
  \begin{center}
    \includegraphics[width=0.32\columnwidth]{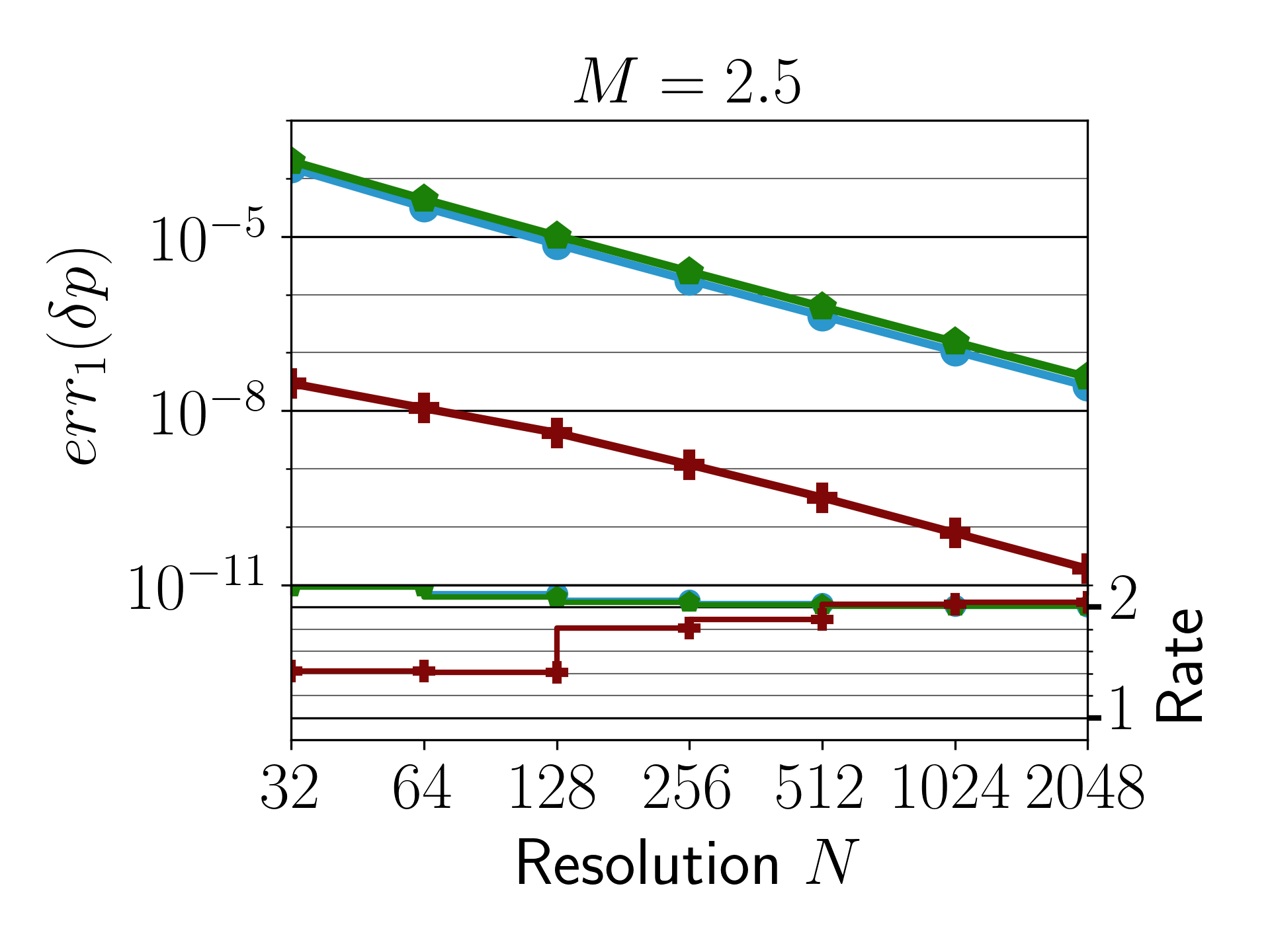}
    \includegraphics[width=0.32\columnwidth]{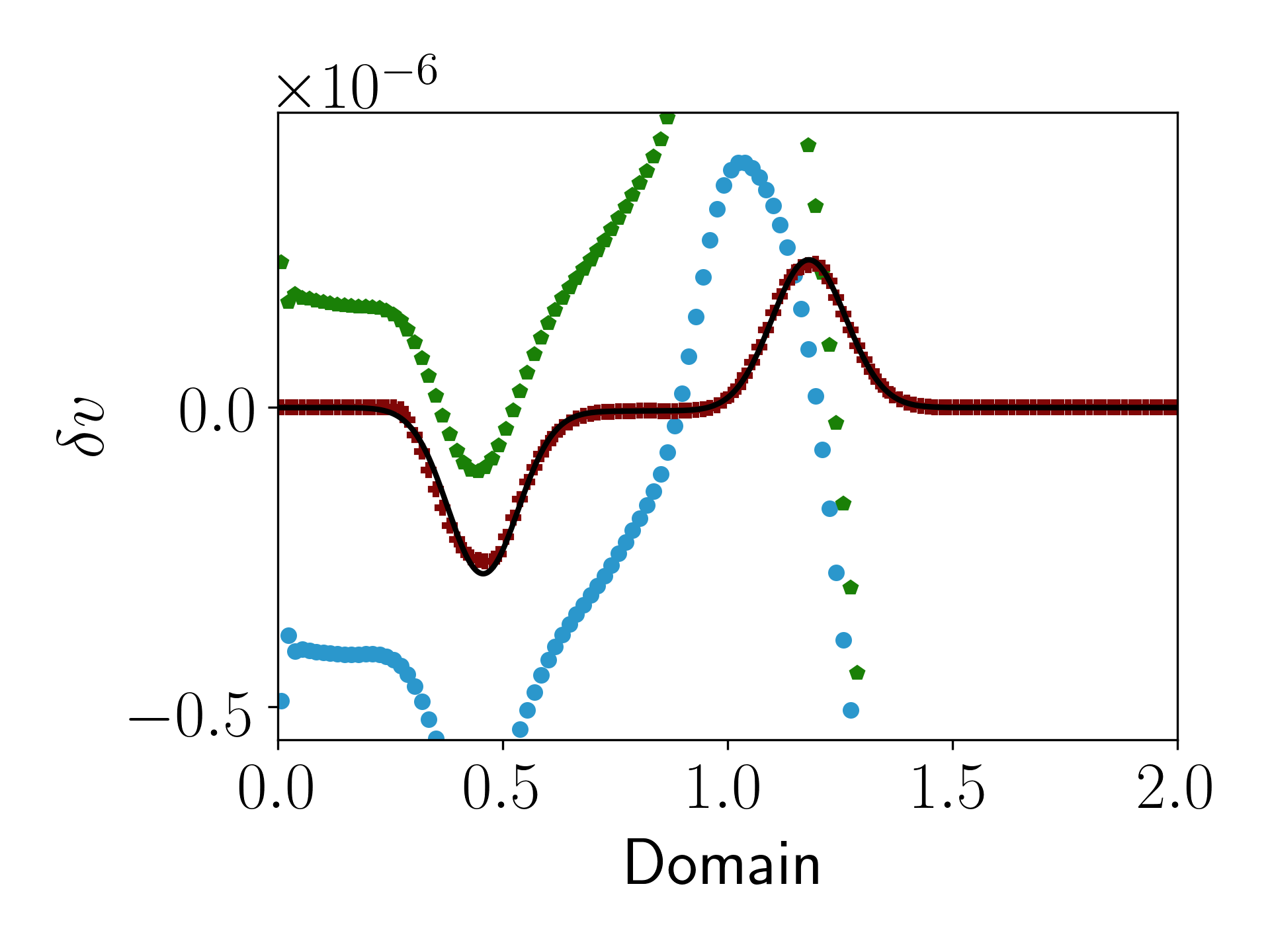}
    \includegraphics[width=0.32\columnwidth]{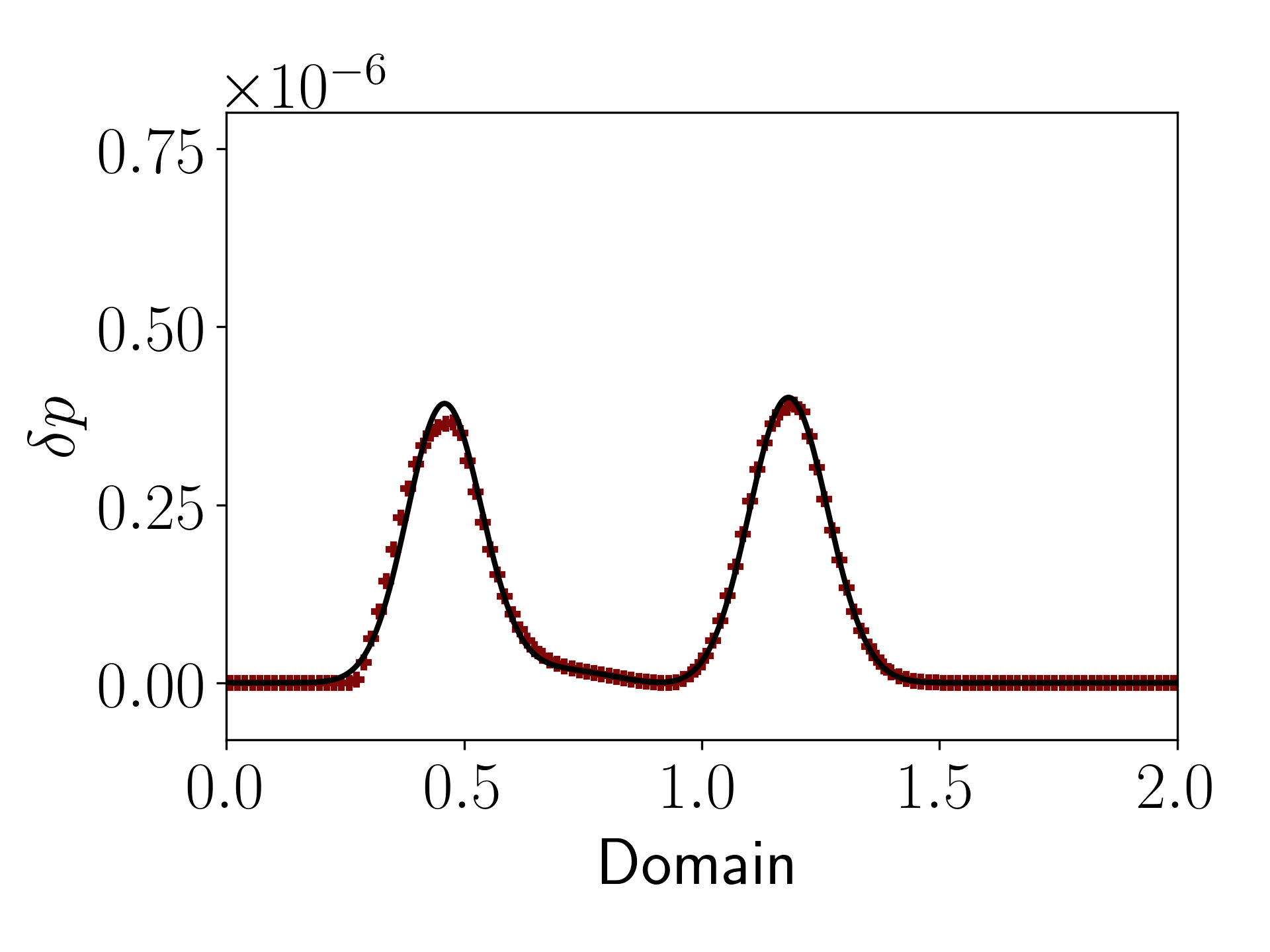}
  \end{center}
  \vspace{-0.2in}
  \caption{This figure shows the results for \cref{numex:gaussian_bump} with
    $A = 10^{-6}$ and $M = 2.5$. Please refer to the caption of
    \cref{fig:gaussian_bump_small_a} for further details.}
  \label{fig:gaussian_bump_small_c}
\end{figure}

\subsubsection{Discontinuous wave propagation}
In this third variant of the numerical experiment, we choose the magnitude of the
pressure perturbation such that due to the non-linearity of the Euler equations
the solution becomes discontinuous before the end of the simulation, which is
chosen to be $\tend = 0.45$ ($M < 1$) or $\tend = 0.25$ ($M > 1$).

The results for $M = 0.01$ are summarized in \cref{fig:gaussian_bump_shock} and
in the corresponding convergence table, \cref{tab:gaussian_bump_shock}.
Numerically, we observe that all three methods are able to propagate the shock
waves.
In fact, they are virtually indistinguishable.
Therefore, we observe that the well-balancing does not impact on the robustness
of the base high resolution shock-capturing finite volume scheme.
The observed convergence rate of approximately one
is expected for solutions with an isolated discontinuity.

\begin{figure}[htbp]
  \begin{center}
    \includegraphics[width=0.32\columnwidth]{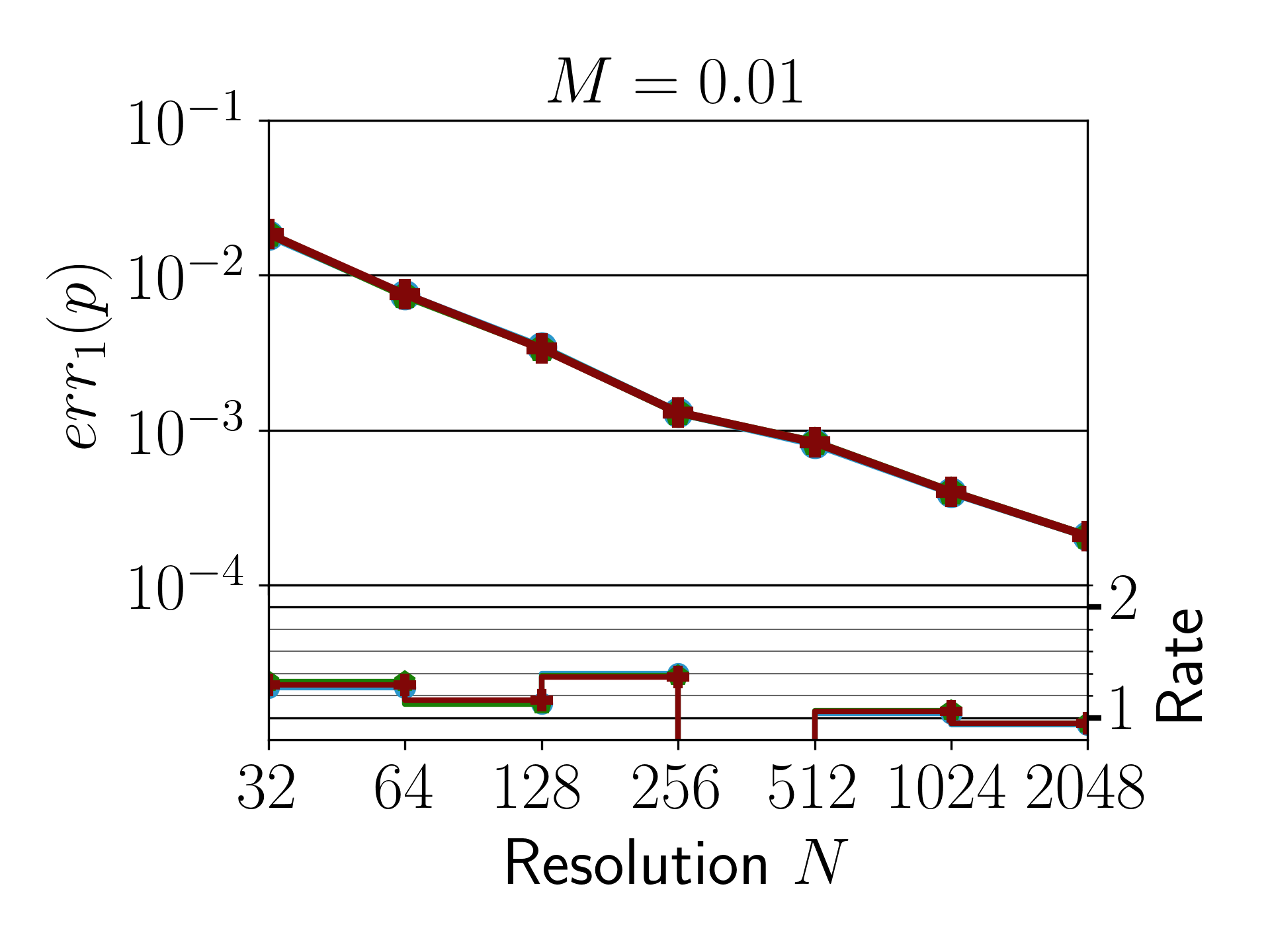}
    \includegraphics[width=0.32\columnwidth]{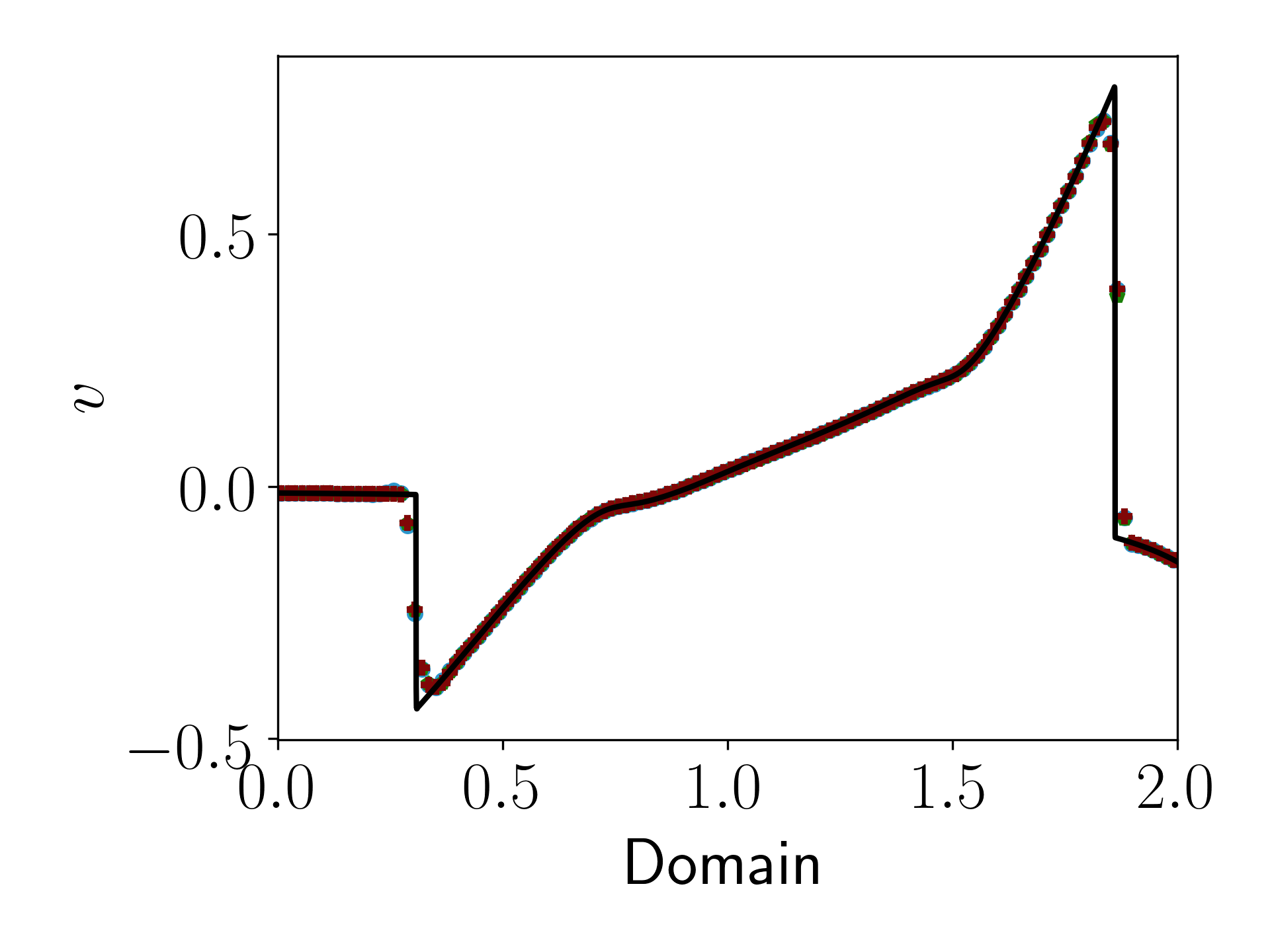}
    \includegraphics[width=0.32\columnwidth]{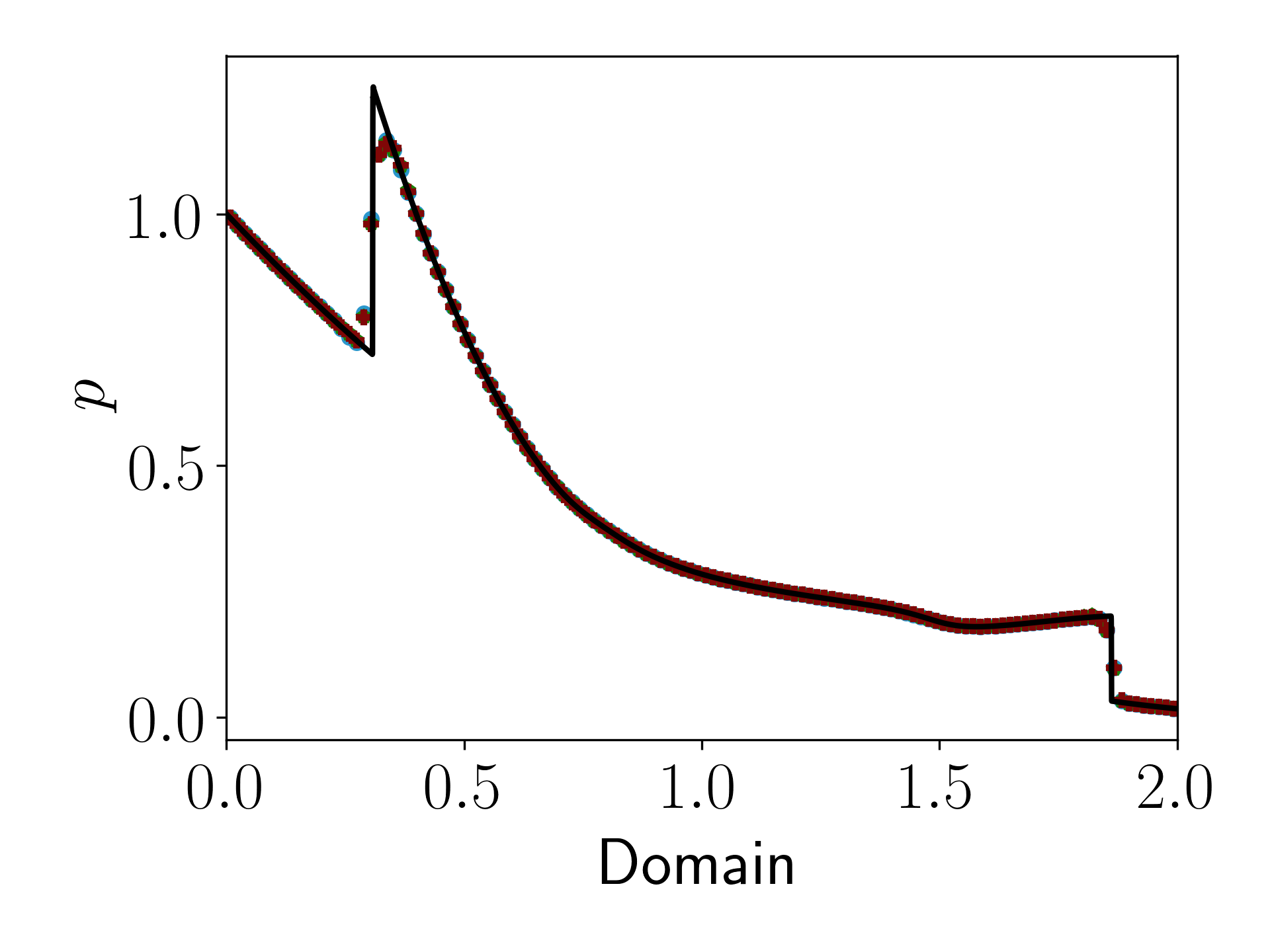}
  \end{center}
  \vspace{-0.2in}
  \caption{The left most column shows the convergence of the methods on the problem
    described in \cref{numex:gaussian_bump} with $A = 1$. The upper part
    of the subplot shows the $L^1$-error of the pressure $p$. The lower subplot
    shows the convergence rate between two consecutive levels of refinement. The
    middle column shows the velocity $v$ and the right column shows the pressure
    $p$. The Mach number at the reference point is $M =
    \num{0.01}$. The scatter plots show the approximation with $N = 128$ cells at
    the final time described in the text. The unbalanced, hydrostatically
    well-balanced and adiabatically well-balanced scheme is shown in blue, green and red,
    respectively. The solid black line is the reference solution.}
  \label{fig:gaussian_bump_shock}
\end{figure}

\clearpage
\subsection{One-dimensional, spherically symmetric experiments} \label{numex:bondi_sec}
The next two experiments are similar to the previous one, in the sense that we
consider a stationary state both with and without a perturbation. However, this
stationary state models the spherically symmetric, steady state accretion of gas
onto a star known as Bondi accretion flows.

In a first numerical experiment the equilibrium solution is assumed to be
continuous and either purely sub- or supersonic. In a second experiment we will
consider an equilibrium in which the sub- and supersonic branches are joined by a
stationary shock.

In both experiments the calculations are performed in spherical coordinates. The
domain is $\domain = [R_0, R_1]$, $R_0 = 0.2$, $R_1 = 1.8$ and the gravitational
potential is $\phi(r) = -Gm / r$ with $G = m = 1$. The adiabatic index is
$\gamma = 4/3$. The ghost-cells are kept
constant and equal to the initial conditions throughout the entire simulation.
All numerical solvers in this subsection use the HLLC numerical flux and, unless
stated explicitly otherwise, the monotonized centered limiter. The tolerance in
\cref{algo:equilibrium_ideal} is $tol=10^{-13}$. The CFL number is $c_{CFL} =
0.45$.

\subsubsection{Smooth equilibrium}\label{numex:bondi}
The initial conditions are 
\begin{align}
  \left(\rho^0, v^0, p^0\right)(r)
  = \left(\rho_{eq}(r), v_{eq}(r), p_{eq}(r) + A \exp(-(r - \bar{r})^2/\sigma^2)\right)
\end{align}
with $\sigma = 0.08$ and $\bar{r} = 0.4 R_0 + 0.6 R_1$. The equilibrium is
defined by the values of the density, velocity and speed of sound at the
reference point $r_0 = 1$:
\begin{align}
  \rho_0 = 1, \quad c_{s,0}^2 = \frac{1}{2}, \quad v_0 = -M c_{s,0}.
\end{align}
This ensures that the critical point is located at $r_0 = 1$ which is the
center of the domain. Therefore, we are sure that with $M = 0.9$ the background
of the initial conditions corresponds to the purely subsonic solution branch of
\cref{eq:eq_sphe} and for $M = 2.0$ the background is purely supersonic. The
parameter $A$ which controls the size of the perturbation will be specified
later.

The initial conditions are computed by first extrapolating the equilibrium from
the reference point $r_0$ to the cell-center of the cell just below $r_0$ and then
iteratively downwards from one cell to the next. The analogous is done for the
upper half of the domain. This improves the initial guess of the equilibrium
extrapolation.

The convergence studies presented in this subsection are all run with $N = 32,
\dots, \num{2048}$ cells for both the unbalanced and adiabatically well-balanced method. The
approximate solution computed by the adiabatically well-balanced scheme on $N = \num{8192}$
cells is used as a reference solution.

\paragraph{Well-balanced property}
First we again check that the adiabatically well-balanced scheme is indeed
well-balanced. Therefore, we choose $A = 0$ and simulate until $\tend = 4$ which
corresponds to approximately four sound crossing times.

The results are shown in \cref{tab:bondi_wb}. The adiabatically well-balanced
scheme preserves the discrete equilibrium up to machine precision. The
unbalanced solver however accrues large $L^1$-errors for both values of the Mach
number.

\begin{figure}[htbp]
  \begin{center}
    \includegraphics[width=0.32\columnwidth]{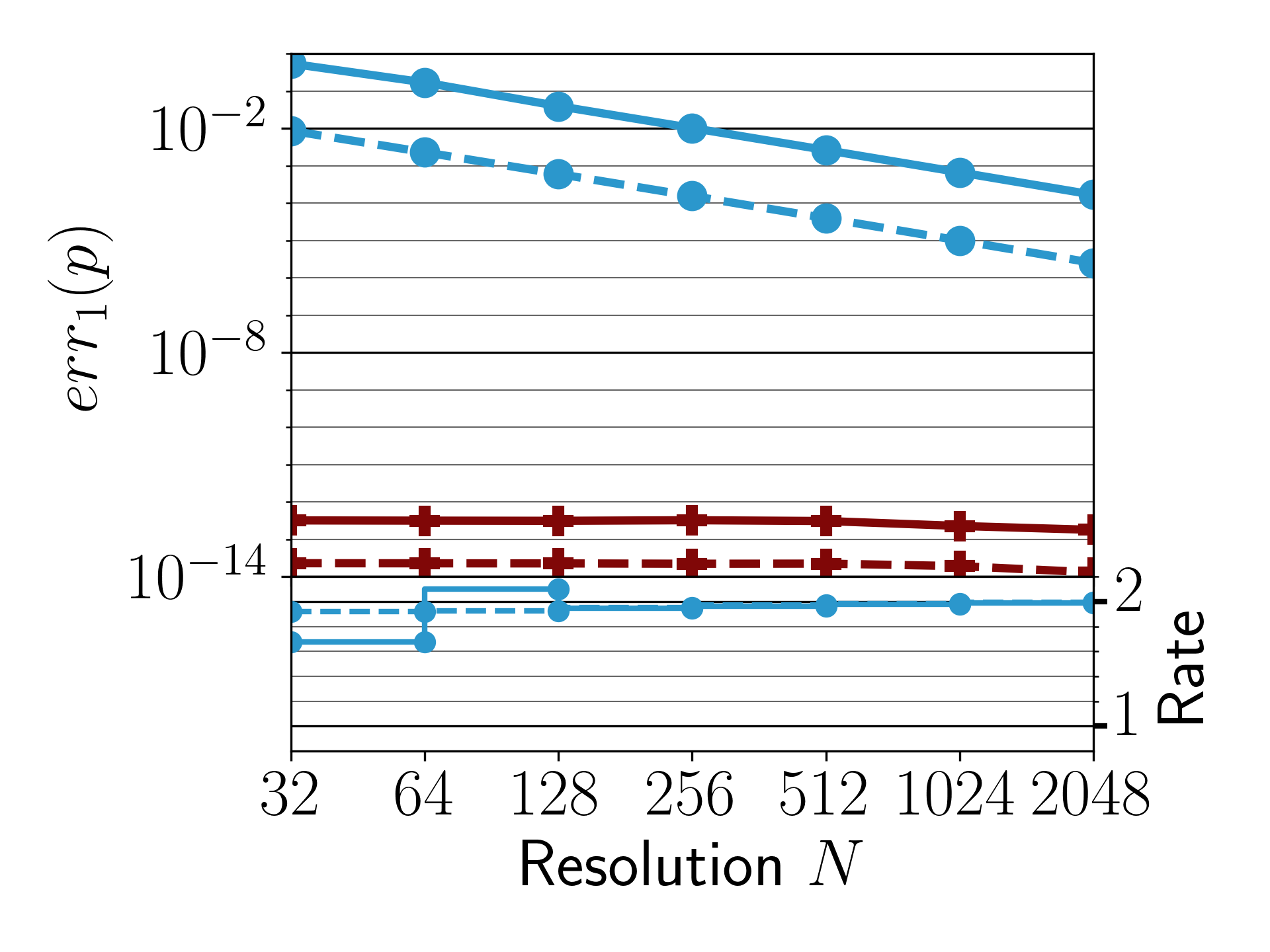}
    \includegraphics[width=0.32\columnwidth]{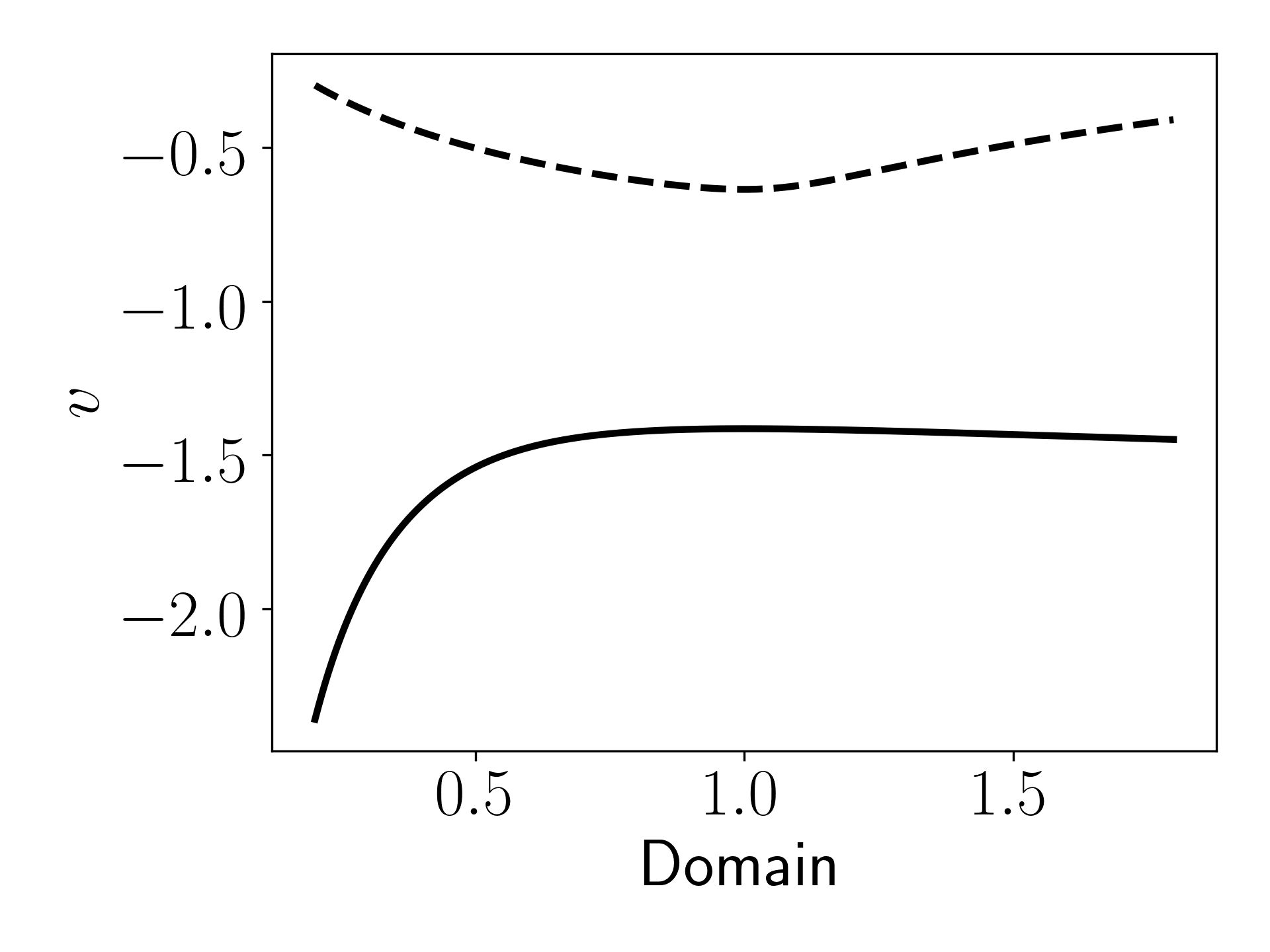}
    \includegraphics[width=0.32\columnwidth]{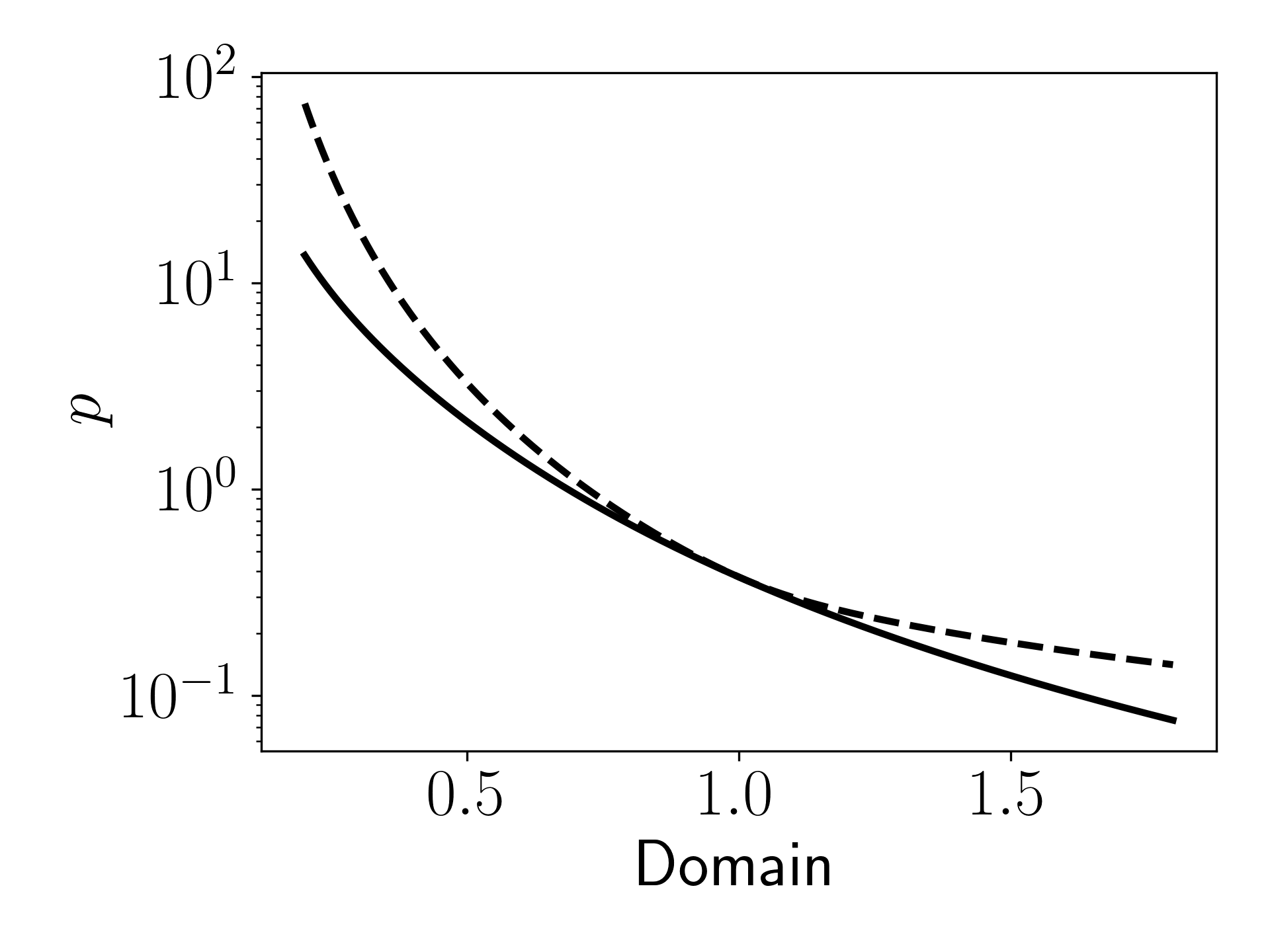}
  \end{center}
  \vspace{-0.2in}
  \caption{The left most column shows the well-balanced property of the methods
    on the problem described in \cref{numex:bondi} (with $A = 0$). The upper part of
    the subplot shows the $L^1$-error of the pressure perturbation $\delta p$. The
    lower subplot shows the convergence rate between two consecutive levels of
    refinement. The unbalanced and adiabatically well-balanced scheme is shown in
    blue and red, respectively. The solid black line is the reference solution. The
    middle column shows the velocity and the right column shows the pressure.
    Solid lines represent $M = 0.9$, the case $M = 2.0$ is depicted as a dashed
    line.}
  \label{fig:bondi_wb_a}
\end{figure}

\paragraph{Smooth wave propagation}
We shall now consider a small perturbation $A = 10^{-4}$ which remains smooth
throughout the simulation. The final time is chosen to be
\begin{align} \label{eq:numex_0210}
  \tend = 0.5 \min\left(
  0.6 \frac{R_1 - R_0}{c_{0} + \abs{v_{0}}},
  0.3 \frac{R_1 - R_0}{\max(c_{0} - \abs{v_{0}}, 10^{-10})}
  \right).
\end{align}

The results are show in \cref{fig:bondi_small_a,fig:bondi_small_b} and the
corresponding convergence table is \cref{tab:bondi_small}. For both values of
the Mach number the adiabatically well-balanced scheme resolves the wave
faithfully, and does not perturb the regions of the domain which have not yet
been reached by the wave. The resulting $L^1$-errors of the pressure
perturbation are a factor of $10^3$ and $10^2$ smaller than those of the
unbalanced method for $M = 0.9$ and $M = 2.0$ respectively. This dramatic
improvement implies that the error on the lowest resolutions $N = 32, 64$ in the
adiabatically well-balanced method is smaller than the errors of the standard
scheme on $N = 2048$ cells for $M = 0.9$ and is comparable to the error on $N =
512, 1024$ cell for $M = 2.0$.

\paragraph{Discontinuous wave propagation}
Next we consider a large perturbation $A = 100$. Due to the large amplitude of
the wave we use the classical \minmod. The solution develops a discontinuity
before the end of the simulation which is
\begin{align}
  \tend = 0.08 \min\left(
  0.6 \frac{R_1 - R_0}{c_{0} + \abs{v_{0}}},
  0.3 \frac{R_1 - R_0}{\max(c_{0} - \abs{v_{0}}, 10^{-10})}
  \right).
\end{align}

The result is shown in \cref{fig:bondi_large} and the corresponding convergence
table is \cref{tab:bondi_large}. The adiabatically well-balanced scheme is
robust in the presence of discontinuities and its performance is virtually
indistinguishable from the unbalanced scheme.

\begin{figure}[htbp]
  \begin{center}
    \includegraphics[width=0.32\columnwidth]{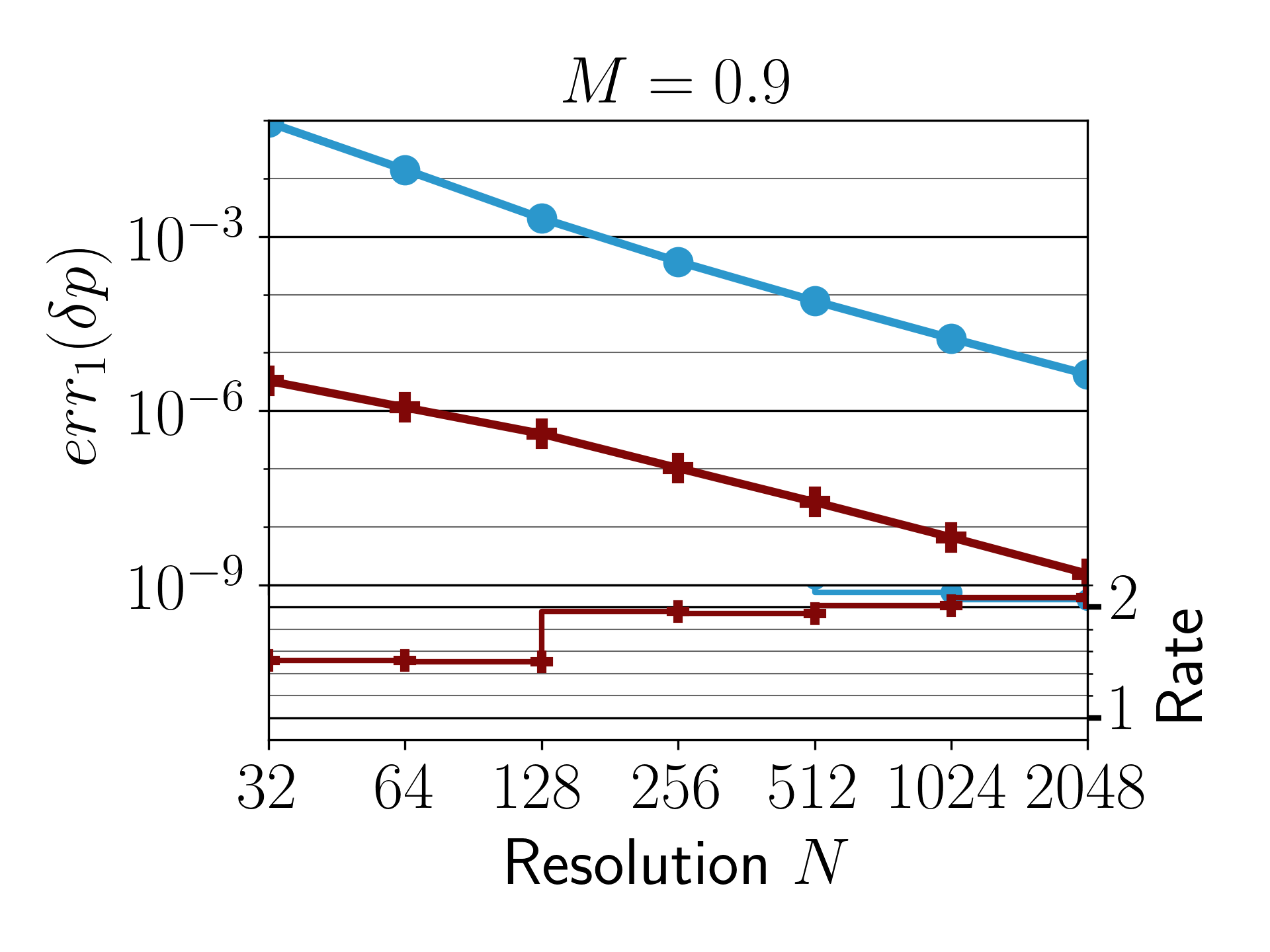}
    \includegraphics[width=0.32\columnwidth]{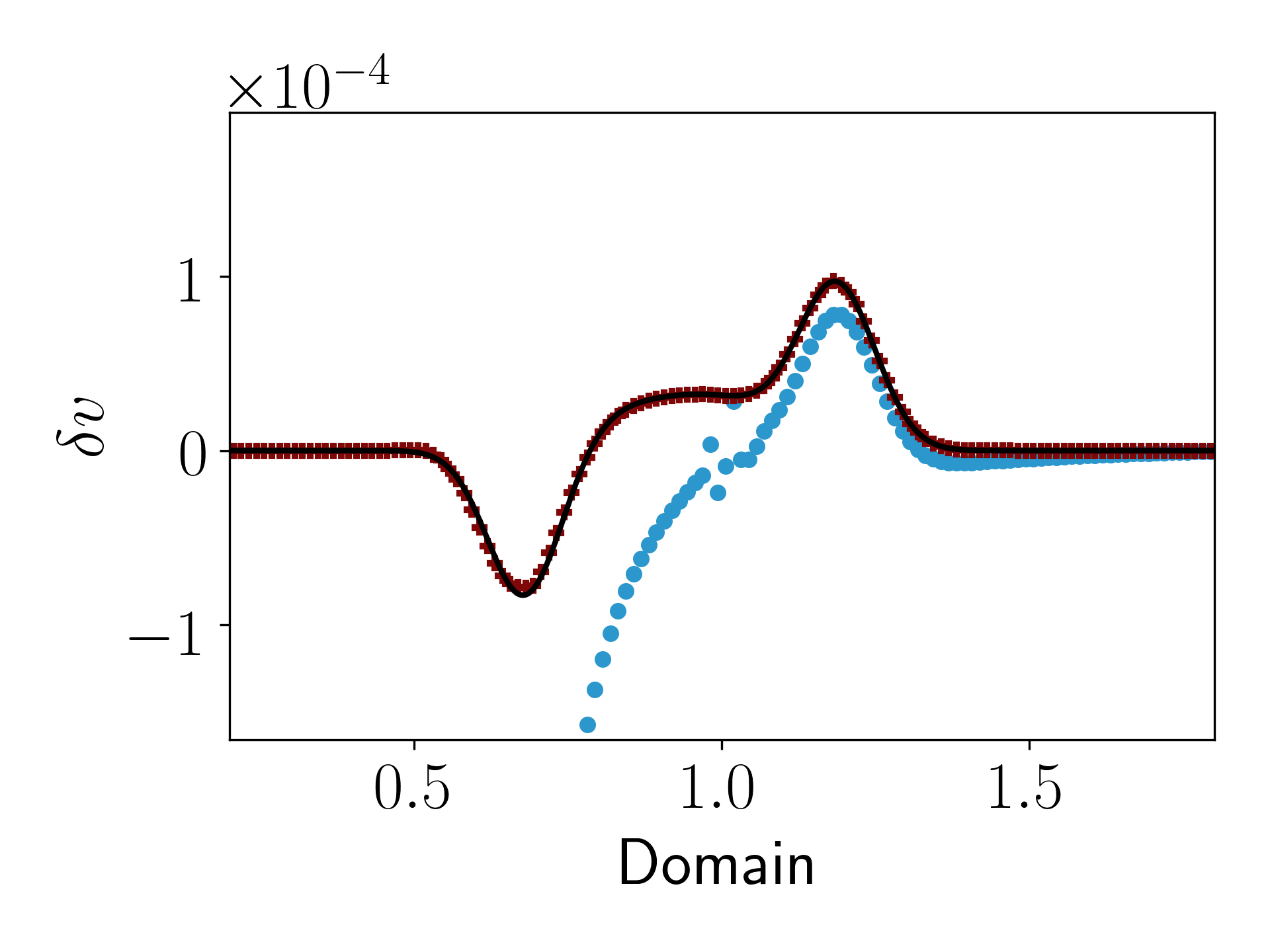}
    \includegraphics[width=0.32\columnwidth]{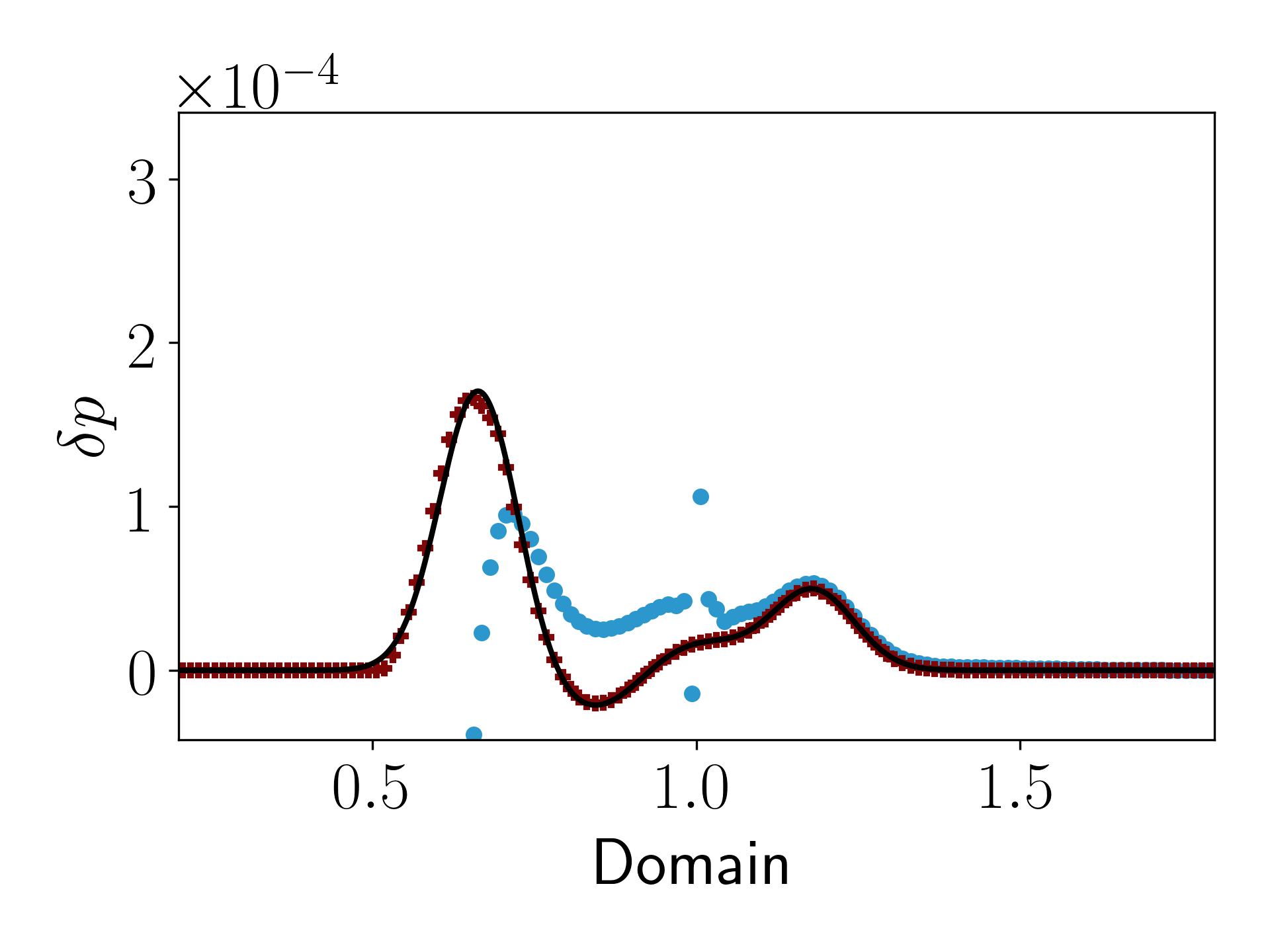}
  \end{center}
  \vspace{-0.2in}
  \caption{The left most column shows the convergence of the methods on the problem
    described in \cref{numex:bondi} with $A = 10^{-4}$. The upper part
    of the subplot shows the $L^1$-error of the pressure perturbation $\delta
    p$. The lower subplot shows the convergence rate between two consecutive
    levels of refinement. The middle column shows the velocity perturbation
    $\delta v$ and the right column shows the pressure perturbation $\delta p$. The
    Mach number at the reference point is $M = 0.9$. The scatter plots show the
    approximation with $N = 128$ cells at the final time described in the text. The
    unbalanced and adiabatically well-balanced scheme is shown in blue and red,
    respectively. The solid black line is the reference solution.}
  \label{fig:bondi_small_a}
\end{figure}

\begin{figure}[htbp]
  \begin{center}
    \includegraphics[width=0.32\columnwidth]{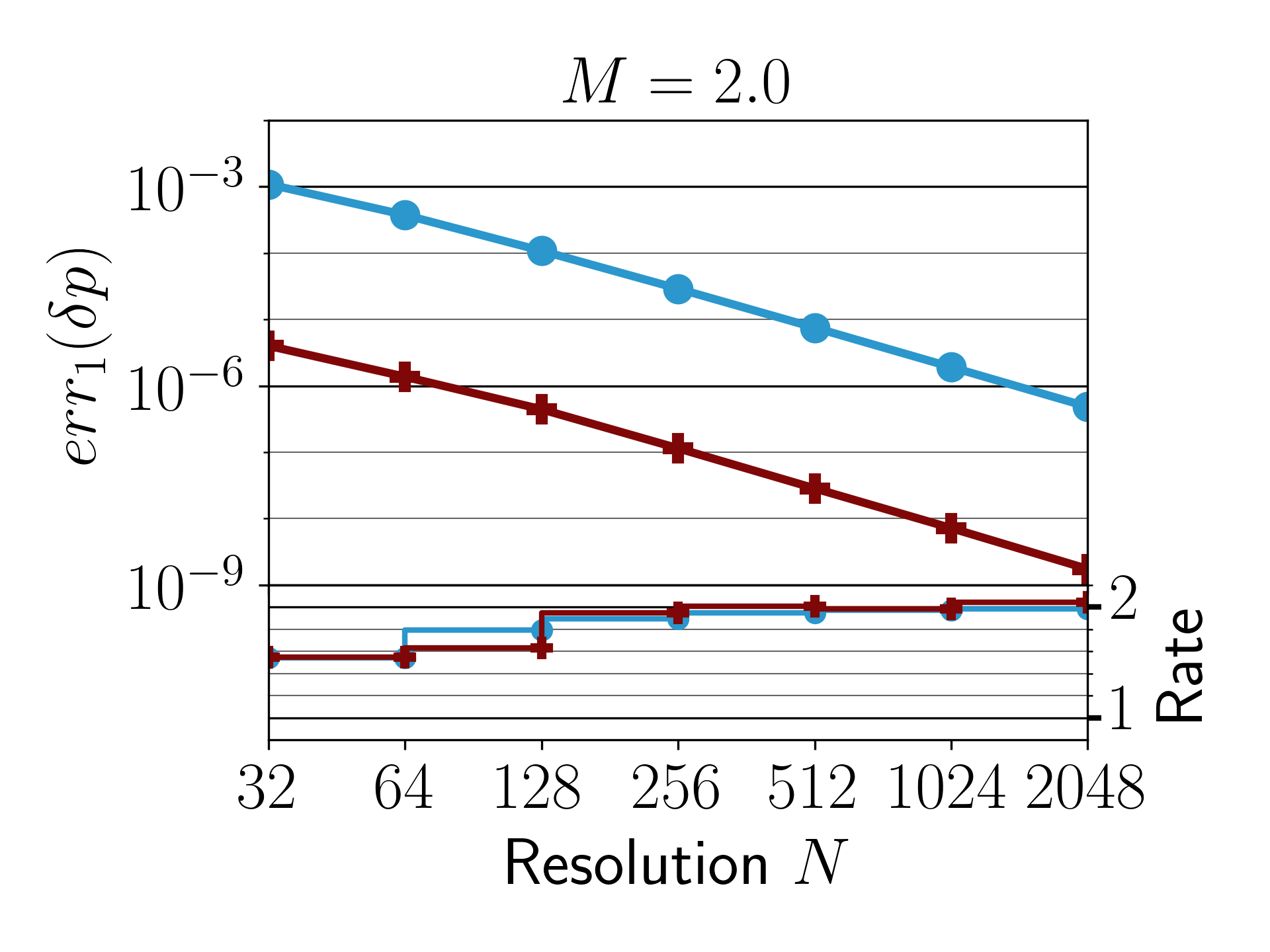}
    \includegraphics[width=0.32\columnwidth]{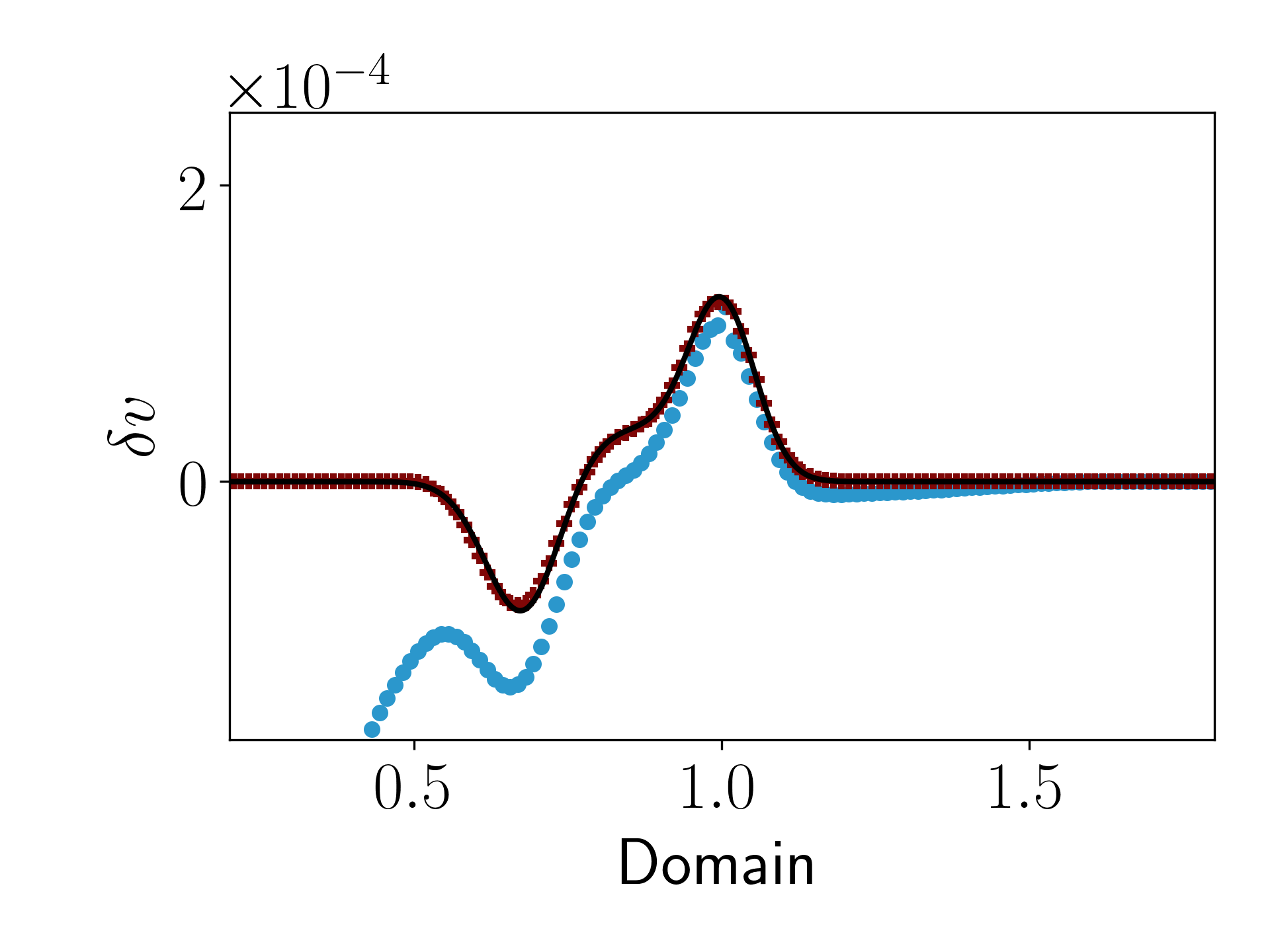}
    \includegraphics[width=0.32\columnwidth]{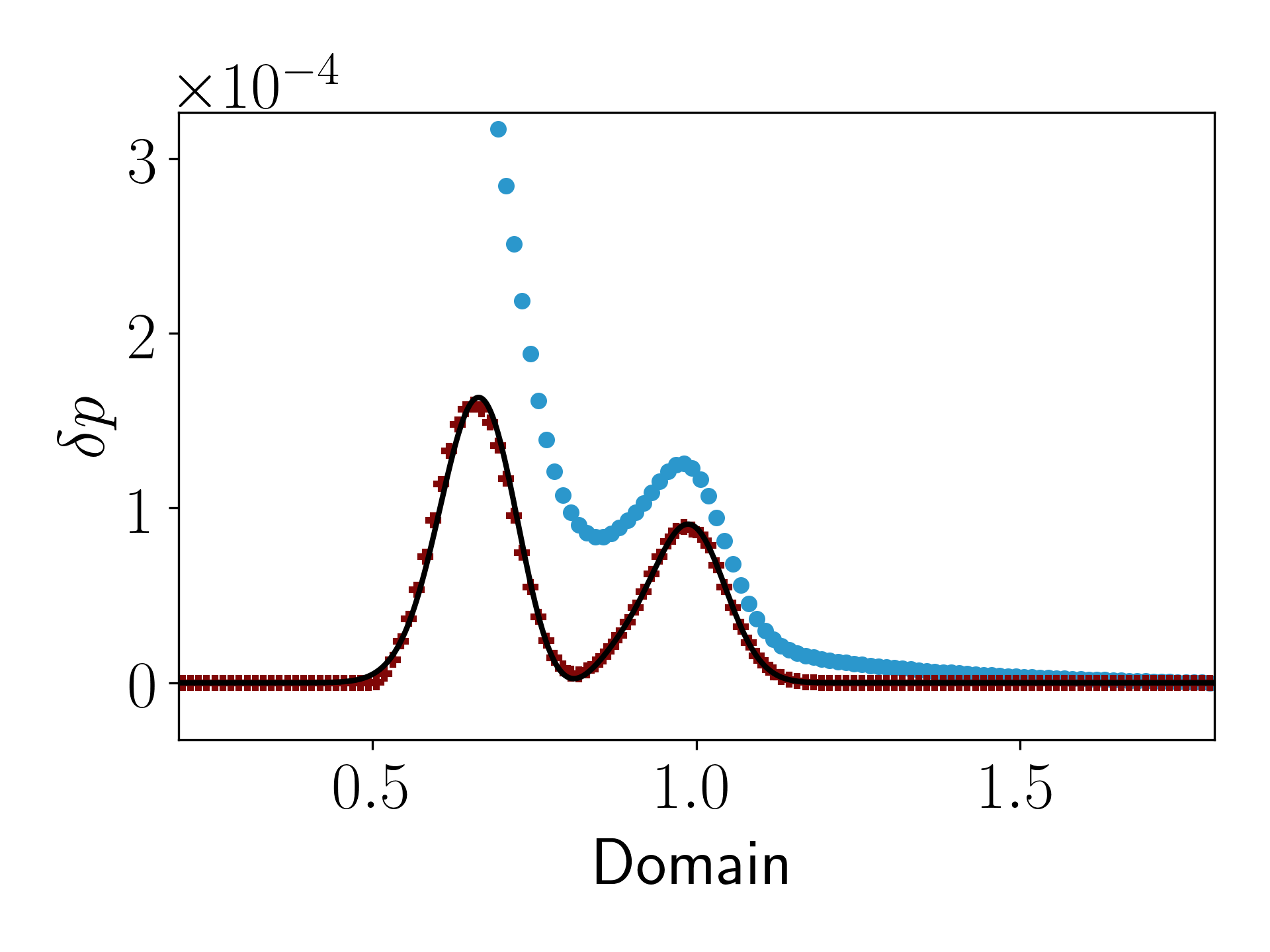}
  \end{center}
  \vspace{-0.2in}
  \caption{The figure shows the results for \cref{numex:bondi} with $A = 10^{-4}$
    and $M = 2$. Please refer to the caption of \cref{fig:bondi_small_a} for a
    detailed caption.}
  \label{fig:bondi_small_b}
\end{figure}

\begin{figure}[htbp]
  \begin{center}
    \includegraphics[width=0.32\columnwidth]{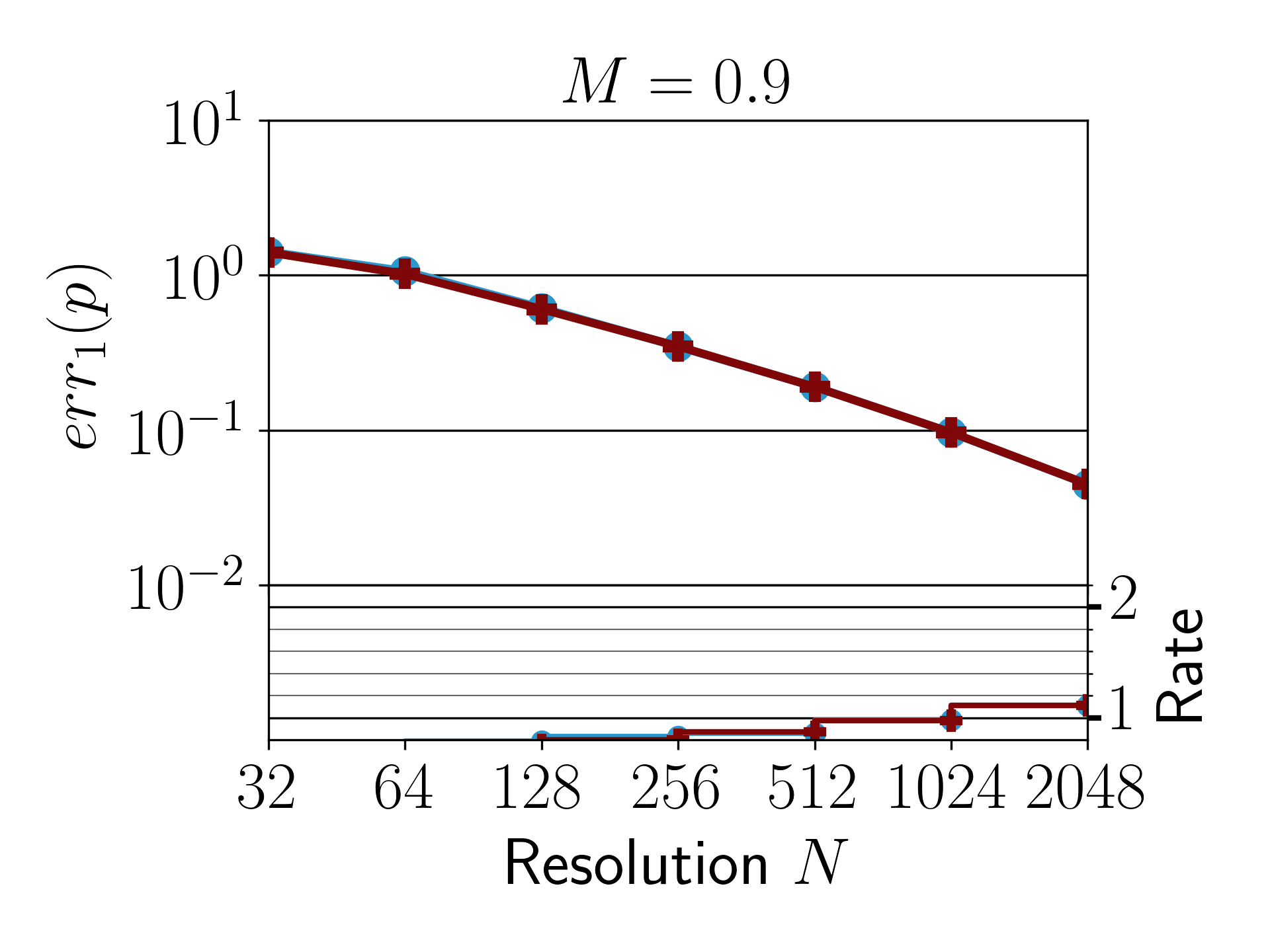}
    \includegraphics[width=0.32\columnwidth]{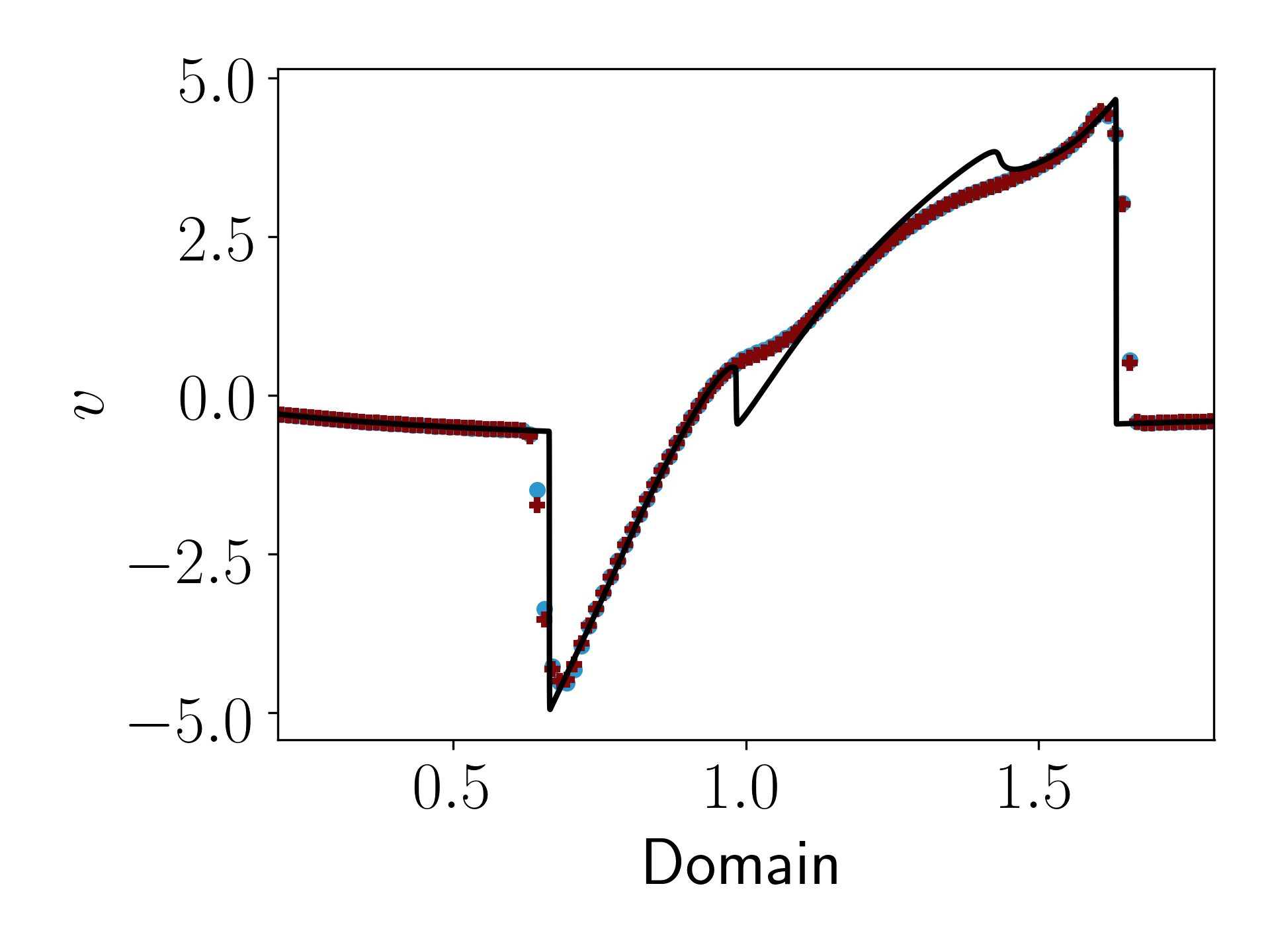}
    \includegraphics[width=0.32\columnwidth]{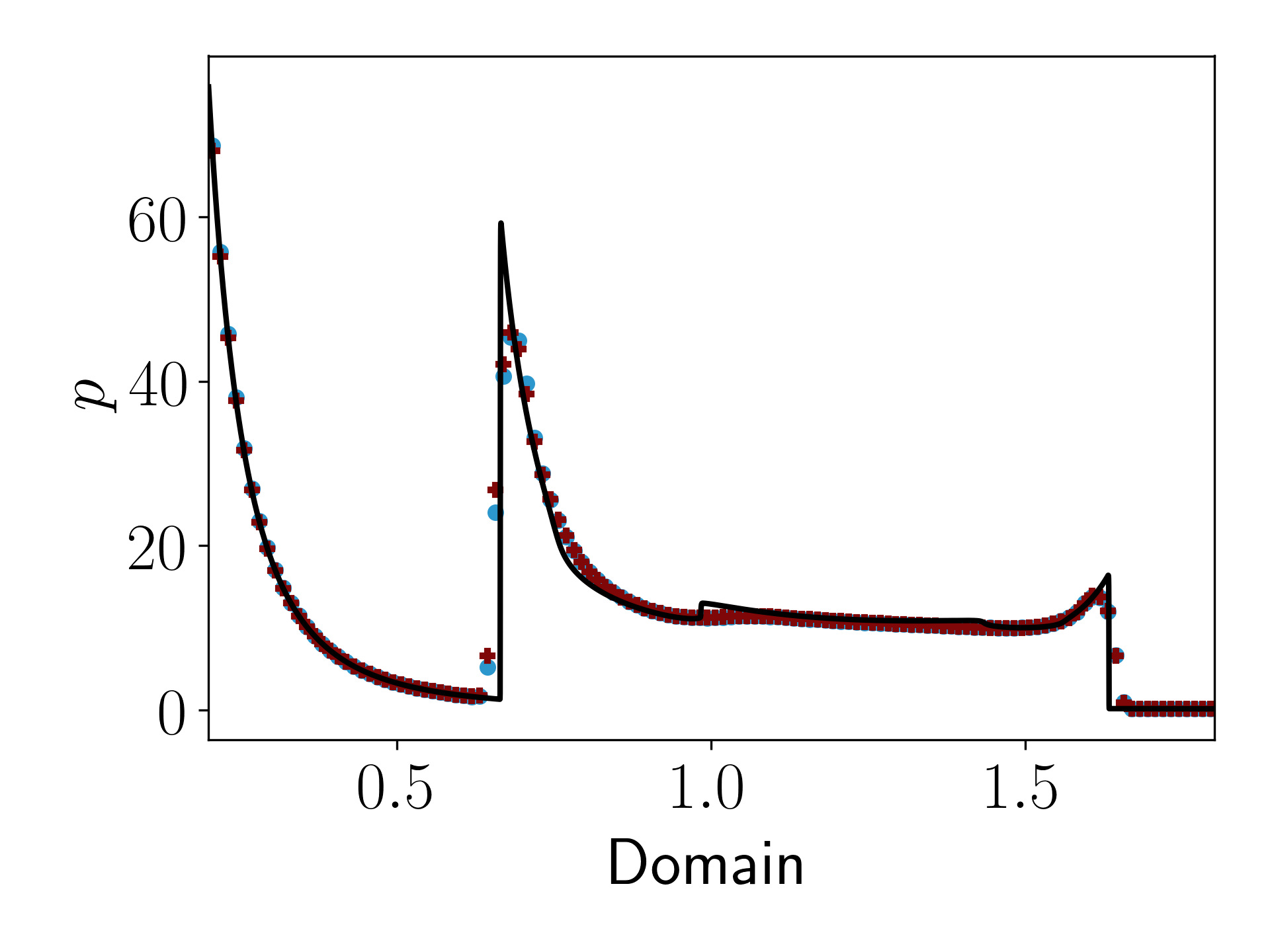}
    \includegraphics[width=0.32\columnwidth]{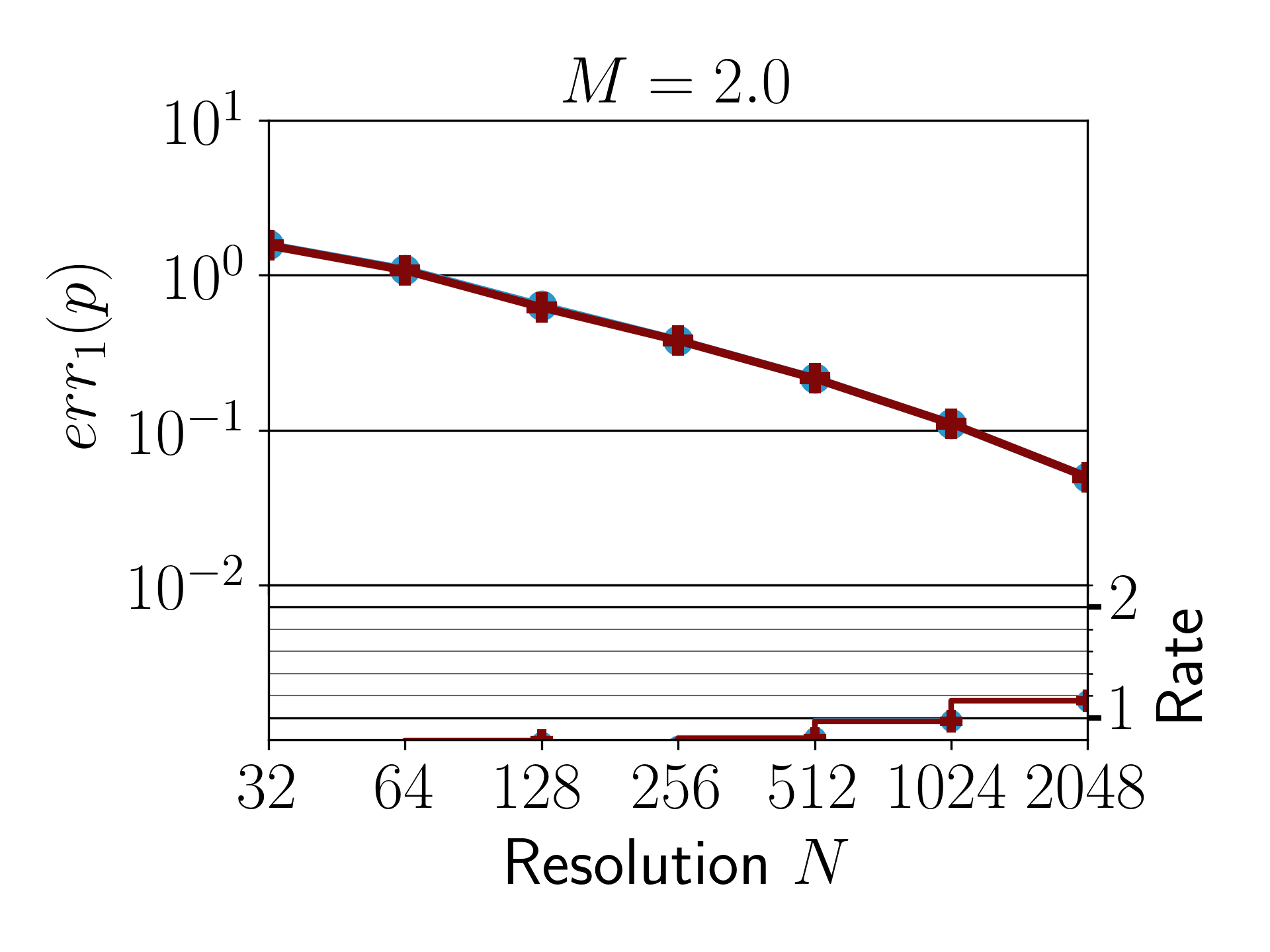}
    \includegraphics[width=0.32\columnwidth]{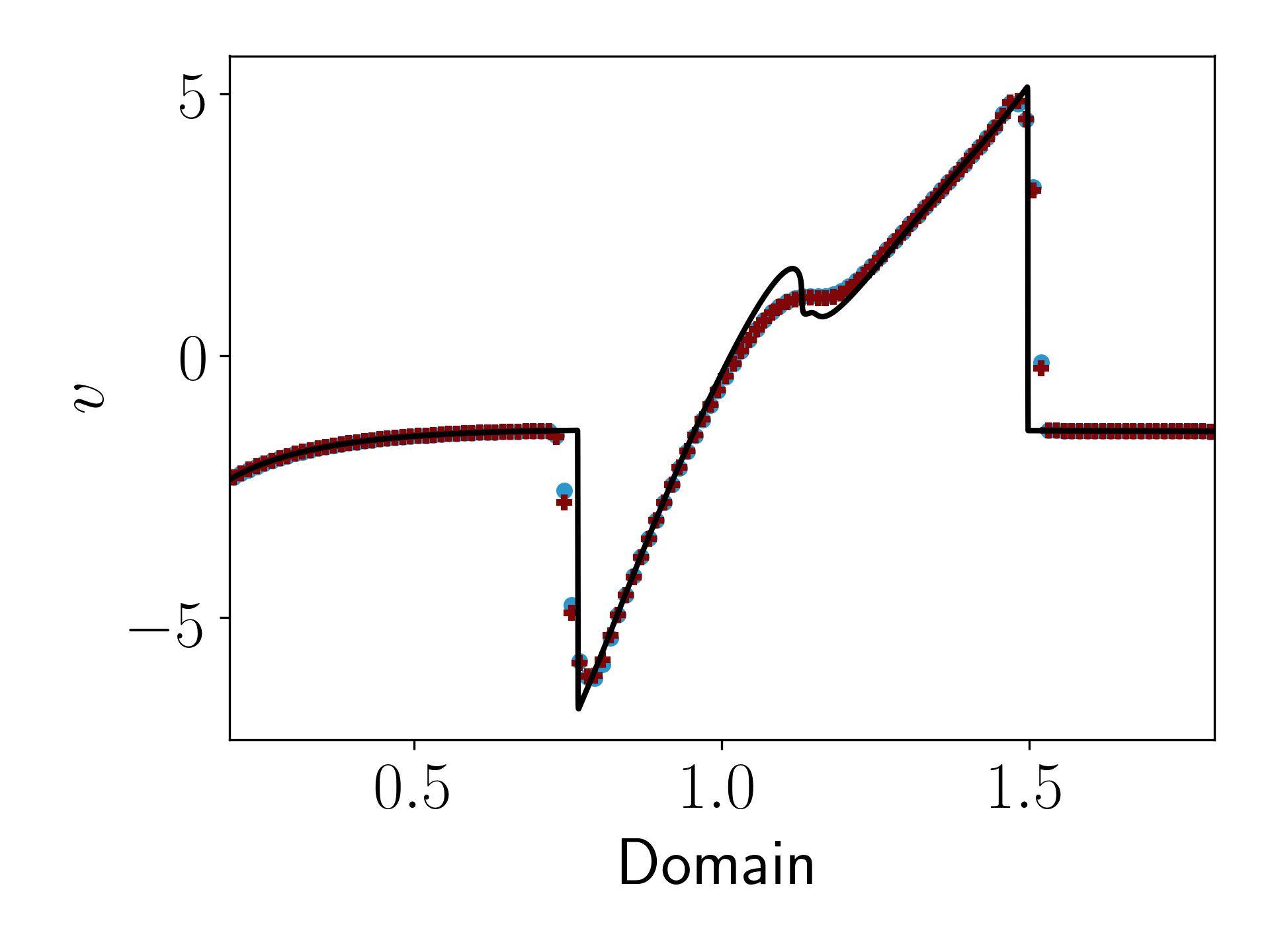}
    \includegraphics[width=0.32\columnwidth]{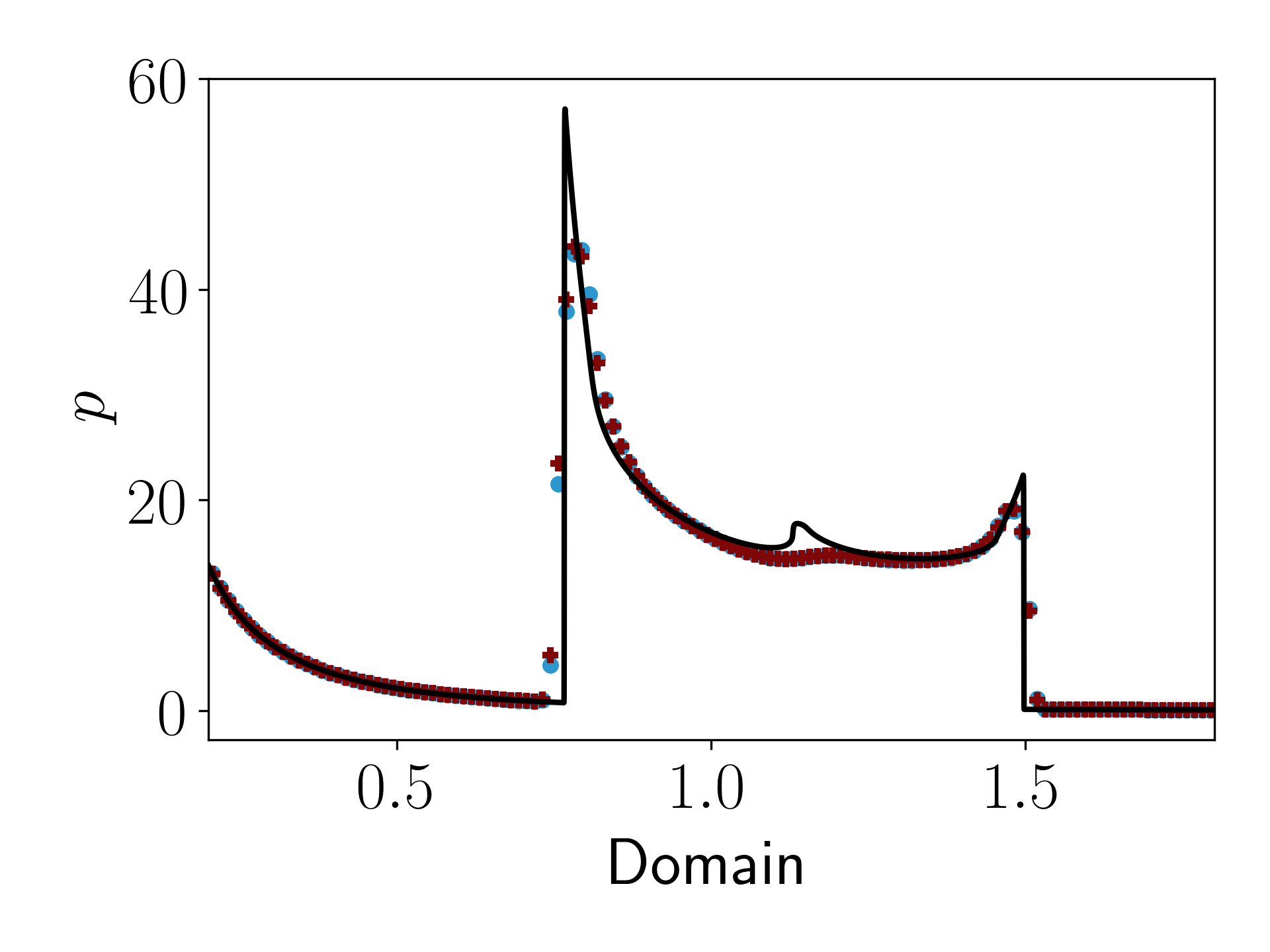}
  \end{center}
  \vspace{-0.2in}
  \caption{The left most column shows the convergence of the methods for the
    numerical experiment described in \cref{numex:bondi} with $A = 100$. The upper part
    of the subplot shows the $L^1$-error of the pressure $p$. The lower subplot
    shows the convergence rate between two consecutive levels of refinement. The
    middle column shows the velocity $v$ and the right column shows the pressure
    $p$. The Mach number at the reference point is $M = \num{0.9}$. The scatter
    plots show the approximation with $N = 128$ cells at the final time described in
    the text. The unbalanced and adiabatically well-balanced schemes are shown in blue and
    red, respectively. The solid black line is the reference solution.}
  \label{fig:bondi_large}
\end{figure}

\subsubsection{Discontinuous equilibrium}\label{numex:bondi_accretion_shock}
So far, the equilibrium was either subsonic everywhere or supersonic everywhere,
but never supersonic on one part of the domain and subsonic in another. In this
experiment we will study a flow which is supersonic in the upper half and
subsonic in the lower half of the domain. The two regions are joined by a
stationary shock. In this experiment we use the classical \minmod rather than
the monotonized centered limiter with the additional clipping of the pressure
and density. Furthermore, we employ the HLL(E) flux which is able to resolve
stationary discontinuities.

The shock is located at $r_0 = 1$. The pre-shock values are defined by
\begin{align}
  \rho_{0,1} = 1,\quad c^2_{0,1} = \frac{1}{2},\quad v_{0,1} = -M c_{0,1}
\end{align}
where $M = 1.2$ is the pre-shock Mach number. The conditions
immediately below the shock (e.g. post-shock) are given by the Rankine-Hugoniot
conditions for a stationary shock \cite{LandauLifshitz1987}
\begin{align}
  \rho_{0,2} = \rho_{0,1} \frac{(\adidx + 1) M^2}{(\adidx -1) M^2 + 2},
  \qquad
  p_{0,2} = p_{0,1} \left(\frac{2\adidx M^2}{\adidx + 1} - \frac{\adidx -1}{\adidx+1}\right),
  \qquad
  v_{0,2} = \frac{\rho_{0,1}}{\rho_{0,2}} v_{0,1}.
\end{align}
The initial conditions in the upper ($k = 1$) and lower ($k = 2$) halves are computed by
\begin{align}
  \left(\rho^0, v^0, p^0\right)(r)
  = \left(\rho_{eq,k}(r),\ v_{eq,k}(r),\ p_{eq,k}(r)\right)
\end{align}
where $\bw_{eq,k}$ is defined by the values $\rho_{0,k}$, $v_{0,k}$ and
$c_{0,k}$ in $r_0 = 1$. Therefore, the initial condition consists of joining
two, perturbed, equilibrium solutions by a stationary shock. Note that the
initial conditions in the upper half of the domain are the same as in
\cref{numex:bondi}. Furthermore, the jump is chosen to lie exactly in the middle
of the domain, hence if the number of cells is even, the shock is guaranteed to
be located at the boundary between two cells. The final time is $\tend = 2.0$
which corresponds to approximately two characteristic crossing times. 

The adiabatically well-balanced scheme preserves this discontinuous equilibrium
to machine precision, whereas the standard scheme accrues large errors. The
results for $N = 128$ is shown in \cref{fig:bondi_accretion_shock}.

\begin{figure}[htbp]
  \begin{center}
    \includegraphics[width=0.32\columnwidth]{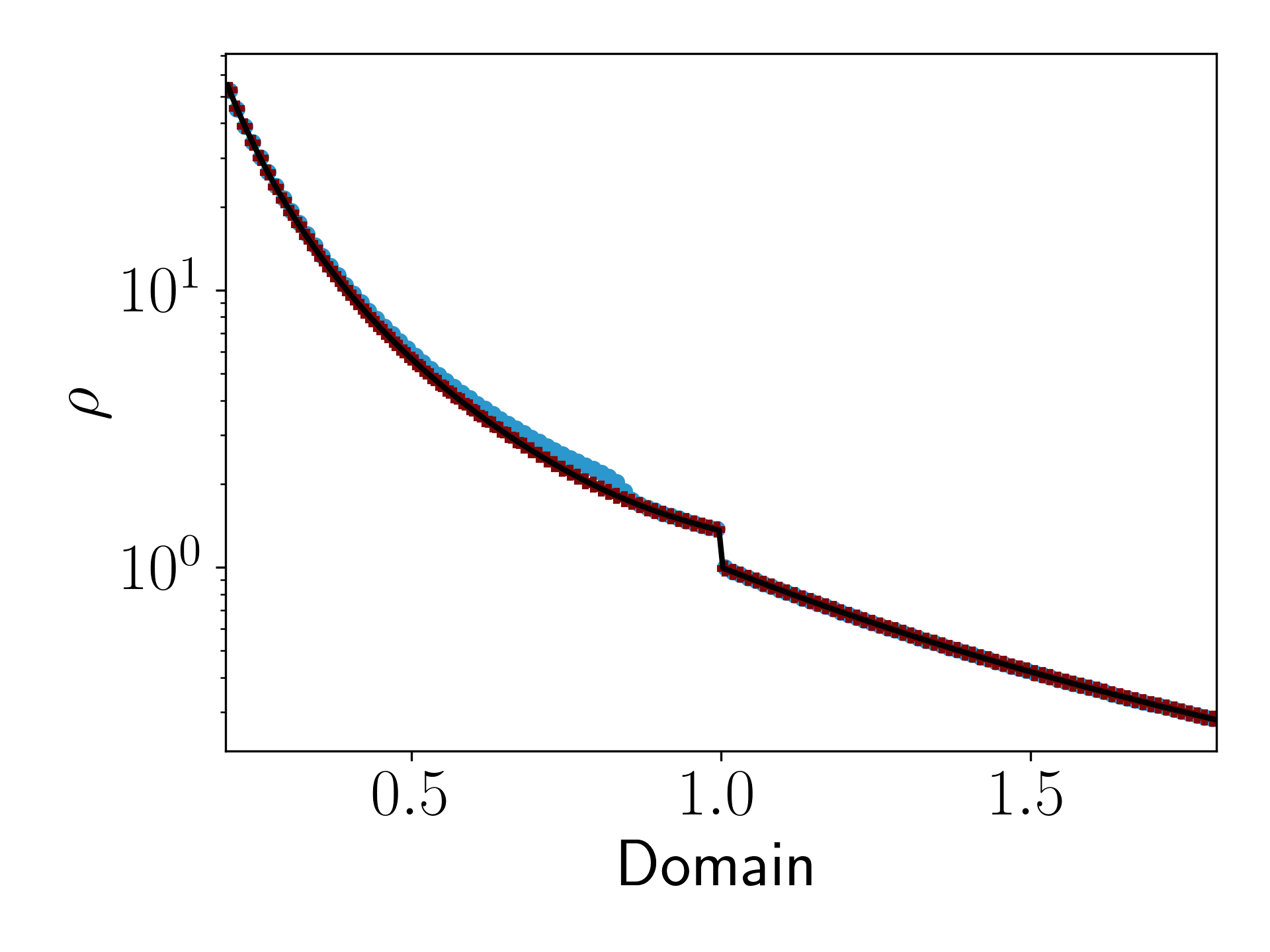}
    \includegraphics[width=0.32\columnwidth]{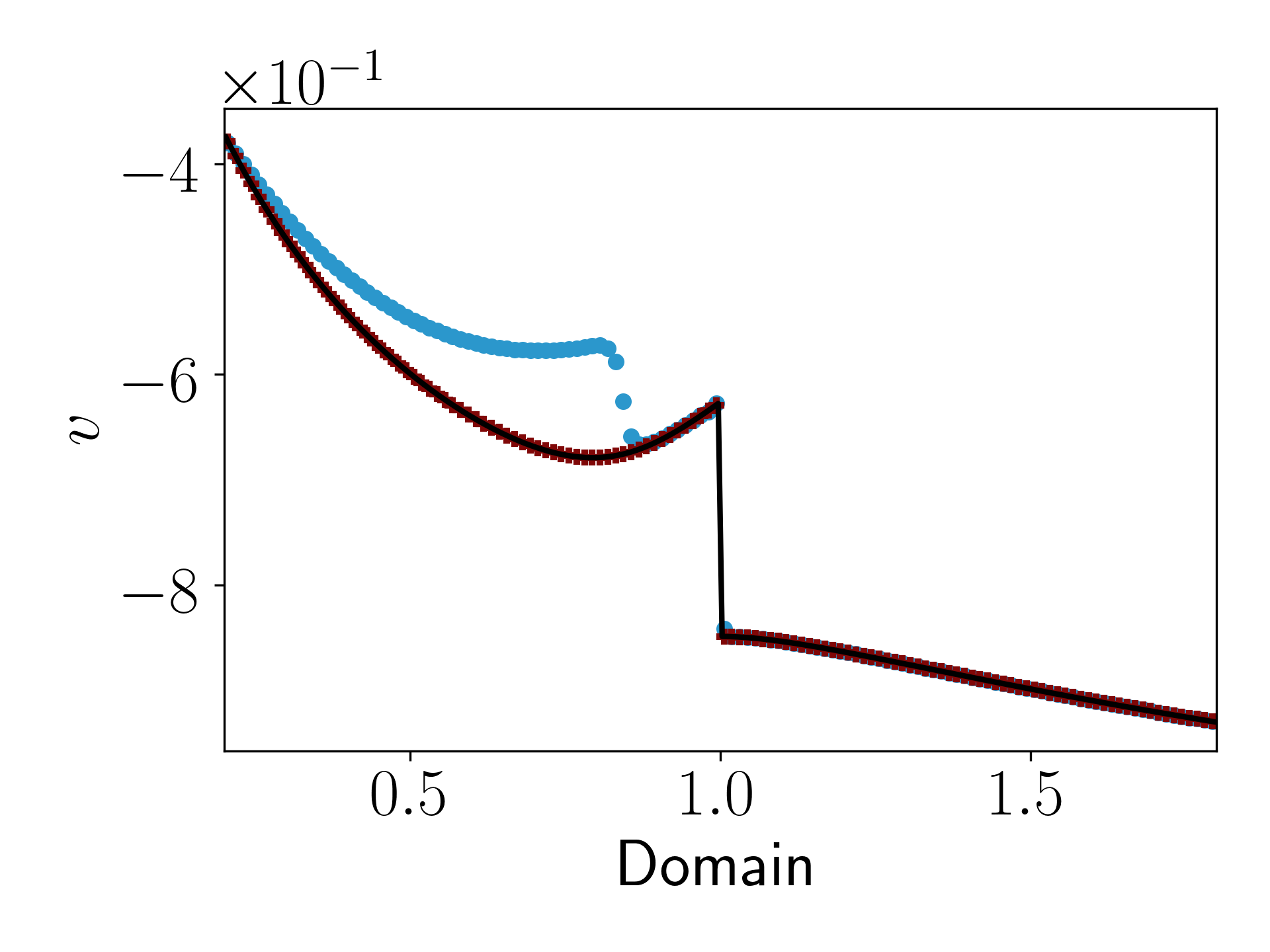}
    \includegraphics[width=0.32\columnwidth]{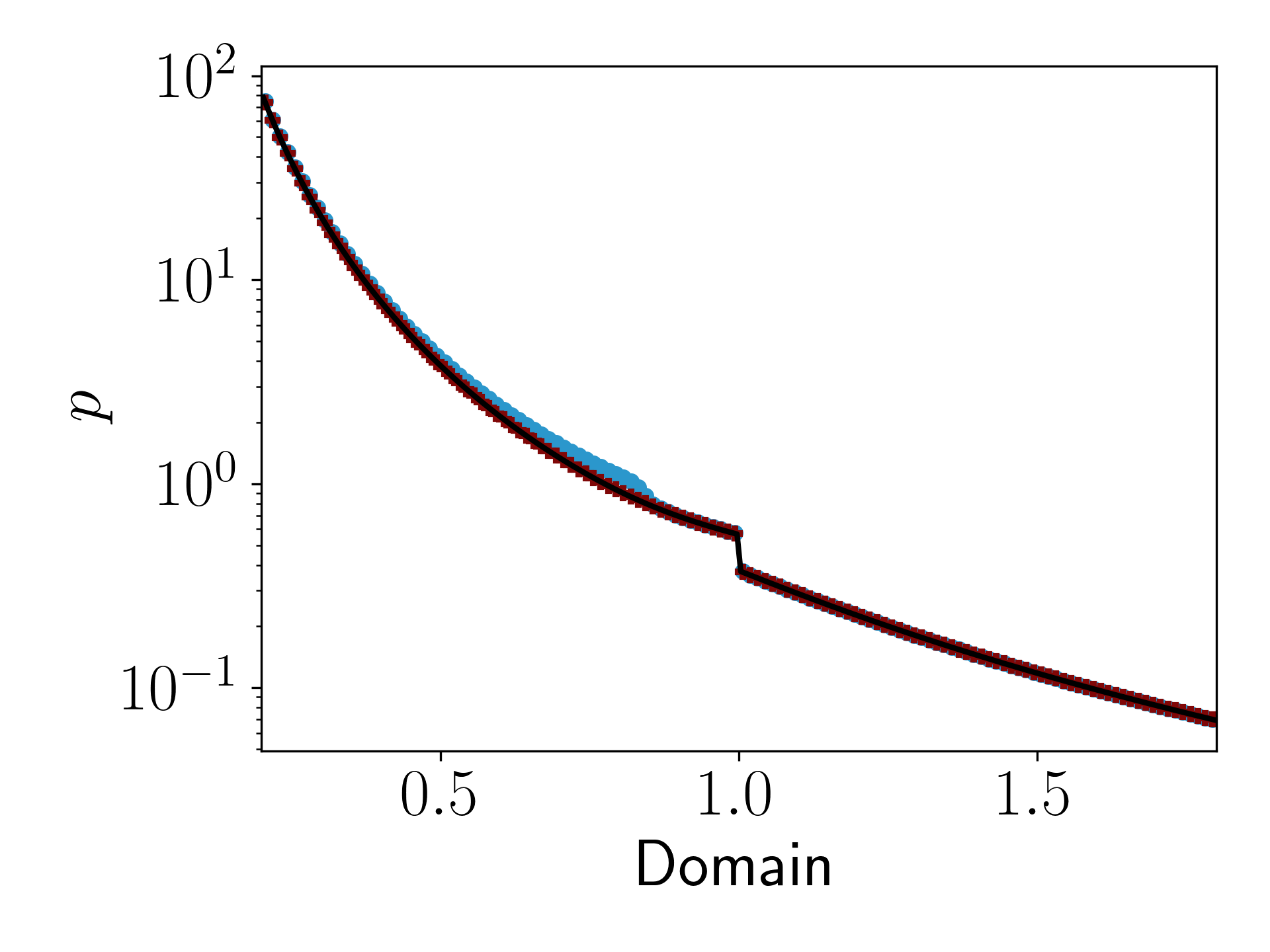}
    \\
    \includegraphics[width=0.32\columnwidth]{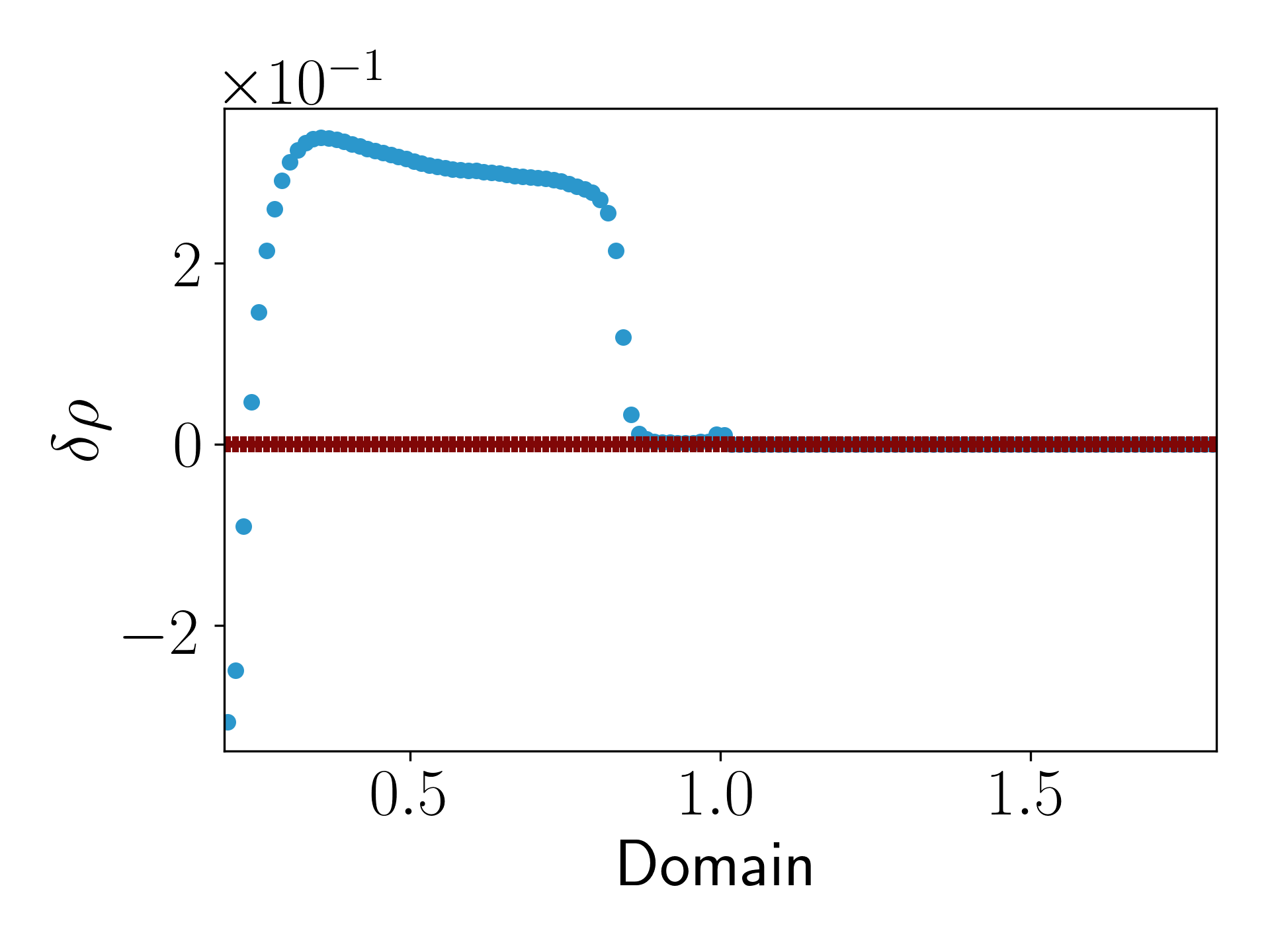}
    \includegraphics[width=0.32\columnwidth]{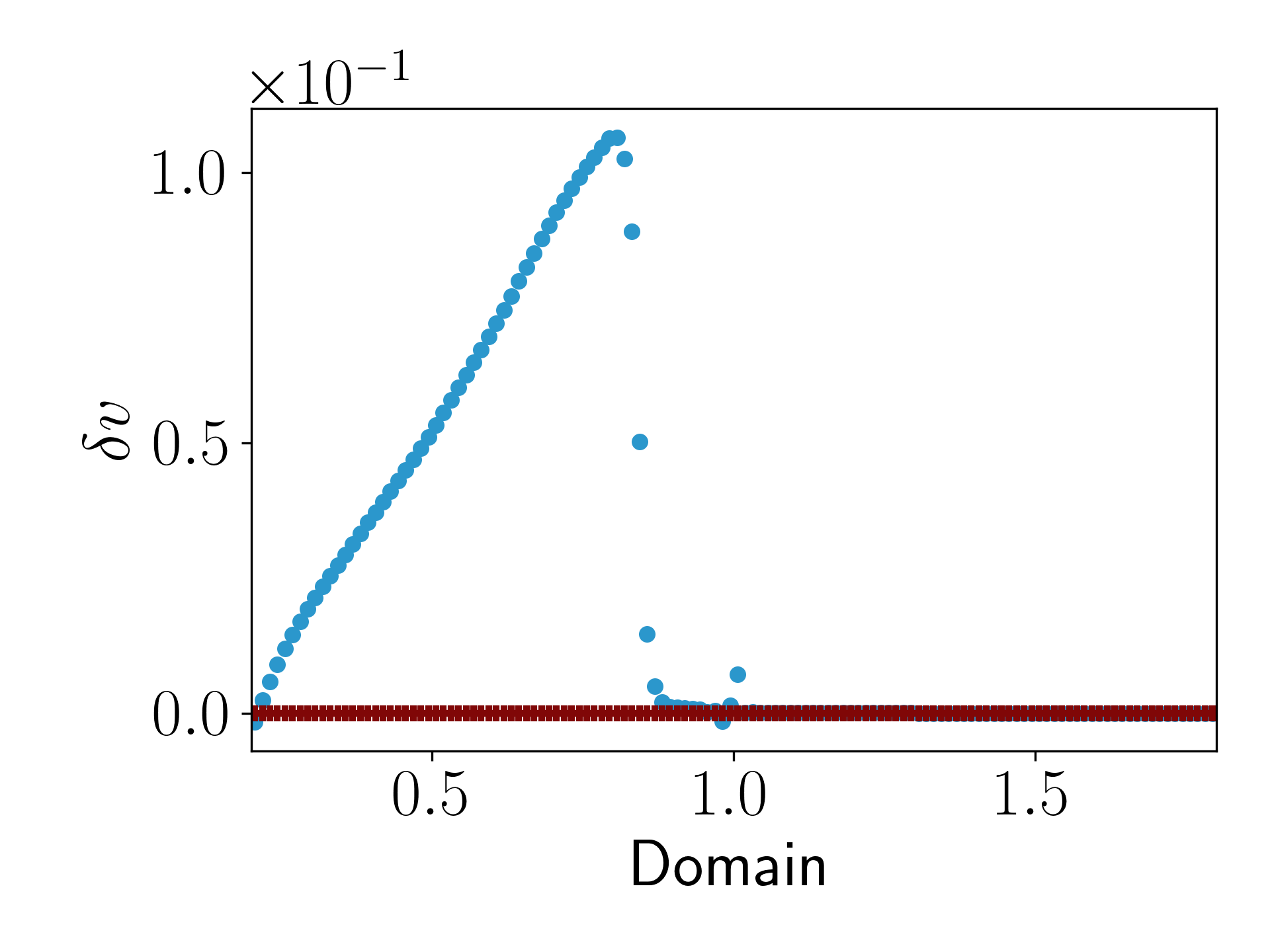}
    \includegraphics[width=0.32\columnwidth]{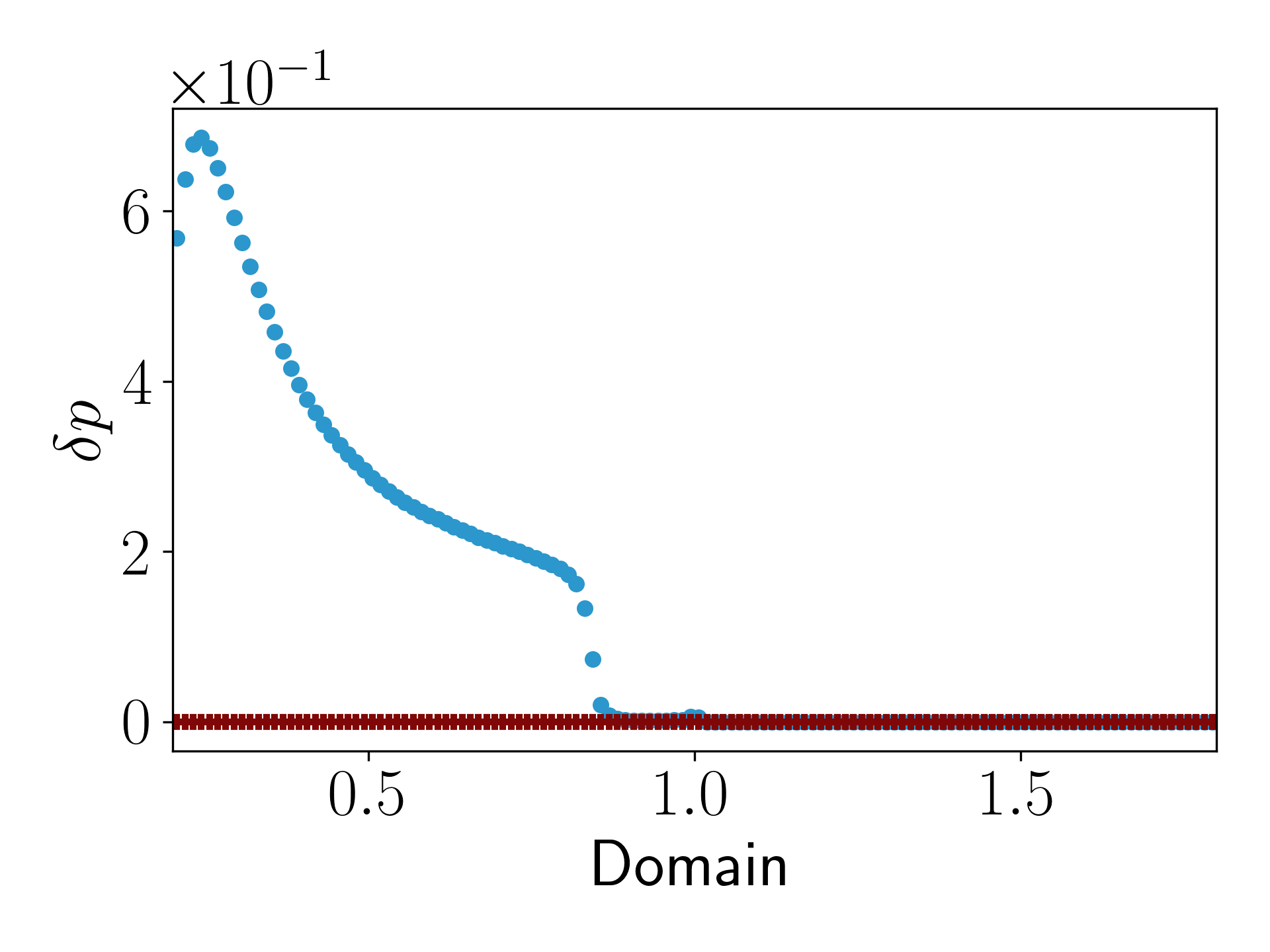}
  \end{center}
  \vspace{-0.2in}
  \caption{In the top row the density, velocity and pressure (left to right) of
    \cref{numex:bondi_accretion_shock}. The bottom row shows the error of those
    variables in the same order. All plots show an approximation of the solution
    with $N = 128$ cells after two characteristic timescales. The unbalanced and
    adiabatically well-balanced scheme is shown in blue and red, respectively.}
  \label{fig:bondi_accretion_shock}
\end{figure}

\subsection{Stellar accretion}
\label{subsec:numex_stella}
As a final test, we present the performance of our well-balanced schemes
involving a complex multi-physics \eos.
The test consists of the simulation of an astrophysical accretion scenario.
In particular, we consider the accretion onto a compact object as it is
typically encountered in a core-collapse.
Such an event marks the death of a massive star and its transition to a
compact object such as a neutron star or a black hole
\cite{ShapiroTeukolsky1983,Arnett1996,KippenhahnEtAl2012}.
In a core-collapse supernova, a standing accretion shock arises as an
expanding shock wave, generated by the sudden halt of the collapse of the
core due to the stiffening of the \eos above super-nuclear density, stalls
and remains nearly stationary for an extended period of time.
During this period, the shock is revived by some (yet unknown in detail)
combination of factors including neutrino heating, convection, rotation, and
magnetic fields, triggering a formidable explosion.
The standing accretion shock is subject to a dynamical instability commonly
known as the standing accretion shock instability (SASI)
\cite{BlondinEtAl2003}.
The latter may have deep implications on the explosion mechanism itself, ejecta
morphology, and pulsar kicks and spins (see e.g. Foglizzo et al.
\cite{FoglizzoEtAl2015} and references therein).

A typical radial profile is shown in \cref{fig:numex_stella_0010}
(dashed lines).
The radial profile was obtained from a simulation as described by Perego
\etal\ \cite{PeregoEtAl2016}.
The figure shows the standing accretion shock.
Above the shock, matter is falling in supersonically.
Below the shock, matter is falling in subsonically and pilling up onto the
nascent proto-neutron star.
We consider a highly simplified setup similar to
\cite{YamasakiFoglizzo2008,KazeroniEtAl2016}, which studies the dynamics of a
standing accretion shock around a proto-neutron star restricted to the
equatorial plane using cylindrical coordinates.

The computational domain spans $r \in [50, 450]$ km in radius and the full
angular realm $\varphi \in [0, 2 \pi]$.
The accreting matter is modeled by a mixture of (photon) radiation, nuclei,
electrons and positrons as provided by the publicly available Helmholtz \eos
of Timmes and Swesty \cite{Timmes2000}.
The gravitational attraction of the proto-neutron star is modeled by a
point mass
\begin{equation}
\label{eq:numex_stella_0010}
  \phi(r) = - \frac{G M}{r},
\end{equation}
where $G$ is the gravitational constant and we set $M = 1.3 M_{\odot}$
(solar masses).
We set the shock radius to $r_{\mathrm{sh}} = 150$ km and the pre-shock
conditions as
\begin{equation}
\label{eq:numex_stella_0020}
  \begin{aligned}
    \rho_{\mathrm{pre}} & = \phantom{-}
                              2.0872995 \times 10^{ 8} ~ \mathrm{g/cm}^3 , \\
    v_{r,\mathrm{pre}}  & = - 4.5618302 \times 10^{ 4} ~ \mathrm{km/s}   , \\
    v_{\varphi,\mathrm{pre}} & = \phantom{-} 0 ~ \mathrm{km/s}           , \\
    p_{\mathrm{pre}}    & = \phantom{-}
                              1.1538646 \times 10^{26} ~ \mathrm{erg/cm}^3
    .
  \end{aligned}
\end{equation}
The post-shock conditions are obtained from the Rankine-Hugoniot relations as
\begin{equation}
\label{eq:numex_stella_0030}
  \begin{aligned}
    \rho_{\mathrm{post}} & = \phantom{-}
                               1.1989635 \times 10^{ 9} ~ \mathrm{g/cm}^3 , \\
    v_{r,\mathrm{post}}  & = - 7.9417814 \times 10^{ 3} ~ \mathrm{km/s}   , \\
    v_{\varphi,\mathrm{pre}} & = \phantom{-} 0 ~ \mathrm{km/s}           , \\
    p_{\mathrm{post}}    & = \phantom{-}
                               3.7029093 \times 10^{27} ~ \mathrm{erg/cm}^3
    .
  \end{aligned}
\end{equation}
In the whole domain, we assume the matter to be composed of nickel isotope
$^{56}$Ni.
By numerically solving for the steady state conditions, one obtains the
solid line profiles in \cref{fig:numex_stella_0010}.
As apparent from the figure, the simplified setup matches the more complex
model quite well (despite the different \eos and geometry).

Next, we give some implementation details related to the complex multi-physics
\eos.
In the publicly available Helmholtz \eos, the photons are treated as black
body radiation in local thermal equilibrium and the nuclei by an ideal gas
law.
The electrons and positrons are treated in a tabular manner allowing speeds
arbitrarily close to the causal limits and arbitrary degree of degeneracy
with a thermodynamically consistent interpolation method.
The \eos interface provides all the relevant thermodynamic quantities given
the temperature $T$, density $\rho$ and composition.
The composition is specified by the triplet $(X_{i}, A_{i}, Z_{i})$ for each
isotope $i$, where $X_{i}$ is the mass fraction, $A_{i}$ the mass number
and $Z_{i}$ the atomic number.
In the present problem setup, we thus have only one isotope with
$(X_{1} = 1, A_{1} = 56, Z_{1} = 28)$.
Because we evolve the Euler equations in conservative form, we have to
determine the temperature corresponding to a given density $\rho$ and
specific internal energy $e$.
Similarly in the local equilibrium reconstruction, we have to determine the
temperature $T$ and the density $\rho$ corresponding to a given specific
enthalpy $h$ and entropy $s$.
This is implemented with robust root finding algorithms combining Newton's
method for speed and the bisection method for robustness
(see e.g. Press \etal \cite{Press1993} for details.).

In the following, we show the performance of the second-order adiabatically
well-balanced scheme and compare it to a standard second-order unbalanced
scheme.
Both schemes have been modified in a standard and identical fashion to cope
with the cylindrical geometry.
The standard unbalanced scheme is obtained by simply disabling the
well-balanced reconstruction and source term discretization of the
well-balanced scheme.
Both schemes use the HLL Riemann solver with simple wave speed estimates
from the fastest left/right traveling characteristic speeds at the
cell interface (see e.g. Toro \cite{Toro2009}).
Note that these speed estimates will not exactly resolve isolated shocks.
Therefore, we divide the full test problem in a subsonic $r \in [50, 150]$ km
and a supersonic $r \in [150, 450]$ part when studying the schemes'
well-balanced property in \cref{subsubsec:numex_stella_wb} and the wave
propagation properties in \cref{subsubsec:numex_stella_pert}.
In \cref{subsubsec:numex_stella_sasi}, we show the performance of the schemes
on the full problem.
In all the tests, the computational domain spans the full angular realm
$\varphi \in [0, 2 \pi]$.

\begin{figure}[hbt]
  \begin{center}
    \includegraphics[width=0.5\textwidth]{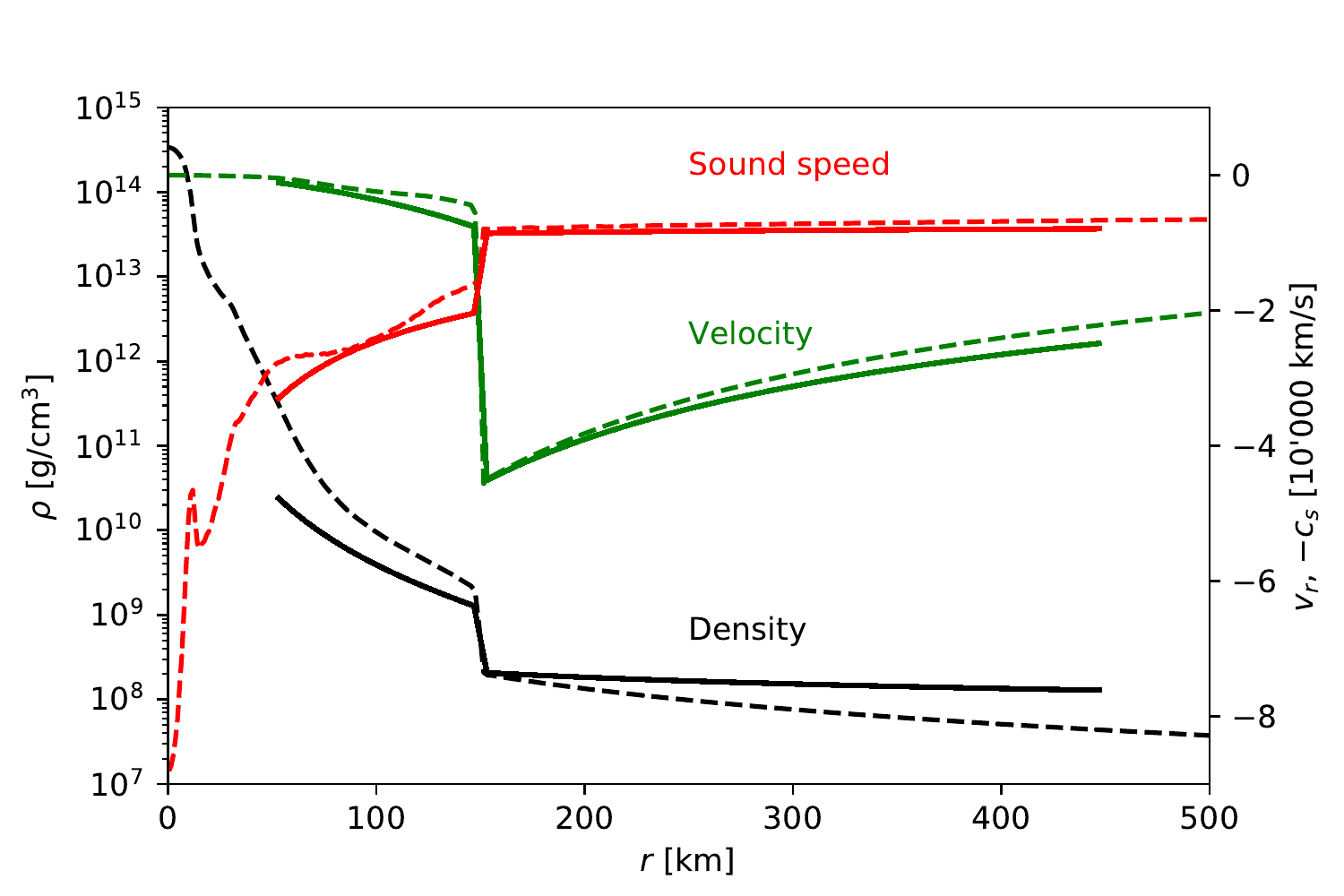}
  \end{center}
  \caption{Radial profiles of density (black lines), radial velocity (green
           lines) and sound speed (red lines) for the stellar accretion test
           problem \cref{subsec:numex_stella}.
           The dashed lines show the profiles from a core-collapse simulation
           as described by Perego \etal\ \cite{PeregoEtAl2016}.
           The solid lines show the highly simplified setup considered in the
           stellar accretion test problem.}
  \label{fig:numex_stella_0010}
\end{figure}

\subsubsection{Well-balanced property}
\label{subsubsec:numex_stella_wb}
We begin by numerically verifying the well-balancing properties of the
developed scheme.
For this purpose, we evolve the subsonic and supersonic accretion steady
states for two characteristic times $\tend = 2 \sndxtime$, and several
resolutions in radial and angular directions:
$(N_{r}, N_{\varphi}) = (32, 64)$, $(64, 128)$, $(128, 256)$, $(256, 512)$.
The domain boundaries in radial direction are simply kept frozen in time at
the equilibrium state.
The relative equilibrium errors in density, radial velocity and pressure are
displayed in \cref{tab:numex_stella_wb_0010} for the subsonic and
\cref{tab:numex_stella_wb_0020} for the supersonic parts.
We observe that the well-balanced scheme produces errors on the order of the
precision with which the local equilibrium is numerically solved
($tol = 10^{-13}$).
In contrast, the unbalanced scheme suffers from comparatively large errors
and is unable to maintain the steady state.

\begin{table}[htb!]
  \centering
  \begin{tabular}{| r || c | c | c |}
    \hline
    $(N_{r}, N_{\varphi})$ & $relerr_{1}(\rho)$
                           & $relerr_{1}(v_{r})$
                           & $relerr_{1}(p)$                 \\
    \hline \hline
      ( 32,  64) & 2.33E-02 / 4.13E-13 & 1.80E-01 / 9.95E-13
                 & 2.59E-02 / 4.07E-13                       \\
      ( 64, 128) & 4.14E-03 / 7.45E-13 & 3.61E-02 / 1.27E-11
                 & 4.46E-03 / 6.32E-13                       \\
      (128, 256) & 7.81E-04 / 2.17E-12 & 8.89E-03 / 4.99E-11
                 & 8.37E-04 / 2.02E-12                       \\
      (256, 512) & 1.86E-04 / 5.85E-12 & 1.74E-03 / 9.09E-11
                 & 1.91E-04 / 5.31E-12                       \\
    \hline
    Order        &     2.33 / -        &     2.21 / -
                 &     2.37 / -                              \\
    \hline
  \end{tabular}
  \caption{Relative equilibrium error in density, radial velocity and
           pressure for the subsonic accretion steady state computed with
           the un-/well-balanced  second-order schemes for two characteristic
           times $\tend = 2 \sndxtime
                  \approx 6.872778 \times 10^{-3}$ s.}
  \label{tab:numex_stella_wb_0010}
\end{table}

\begin{table}[htb!]
  \centering
  \begin{tabular}{| r || c | c | c |}
    \hline
    $(N_{r}, N_{\varphi})$ & $relerr_{1}(\rho)$
                           & $relerr_{1}(v_{r})$
                           & $relerr_{1}(p)$                 \\
    \hline \hline
      ( 32,  64) & 1.58E-04 / 2.68E-15 & 3.36E-04 / 2.94E-15
                 & 2.13E-04 / 1.22E-14                       \\
      ( 64, 128) & 3.78E-05 / 1.87E-15 & 8.42E-05 / 1.32E-15
                 & 4.81E-05 / 9.95E-15                       \\
      (128, 256) & 9.24E-06 / 1.25E-12 & 2.11E-05 / 1.25E-12
                 & 1.13E-05 / 8.15E-12                       \\
      (256, 512) & 2.28E-06 / 8.30E-13 & 5.29E-06 / 8.39E-13
                 & 2.75E-06 / 5.84E-12                       \\
    \hline
    Order        &     2.04 / -        &     2.00 / -
                 &     2.09 / -                              \\
    \hline
  \end{tabular}
  \caption{Relative equilibrium error in density, radial velocity and
           pressure for the supersonic accretion steady state computed with
           the un-/well-balanced second-order schemes for two characteristic
           times  $\tend = 2 \sndxtime 
           \approx 1.506179 \times 10^{-2}$ s.}
  \label{tab:numex_stella_wb_0020}
\end{table}

\subsubsection{Small and large amplitude perturbations}
\label{subsubsec:numex_stella_pert}
To verify the capability of the schemes to evolve perturbation on top of the
steady state, we add five Gaussian hump density perturbations to the flow as
\begin{equation}
\label{eq:numex_stella_pert_0010}
  \rho(r, \varphi) = \left(   1
                            + A \sum_{k=0}^{4}
                                  e^{- \frac{x_{k}^{2} + y_{k}^{2}}{w^{2}}}
                     \right) \rho_{eq}(r)
\end{equation}
with
\begin{equation}
\label{eq:numex_stella_pert_0020}
  \begin{aligned}
    x_{k} & = r \cos(\varphi) - r_{0} \cos(\varphi_{k}) , \\
    y_{k} & = r \sin(\varphi) - r_{0} \sin(\varphi_{k})
  \end{aligned}
\end{equation}
and $\varphi_{k} = 2 \pi k/ 5$ for $k=0, \dots, 4$.
Here, $\rho_{eq}(r)$ is the subsonic and supersonic steady state,
respectively, and $r_{0}$ is the radius, $w$ the width and $A$ the amplitude of
the perturbations.
The boundary conditions are kept frozen at the initial steady state.

For the subsonic case, we set $r_{0}=110$ km and $w=10$ km.
The final time is
$\tend = 1.247314 \times 10^{-3} ~ \mathrm{s}~ \approx 0.36 ~ \sndxtime$.
The small amplitude test is run with $A = 10^{-3}$ and the relative
perturbation errors are displayed in \cref{tab:numex_stella_pert_0010}.
The errors of the well-balanced scheme are consistently smaller by 2-3 orders
of magnitude than errors of the unbalanced scheme.
It is clear that the well-balanced scheme is vastly superior in resolving the
small perturbations.
Furthermore, we observe that the errors of the well-balanced scheme at the
lowest resolution are comparable to the errors of the unbalanced scheme at the
highest resolution.
This is further illustrated in the left panel of
\cref{fig:numex_stella_pert_0010} from which it is apparent that the unbalanced
scheme suffers from large spurious deviations.
Both schemes attain their design second-order accuracy.

For the supersonic case, we use the parameters $r_{0}=375$ km and $w=20$ km.
The final time is
$\tend = 4.518538 \times 10^{-3} ~ \mathrm{s}~ \approx 0.3 0~ \sndxtime$.
The small amplitude test is run with $A = 10^{-3}$ and the relative
perturbation errors are displayed in \cref{tab:numex_stella_pert_0030}.
Like in the previous case, we observe that the errors of the well-balanced
scheme are consistently smaller by several orders of magnitude.
This is further highlighted in the right panel of
\cref{fig:numex_stella_pert_0010} from which it is apparent that the unbalanced
scheme suffers from large spurious differences.

To assess the robustness of the schemes, we run both the subsonic and the
supersonic test case with a hundred times greater perturbation $A = 10^{-1}$.
The final times are identical to the respective small amplitude experiments.
The results are shown in \cref{tab:numex_stella_pert_0020} for the subsonic
case and in \cref{tab:numex_stella_pert_0040}.
As to be expected, the difference between the well-balanced and unbalanced
schemes decreases as the size of the perturbation is increased.
Therefore, we observe that there is no loss in robustness and resolution
capability for large amplitude perturbations with the well-balanced scheme
compared to the unbalanced one.

\begin{table}[htb!]
  \centering
  \begin{tabular}{| r || c | c | c |}
    \hline
    $(N_{r}, N_{\varphi})$ & $relerr_{1}(\delta \rho)$
                           & $relerr_{1}(\delta v_{r})$
                           & $relerr_{1}(\delta p)$          \\
    \hline \hline
      ( 32,  64) & 6.01E-03 / 6.22E-06 & 4.65E-02 / 5.66E-05
                 & 8.63E-03 / 1.39E-06                       \\
      ( 64, 128) & 1.27E-03 / 2.59E-06 & 9.97E-03 / 2.46E-05
                 & 1.66E-03 / 4.66E-07                       \\
      (128, 256) & 2.85E-04 / 7.66E-07 & 2.31E-03 / 6.93E-06
                 & 3.64E-04 / 1.19E-07                       \\
      (256, 512) & 6.71E-05 / 1.23E-07 & 5.57E-04 / 1.13E-06
                 & 8.50E-05 / 2.28E-08                       \\
    \hline
    Order        &     2.16 / 1.87     &     2.13 / 1.88
                 &     2.22 / 1.98                           \\
    \hline
  \end{tabular}
  \caption{Relative perturbation error in density, radial velocity and pressure
           for the subsonic accretion equilibrium with small amplitude density
           perturbations computed with the un-/well-balanced second-order
           schemes.}
  \label{tab:numex_stella_pert_0010}
\end{table}

\begin{table}[htb!]
  \centering
  \begin{tabular}{| r || c | c | c |}
    \hline
    $(N_{r}, N_{\varphi})$ & $relerr_{1}(\delta \rho)$
                           & $relerr_{1}(\delta v_{r})$
                           & $relerr_{1}(\delta p)$          \\
    \hline \hline
      ( 32,  64) & 6.26E-03 / 6.16E-04 & 4.87E-02 / 5.43E-03
                 & 8.63E-03 / 1.33E-04                       \\
      ( 64, 128) & 1.44E-03 / 2.57E-04 & 1.14E-02 / 2.32E-03
                 & 1.67E-03 / 4.46E-05                       \\
      (128, 256) & 3.42E-04 / 7.61E-05 & 2.71E-03 / 6.42E-04
                 & 3.65E-04 / 1.11E-05                       \\
      (256, 512) & 7.52E-05 / 1.22E-05 & 6.10E-04 / 1.03E-04
                 & 8.51E-05 / 2.14E-06                       \\
    \hline
    Order        &     2.12 / 1.87     &     2.10 / 1.90
                 &     2.22 / 1.99                           \\
    \hline
  \end{tabular}
  \caption{Relative perturbation error in density, radial velocity and pressure
           for the subsonic accretion equilibrium with large amplitude density
           perturbations computed with the un-/well-balanced second-order
           schemes.}
  \label{tab:numex_stella_pert_0020}
\end{table}

\begin{table}[htb!]
  \centering
  \begin{tabular}{| r || c | c | c |}
    \hline
    $(N_{r}, N_{\varphi})$ & $relerr_{1}(\delta \rho)$
                           & $relerr_{1}(\delta v_{r})$
                           & $relerr_{1}(\delta p)$          \\
    \hline \hline
      ( 32,  64) & 4.74E-05 / 7.84E-06 & 2.97E-04 / 8.41E-08
                 & 5.52E-05 / 2.14E-07                       \\
      ( 64, 128) & 1.36E-05 / 4.82E-06 & 7.47E-05 / 5.27E-08
                 & 9.95E-06 / 7.63E-08                       \\
      (128, 256) & 3.75E-06 / 1.68E-06 & 1.88E-05 / 1.88E-08
                 & 2.06E-06 / 2.47E-08                       \\
      (256, 512) & 9.34E-07 / 4.36E-07 & 4.70E-06 / 4.79E-09
                 & 4.67E-07 / 6.47E-09                       \\
    \hline
    Order        &     2.12 / 1.40     &     1.99 / 1.39
                 &     2.29 / 1.68                           \\
    \hline
  \end{tabular}
  \caption{Relative perturbation error in density, radial velocity and pressure
           for the supersonic accretion equilibrium with small amplitude
           density perturbations computed with the un-/well-balanced
           second-order schemes.}
  \label{tab:numex_stella_pert_0030}
\end{table}

\begin{table}[htb!]
  \centering
  \begin{tabular}{| r || c | c | c |}
    \hline
    $(N_{r}, N_{\varphi})$ & $relerr_{1}(\delta \rho)$
                           & $relerr_{1}(\delta v_{r})$
                           & $relerr_{1}(\delta p)$          \\
    \hline \hline
      ( 32,  64) & 8.49E-04 / 7.62E-04 & 2.98E-04 / 7.66E-06
                 & 6.36E-05 / 2.27E-05                       \\
      ( 64, 128) & 4.89E-04 / 4.73E-04 & 7.64E-05 / 4.95E-06
                 & 1.47E-05 / 7.46E-06                       \\
      (128, 256) & 1.67E-04 / 1.65E-04 & 1.97E-05 / 1.74E-06
                 & 4.01E-06 / 2.33E-06                       \\
      (256, 512) & 4.30E-05 / 4.26E-05 & 4.87E-06 / 4.36E-07
                 & 9.62E-07 / 5.96E-07                       \\
    \hline
    Order        &     1.45 / 1.40     &     1.98 / 1.39
                 &     2.00 / 1.74                           \\
    \hline
  \end{tabular}
  \caption{Relative perturbation error in density, radial velocity and pressure
           for the supersonic accretion equilibrium with large amplitude
           density perturbations computed with the un-/well-balanced
           second-order schemes.}
  \label{tab:numex_stella_pert_0040}
\end{table}

\begin{figure}[hbt!]
  \centering
  \begin{subfigure}[b]{0.5\textwidth}
    \centering
    \includegraphics[width=\textwidth]{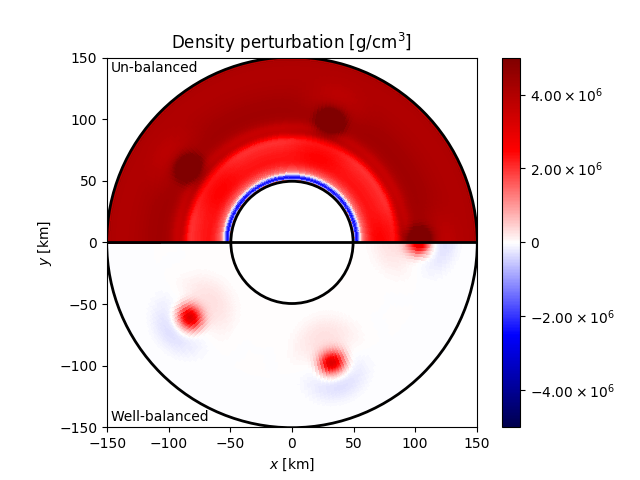}
    \caption{Subsonic steady state}
    \label{fig:numex_stella_pert_0011}
  \end{subfigure}%
  \begin{subfigure}[b]{0.5\textwidth}
    \centering
    \includegraphics[width=\textwidth]{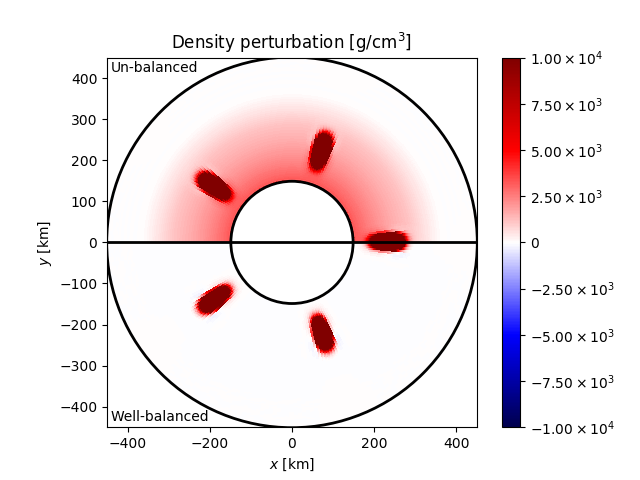}
    \caption{Supersonic steady state}
    \label{fig:numex_stella_pert_0012}
  \end{subfigure}
  \caption{Small amplitude density perturbations on the subsonic (left panel)
           and supersonic (right panel) steady state.
           In both panels, the lower/upper plane show the results obtained
           with the un-/well-balanced schemes at resolution
           $(N_{r}, N_{\varphi}) = (128, 256)$, respectively.
           The color axis is clipped to highlight the absence of spurious
           deviations away from the perturbations for the well-balanced
           scheme.}
  \label{fig:numex_stella_pert_0010}
\end{figure}

\subsubsection{Full problem}
\label{subsubsec:numex_stella_sasi}
As a final case, we test the ability of the schemes to preserve the full
problem joining the sub- and super-sonic regions with a standing shock.
We evolve the setup for several characteristic time scales
$\tend = 4 \sndxtime$.
The outer radial boundary is kept frozen at the initial state and we impose
outflow conditions at the lower boundary.
As noted previously, the HLL Riemann solver will not exactly resolve
stationary shocks.
Therefore, we can not expect the well-balanced scheme to exactly preserve the
transonic steady state.

In the right panel of \cref{fig:numex_stella_sasi_0010}, we show the shock
radius as a function of time for several resolutions
$(N_{r}, N_{\varphi}) = (128, 256), (256, 512), (512, 1024)$.
The latter is simply evaluated by determining the radius of the maximum
absolute difference in radial velocity.
From the figure, it is apparent that the well-balanced scheme (blue lines) is
able to preserve the initial shock position very well.
On the other hand, the shock position deviates for the unbalanced schemes.
This is further illustrated in the left panel of
\cref{fig:numex_stella_sasi_0010}, where we display a contour of radial
velocity of the initial condition together with the results obtained with the
well-balanced and unbalanced schemes after two characteristic time scales.
We observe that the results of the well-balanced scheme are virtually
indistinguishable from the initial conditions.

However, we note that a thorough analysis of the standing accretion shock
instability onset and dynamics is beyond the scope of the present paper.

\begin{figure}[hbt!]
  \centering
  \begin{subfigure}[b]{0.5\textwidth}
    \centering
    \includegraphics[width=\textwidth]{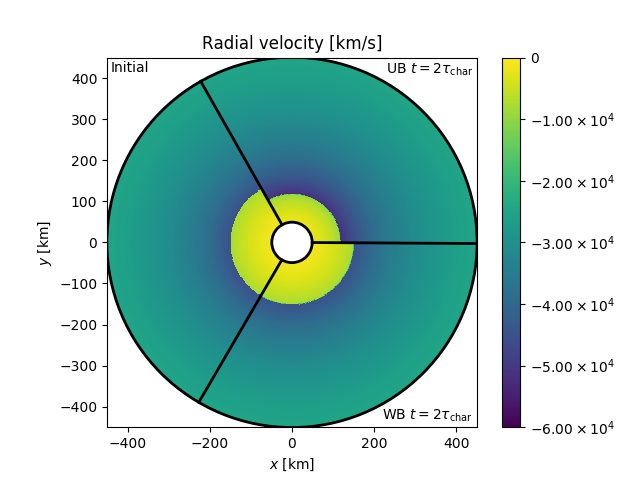}
    \caption{Accretion shock}
    \label{fig:numex_stella_sasi_0011}
  \end{subfigure}%
  \begin{subfigure}[b]{0.5\textwidth}
    \centering
    \includegraphics[width=\textwidth]{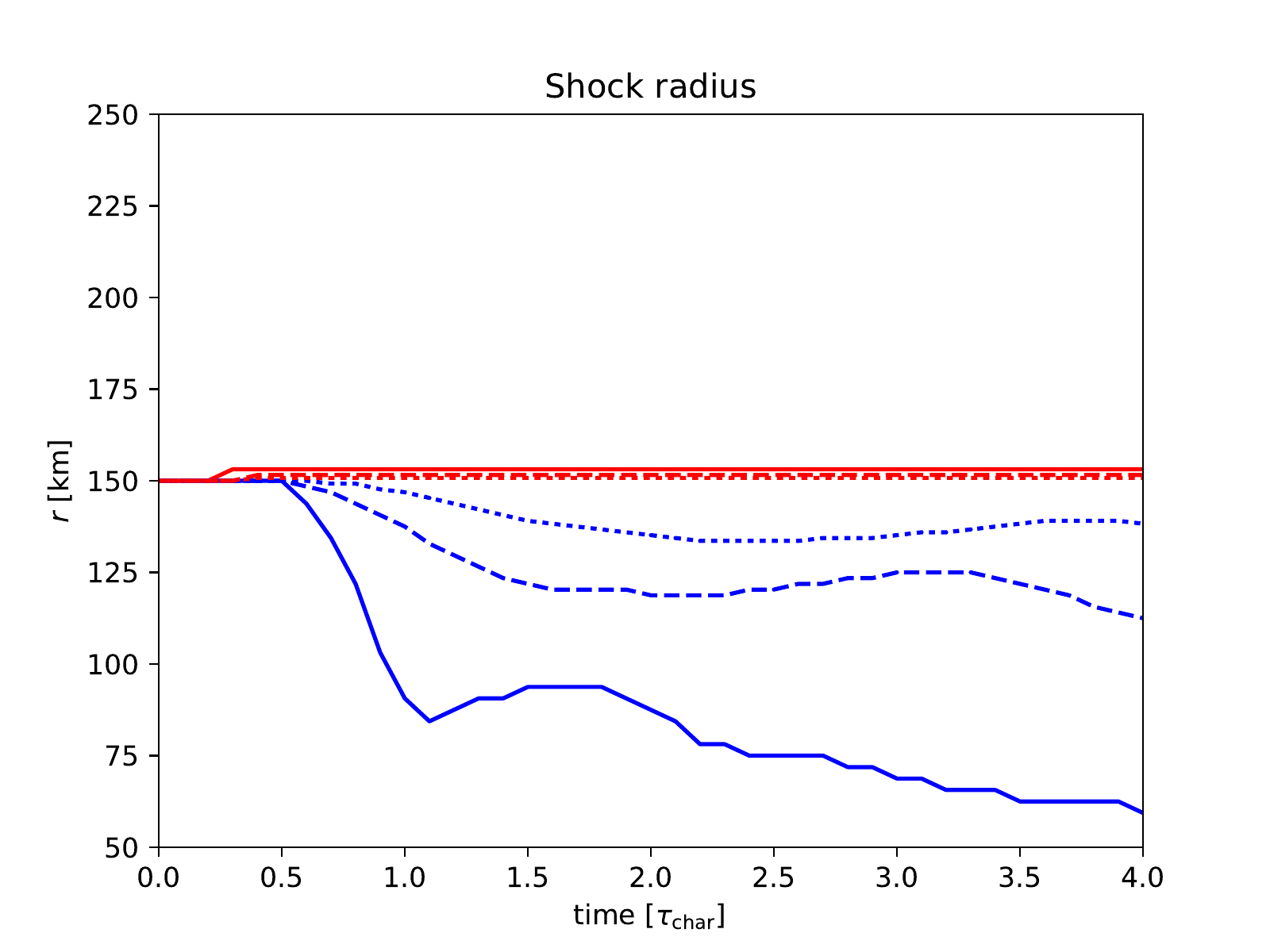}
    \caption{Shock radius}
    \label{fig:numex_stella_sasi_0012}
  \end{subfigure}
  \caption{The left panel shows the radial velocity contour for resolution
           $(N_{r}, N_{\varphi}) = (256, 512)$.
           The right panel shows the shock radius as a function of time for
           un-/well-balanced (blue/red lines) for resolutions
           $(N_{r}, N_{\varphi}) = (128, 256), (256, 512), (512, 1024)$
           (solid, dashed, dotted lines).}
  \label{fig:numex_stella_sasi_0010}
\end{figure}

\clearpage

\section{Conclusion}
\label{sec:conc}
In this paper, we have presented novel well-balanced first- and second-order
accurate finite volume schemes for the Euler equations with gravity.
The schemes are able to exactly (up to round-off errors) preserve any
one-dimensional steady adiabatic flow.
Flows of this type are an idealized model for accretion and wind phenomena
commonly encountered in astrophysics.
The method is based on a local equilibrium reconstruction combined with a
well-balanced source term discretization.
The schemes are extended to cylindrical and spherical geometries.
A dimension-by-dimension extension to multiple dimensions is also proposed.
However, the latter is only exactly well-balanced for multi-dimensional states
with streamlines aligned along a computational axis.
The schemes' performance and robustness are verified on several numerical
experiments.
The last test case consists of a model stellar accretion setup commonly
encountered in core-collapse supernovae scenarios and features a complex
multi-physics equation of state.

The current paper deals with a large class of adiabatic steady states.
However, there are scenarios where the flow does not proceed adiabatically,
e.g. due to radiation losses.
This is especially the case in astrophysics.
Developing schemes capable of balancing such non-adiabatic steady states is
indeed worthwhile.
Moreover, the current schemes are limited to steady states with streamlines
aligned with one computational axis.
Although the usage of curvilinear coordinates may help to deal with this
limitation, it would be computationally desirable to remove this restriction.
For instance, cylindrical and spherical coordinates feature coordinate
singularities which have implications for the resolution and regularity of the
grid and, thereby, the size of time steps. An extension beyond second-order
accuracy is also highly desirable. Such extensions are subject to current
research and will be dealt with in forthcoming publications.

\section*{Acknowledgments}
The work was supported by the Swiss National Science Foundation (SNSF) under
grant 200021-169631.
The authors acknowledge the computational resources provided by the EULER
cluster of ETHZ.
The one-dimensional numerical algorithm was implemented in Julia
\cite{Bezanson2017}. All post-processing and plotting was done using the
outstanding Python packages NumPy and SciPy \cite{scipy}, and Matplotlib
\cite{matplotlib}.
\bibliographystyle{elsarticle-num.bst}

\let\jnl=\rm
\def\aap{\jnl{A\&A}}               %
\def\apj{\jnl{ApJ}}                %
\def\apjl{\jnl{ApJ}}               %
\def\apjs{\jnl{ApJS}}              %
\def\physrep{\jnl{Phys.~Rep.}} %

\bibliography{refs.bib}

\appendix
\section{Local equilibrium determination for general \eos}
\label{sec:app_loc_eq_gen}
Evaluating the equilibrium $\bW_{eq,i}$ in $x$ for a general convex \eos
requires a slightly different algorithm than the one presented in the main
text, i.e. \cref{algo:equilibrium_ideal}.
The differences are mainly due to the fact that
$\Be_\ast$ and $\rho_{\ast}$ cannot be computed analytically. Therefore, instead
of comparing $\rho^{(k)}$ with $\rho_\ast$ to determine if it switched from the
sub- to the supersonic branch (or vice versa), we propose to look at the
derivative of $\Be$ instead. The algorithm terminates for four reasons: a) the
equilibrium is found, b) the algorithm fails to make any progress towards a
root, c) the algorithm has converged towards the minimum, or d) the maximum
number of iterations is reached. In the latter three cases we consider the
algorithm to have failed to find an equilibrium. Subsequently, that cell will be
marked and the standard reconstruction and source term are used in that cell.

\begin{algorithm}[htb]
  Initial guess $\rho^{(0)} = \rho_{0}$\;
  \For{k = 0, 1, 2, ...}{
    $\rho^{(trial)} = \rho^{(k)}
    - \sign(\frac{\Be(\rho^{(k)})}{\Be^\prime(\rho^{(k)})})
    \min(
    \abs{\frac{\Be(\rho^{(k)})}{\Be^\prime(\rho^{(k)})}},
    \rho^{(k)}/4
    )$\;

    $\rho^{(trial)} = \max(0, \rho^{(trial)})$\;

    \For{$\ell$ = 0, 1, 2, ...} {
      \If{$\Be^\prime(\rho^{(0)}) ~ \Be^\prime(\rho^{(trial)}) < 0$} {
        $\rho^{(trial)} = \frac{1}{2}(\rho^{(trial)} + \rho^{(k)})$
      }
    }

    $\Delta \rho = \rho^{(trial)} - \rho^{(k)}$\;
    $\rho^{(k+1)} = \rho^{(trial)}$\;

    \If{$\abs{\Be(\rho^{(k+1)})} < tol~\abs{\Be_0 - \phi(x)}$} {
      Successfully found the equilibrium. \\
      Stop
    }

    \If{$\abs{\Delta \rho} < tol~\max(\rho^{(0)}, \rho^{(k+1)})$} {
      Failed due to lacking progress per step. \\
      Stop
    }

    \If{$\abs{\Be^\prime(\rho^{(k+1)}) \rho^{(k+1)}} < tol~\abs{\Be(\rho^{(k+1)})}$} {
      Converged to minimum. \\
      Stop
    }
  }
  \caption{Local equilibrium determination for a general convex \eos.}
  \label{algo:equilibrium_general}
\end{algorithm}
\clearpage
\section{Convergence tables}
In this section we present the convergence tables for the numerical experiments
presented in \cref{sec:numex}.

\begin{table}[htbp]
  \begin{center}
    \begin{tabular}{r
                S[table-format=3.2e2]r
                S[table-format=3.2e2]r
                S[table-format=3.2e2]r
}
\toprule
N & 
\multicolumn{2}{c}{unbalanced} & 
\multicolumn{2}{c}{hydrostatic} & 
\multicolumn{2}{c}{full} \\
$M = \num{0}$ & 
\multicolumn{1}{c}{$\err_1($$\delta p$$)$} & \multicolumn{1}{c}{Rate} & 
\multicolumn{1}{c}{$\err_1($$\delta p$$)$} & \multicolumn{1}{c}{Rate} & 
\multicolumn{1}{c}{$\err_1($$\delta p$$)$} & \multicolumn{1}{c}{Rate} \\
\midrule
 32 &  3.38e-05  &      --  &  6.78e-15  &      --  &  6.70e-15  &      --   \\
 64 &  7.05e-06  &     2.26 &  6.52e-15  &     0.06 &  6.85e-15  &    -0.03  \\
128 &  1.60e-06  &     2.14 &  6.65e-15  &    -0.03 &  6.14e-15  &     0.16  \\
256 &  3.79e-07  &     2.07 &  6.82e-15  &    -0.04 &  6.80e-15  &    -0.15  \\
512 &  9.23e-08  &     2.04 &  7.38e-15  &    -0.11 &  7.06e-15  &    -0.05  \\
1024 &  2.28e-08  &     2.02 &  7.04e-15  &     0.07 &  6.96e-15  &     0.02  \\
2048 &  5.65e-09  &     2.01 &  7.50e-15  &    -0.09 &  5.32e-15  &     0.39  \\
\bottomrule
\end{tabular}
    \begin{tabular}{r
                S[table-format=3.2e2]r
                S[table-format=3.2e2]r
                S[table-format=3.2e2]r
}
\toprule
N & 
\multicolumn{2}{c}{unbalanced} & 
\multicolumn{2}{c}{hydrostatic} & 
\multicolumn{2}{c}{full} \\
$M = \num{0.01}$ & 
\multicolumn{1}{c}{$\err_1($$\delta p$$)$} & \multicolumn{1}{c}{Rate} & 
\multicolumn{1}{c}{$\err_1($$\delta p$$)$} & \multicolumn{1}{c}{Rate} & 
\multicolumn{1}{c}{$\err_1($$\delta p$$)$} & \multicolumn{1}{c}{Rate} \\
\midrule
 32 &  1.27e-04  &      --  &  1.27e-04  &      --  &  4.64e-15  &      --   \\
 64 &  2.36e-05  &     2.43 &  2.63e-05  &     2.27 &  5.00e-15  &    -0.11  \\
128 &  5.23e-06  &     2.17 &  6.32e-06  &     2.06 &  4.92e-15  &     0.02  \\
256 &  1.24e-06  &     2.08 &  1.61e-06  &     1.97 &  5.51e-15  &    -0.16  \\
512 &  3.02e-07  &     2.04 &  4.25e-07  &     1.92 &  4.90e-15  &     0.17  \\
1024 &  7.45e-08  &     2.02 &  1.07e-07  &     2.00 &  4.70e-15  &     0.06  \\
2048 &  1.85e-08  &     2.01 &  2.64e-08  &     2.01 &  4.34e-15  &     0.11  \\
\bottomrule
\end{tabular}
    \begin{tabular}{r
                S[table-format=3.2e2]r
                S[table-format=3.2e2]r
                S[table-format=3.2e2]r
}
\toprule
N & 
\multicolumn{2}{c}{unbalanced} & 
\multicolumn{2}{c}{hydrostatic} & 
\multicolumn{2}{c}{full} \\
$M = \num{2.5}$ & 
\multicolumn{1}{c}{$\err_1($$\delta p$$)$} & \multicolumn{1}{c}{Rate} & 
\multicolumn{1}{c}{$\err_1($$\delta p$$)$} & \multicolumn{1}{c}{Rate} & 
\multicolumn{1}{c}{$\err_1($$\delta p$$)$} & \multicolumn{1}{c}{Rate} \\
\midrule
 32 &  4.28e-04  &      --  &  5.19e-04  &      --  &  6.42e-13  &      --   \\
 64 &  8.98e-05  &     2.25 &  1.13e-04  &     2.21 &  6.40e-13  &     0.00  \\
128 &  2.06e-05  &     2.13 &  2.62e-05  &     2.10 &  6.34e-13  &     0.01  \\
256 &  4.92e-06  &     2.06 &  6.34e-06  &     2.05 &  6.26e-13  &     0.02  \\
512 &  1.20e-06  &     2.03 &  1.56e-06  &     2.02 &  6.02e-13  &     0.06  \\
1024 &  2.98e-07  &     2.02 &  3.86e-07  &     2.01 &  5.95e-13  &     0.02  \\
2048 &  7.41e-08  &     2.01 &  9.61e-08  &     2.01 &  5.29e-13  &     0.17  \\
\bottomrule
\end{tabular}
    \caption{Convergence table of the second-order methods for
      \cref{numex:gaussian_bump}, with $M = 0, 0.01, 2.5$ and $A = 0$.}
    \label{tab:gaussian_bump_wb}
  \end{center}
\end{table}

\begin{table}[htbp]
  \begin{center}
    \begin{tabular}{r
                S[table-format=3.2e2]r
                S[table-format=3.2e2]r
                S[table-format=3.2e2]r
}
\toprule
N & 
\multicolumn{2}{c}{unbalanced} & 
\multicolumn{2}{c}{hydrostatic} & 
\multicolumn{2}{c}{full} \\
$M = \num{0}$ & 
\multicolumn{1}{c}{$\err_1($$\delta p$$)$} & \multicolumn{1}{c}{Rate} & 
\multicolumn{1}{c}{$\err_1($$\delta p$$)$} & \multicolumn{1}{c}{Rate} & 
\multicolumn{1}{c}{$\err_1($$\delta p$$)$} & \multicolumn{1}{c}{Rate} \\
\midrule
 32 &  3.90e-06  &      --  &  3.32e-08  &      --  &  3.33e-08  &      --   \\
 64 &  9.80e-07  &     1.99 &  1.25e-08  &     1.40 &  1.25e-08  &     1.41  \\
128 &  2.46e-07  &     2.00 &  4.15e-09  &     1.60 &  4.13e-09  &     1.60  \\
256 &  6.14e-08  &     2.00 &  1.08e-09  &     1.94 &  1.14e-09  &     1.85  \\
512 &  1.53e-08  &     2.00 &  2.77e-10  &     1.97 &  2.95e-10  &     1.95  \\
1024 &  3.83e-09  &     2.00 &  6.93e-11  &     2.00 &  6.92e-11  &     2.09  \\
2048 &  9.58e-10  &     2.00 &  1.67e-11  &     2.05 &  1.67e-11  &     2.05  \\
\bottomrule
\end{tabular}
    \begin{tabular}{r
                S[table-format=3.2e2]r
                S[table-format=3.2e2]r
                S[table-format=3.2e2]r
}
\toprule
N & 
\multicolumn{2}{c}{unbalanced} & 
\multicolumn{2}{c}{hydrostatic} & 
\multicolumn{2}{c}{full} \\
$M = \num{0.01}$ & 
\multicolumn{1}{c}{$\err_1($$\delta p$$)$} & \multicolumn{1}{c}{Rate} & 
\multicolumn{1}{c}{$\err_1($$\delta p$$)$} & \multicolumn{1}{c}{Rate} & 
\multicolumn{1}{c}{$\err_1($$\delta p$$)$} & \multicolumn{1}{c}{Rate} \\
\midrule
 32 &  2.16e-05  &      --  &  1.57e-05  &      --  &  3.15e-08  &      --   \\
 64 &  4.27e-06  &     2.34 &  2.76e-06  &     2.51 &  1.26e-08  &     1.33  \\
128 &  9.77e-07  &     2.13 &  5.98e-07  &     2.21 &  4.15e-09  &     1.60  \\
256 &  2.35e-07  &     2.06 &  1.40e-07  &     2.09 &  1.13e-09  &     1.88  \\
512 &  5.77e-08  &     2.03 &  3.39e-08  &     2.04 &  2.91e-10  &     1.95  \\
1024 &  1.43e-08  &     2.01 &  8.36e-09  &     2.02 &  6.87e-11  &     2.08  \\
2048 &  3.56e-09  &     2.01 &  2.07e-09  &     2.01 &  1.66e-11  &     2.05  \\
\bottomrule
\end{tabular}
    \begin{tabular}{r
                S[table-format=3.2e2]r
                S[table-format=3.2e2]r
                S[table-format=3.2e2]r
}
\toprule
N & 
\multicolumn{2}{c}{unbalanced} & 
\multicolumn{2}{c}{hydrostatic} & 
\multicolumn{2}{c}{full} \\
$M = \num{2.5}$ & 
\multicolumn{1}{c}{$\err_1($$\delta p$$)$} & \multicolumn{1}{c}{Rate} & 
\multicolumn{1}{c}{$\err_1($$\delta p$$)$} & \multicolumn{1}{c}{Rate} & 
\multicolumn{1}{c}{$\err_1($$\delta p$$)$} & \multicolumn{1}{c}{Rate} \\
\midrule
 32 &  1.51e-04  &      --  &  1.99e-04  &      --  &  2.96e-08  &      --   \\
 64 &  3.19e-05  &     2.24 &  4.37e-05  &     2.18 &  1.11e-08  &     1.42  \\
128 &  7.34e-06  &     2.12 &  1.03e-05  &     2.09 &  4.16e-09  &     1.41  \\
256 &  1.76e-06  &     2.06 &  2.49e-06  &     2.04 &  1.18e-09  &     1.81  \\
512 &  4.30e-07  &     2.03 &  6.14e-07  &     2.02 &  3.20e-10  &     1.89  \\
1024 &  1.06e-07  &     2.02 &  1.52e-07  &     2.01 &  7.85e-11  &     2.03  \\
2048 &  2.65e-08  &     2.01 &  3.79e-08  &     2.01 &  1.90e-11  &     2.05  \\
\bottomrule
\end{tabular}
    \caption{Convergence table of the second-order methods for
      \cref{numex:gaussian_bump}, with $M = 0, 0.01, 2.5$ and $A = 10^{-6}$.}
    \label{tab:gaussian_bump_small}
  \end{center}
\end{table}

\begin{table}[htbp]
  \begin{center}
    \begin{tabular}{r
                S[table-format=3.2e2]r
                S[table-format=3.2e2]r
                S[table-format=3.2e2]r
}
\toprule
N & 
\multicolumn{2}{c}{unbalanced} & 
\multicolumn{2}{c}{hydrostatic} & 
\multicolumn{2}{c}{full} \\
$M = \num{0}$ & 
\multicolumn{1}{c}{$\err_1($$p$$)$} & \multicolumn{1}{c}{Rate} & 
\multicolumn{1}{c}{$\err_1($$p$$)$} & \multicolumn{1}{c}{Rate} & 
\multicolumn{1}{c}{$\err_1($$p$$)$} & \multicolumn{1}{c}{Rate} \\
\midrule
 32 &  1.29e-02  &      --  &  1.37e-02  &      --  &  1.38e-02  &      --   \\
 64 &  8.47e-03  &     0.61 &  8.33e-03  &     0.72 &  8.38e-03  &     0.72  \\
128 &  4.13e-03  &     1.04 &  3.99e-03  &     1.06 &  4.00e-03  &     1.06  \\
256 &  1.74e-03  &     1.25 &  1.76e-03  &     1.18 &  1.76e-03  &     1.18  \\
512 &  6.90e-04  &     1.33 &  7.16e-04  &     1.30 &  7.09e-04  &     1.31  \\
1024 &  4.16e-04  &     0.73 &  4.19e-04  &     0.77 &  4.18e-04  &     0.76  \\
2048 &  1.64e-04  &     1.34 &  1.65e-04  &     1.34 &  1.65e-04  &     1.34  \\
\bottomrule
\end{tabular}
    \begin{tabular}{r
                S[table-format=3.2e2]r
                S[table-format=3.2e2]r
                S[table-format=3.2e2]r
}
\toprule
N & 
\multicolumn{2}{c}{unbalanced} & 
\multicolumn{2}{c}{hydrostatic} & 
\multicolumn{2}{c}{full} \\
$M = \num{0.01}$ & 
\multicolumn{1}{c}{$\err_1($$p$$)$} & \multicolumn{1}{c}{Rate} & 
\multicolumn{1}{c}{$\err_1($$p$$)$} & \multicolumn{1}{c}{Rate} & 
\multicolumn{1}{c}{$\err_1($$p$$)$} & \multicolumn{1}{c}{Rate} \\
\midrule
 32 &  1.81e-02  &      --  &  1.85e-02  &      --  &  1.85e-02  &      --   \\
 64 &  7.49e-03  &     1.27 &  7.38e-03  &     1.33 &  7.51e-03  &     1.30  \\
128 &  3.43e-03  &     1.13 &  3.38e-03  &     1.13 &  3.36e-03  &     1.16  \\
256 &  1.30e-03  &     1.40 &  1.30e-03  &     1.38 &  1.30e-03  &     1.37  \\
512 &  8.17e-04  &     0.67 &  8.36e-04  &     0.63 &  8.33e-04  &     0.64  \\
1024 &  3.96e-04  &     1.04 &  4.00e-04  &     1.06 &  4.00e-04  &     1.06  \\
2048 &  2.07e-04  &     0.94 &  2.07e-04  &     0.95 &  2.07e-04  &     0.95  \\
\bottomrule
\end{tabular}
    \begin{tabular}{r
                S[table-format=3.2e2]r
                S[table-format=3.2e2]r
                S[table-format=3.2e2]r
}
\toprule
N & 
\multicolumn{2}{c}{unbalanced} & 
\multicolumn{2}{c}{hydrostatic} & 
\multicolumn{2}{c}{full} \\
$M = \num{2.5}$ & 
\multicolumn{1}{c}{$\err_1($$p$$)$} & \multicolumn{1}{c}{Rate} & 
\multicolumn{1}{c}{$\err_1($$p$$)$} & \multicolumn{1}{c}{Rate} & 
\multicolumn{1}{c}{$\err_1($$p$$)$} & \multicolumn{1}{c}{Rate} \\
\midrule
 32 &  3.14e-02  &      --  &  3.02e-02  &      --  &  2.87e-02  &      --   \\
 64 &  1.41e-02  &     1.16 &  1.45e-02  &     1.06 &  1.04e-02  &     1.46  \\
128 &  5.32e-03  &     1.40 &  5.23e-03  &     1.47 &  4.00e-03  &     1.38  \\
256 &  1.79e-03  &     1.57 &  1.76e-03  &     1.57 &  1.86e-03  &     1.11  \\
512 &  5.61e-04  &     1.67 &  5.56e-04  &     1.66 &  5.74e-04  &     1.70  \\
1024 &  1.49e-04  &     1.91 &  1.50e-04  &     1.89 &  1.51e-04  &     1.93  \\
2048 &  3.56e-05  &     2.06 &  3.58e-05  &     2.07 &  3.61e-05  &     2.06  \\
\bottomrule
\end{tabular}
    \caption{Convergence table of the second-order methods for
      \cref{numex:gaussian_bump}, with $M = 0, 0.01, 2.5$ and $A = 1$.}
    \label{tab:gaussian_bump_shock}
  \end{center}
\end{table}

\begin{table}[htbp]
  \begin{center}
    \begin{tabular}{r
                S[table-format=3.2e2]r
                S[table-format=3.2e2]r
}
\toprule
N & 
\multicolumn{2}{c}{unbalanced} & 
\multicolumn{2}{c}{full} \\
$M = \num{0.9}$ & 
\multicolumn{1}{c}{$\err_1($$\delta p$$)$} & \multicolumn{1}{c}{Rate} & 
\multicolumn{1}{c}{$\err_1($$\delta p$$)$} & \multicolumn{1}{c}{Rate} \\
\midrule
 32 &  2.31e-01  &      --  &  3.79e-13  &      --   \\
 64 &  6.71e-02  &     1.78 &  3.73e-13  &     0.02  \\
128 &  1.28e-02  &     2.39 &  3.75e-13  &    -0.01  \\
256 &  2.20e-03  &     2.54 &  3.93e-13  &    -0.07  \\
512 &  4.38e-04  &     2.33 &  3.64e-13  &     0.11  \\
1024 &  9.62e-05  &     2.19 &  2.91e-13  &     0.32  \\
2048 &  2.24e-05  &     2.10 &  3.43e-13  &    -0.24  \\
\bottomrule
\end{tabular}
    \begin{tabular}{r
                S[table-format=3.2e2]r
                S[table-format=3.2e2]r
}
\toprule
N & 
\multicolumn{2}{c}{unbalanced} & 
\multicolumn{2}{c}{full} \\
$M = \num{2}$ & 
\multicolumn{1}{c}{$\err_1($$\delta p$$)$} & \multicolumn{1}{c}{Rate} & 
\multicolumn{1}{c}{$\err_1($$\delta p$$)$} & \multicolumn{1}{c}{Rate} \\
\midrule
 32 &  2.03e-03  &      --  &  4.74e-14  &      --   \\
 64 &  6.41e-04  &     1.66 &  4.72e-14  &     0.00  \\
128 &  1.80e-04  &     1.84 &  4.65e-14  &     0.02  \\
256 &  4.76e-05  &     1.92 &  4.59e-14  &     0.02  \\
512 &  1.22e-05  &     1.96 &  4.62e-14  &    -0.01  \\
1024 &  3.11e-06  &     1.98 &  4.23e-14  &     0.13  \\
2048 &  7.83e-07  &     1.99 &  2.99e-14  &     0.50  \\
\bottomrule
\end{tabular}
    \caption{Convergence table of the second-order methods for
      \cref{numex:bondi}, with $M = 0.9, 2.0$ and $A = 0$.}
    \label{tab:bondi_wb}
  \end{center}
\end{table}

\begin{table}[htbp]
  \begin{center}
    \begin{tabular}{r
                S[table-format=3.2e2]r
                S[table-format=3.2e2]r
}
\toprule
N & 
\multicolumn{2}{c}{unbalanced} & 
\multicolumn{2}{c}{full} \\
$M = \num{0.9}$ & 
\multicolumn{1}{c}{$\err_1($$\delta p$$)$} & \multicolumn{1}{c}{Rate} & 
\multicolumn{1}{c}{$\err_1($$\delta p$$)$} & \multicolumn{1}{c}{Rate} \\
\midrule
 32 &  9.17e-02  &      --  &  3.27e-06  &      --   \\
 64 &  1.41e-02  &     2.70 &  1.14e-06  &     1.52  \\
128 &  2.05e-03  &     2.78 &  4.01e-07  &     1.51  \\
256 &  3.70e-04  &     2.47 &  1.03e-07  &     1.96  \\
512 &  7.75e-05  &     2.26 &  2.68e-08  &     1.94  \\
1024 &  1.77e-05  &     2.13 &  6.62e-09  &     2.02  \\
2048 &  4.21e-06  &     2.07 &  1.56e-09  &     2.09  \\
\bottomrule
\end{tabular}
    \begin{tabular}{r
                S[table-format=3.2e2]r
                S[table-format=3.2e2]r
}
\toprule
N & 
\multicolumn{2}{c}{unbalanced} & 
\multicolumn{2}{c}{full} \\
$M = \num{2}$ & 
\multicolumn{1}{c}{$\err_1($$\delta p$$)$} & \multicolumn{1}{c}{Rate} & 
\multicolumn{1}{c}{$\err_1($$\delta p$$)$} & \multicolumn{1}{c}{Rate} \\
\midrule
 32 &  1.09e-03  &      --  &  4.01e-06  &      --   \\
 64 &  3.74e-04  &     1.54 &  1.37e-06  &     1.55  \\
128 &  1.08e-04  &     1.79 &  4.43e-07  &     1.63  \\
256 &  2.91e-05  &     1.90 &  1.15e-07  &     1.95  \\
512 &  7.53e-06  &     1.95 &  2.85e-08  &     2.01  \\
1024 &  1.92e-06  &     1.97 &  7.20e-09  &     1.99  \\
2048 &  4.84e-07  &     1.99 &  1.74e-09  &     2.05  \\
\bottomrule
\end{tabular}
    \caption{Convergence table of the second-order methods for
      \cref{numex:bondi}, with $M = 0.9, 2.0$ and $A = 10^{-4}$.}
    \label{tab:bondi_small}
  \end{center}
\end{table}

\begin{table}[htbp]
  \begin{center}
    \begin{tabular}{r
                S[table-format=3.2e2]r
                S[table-format=3.2e2]r
}
\toprule
N & 
\multicolumn{2}{c}{unbalanced} & 
\multicolumn{2}{c}{full} \\
$M = \num{0.9}$ & 
\multicolumn{1}{c}{$\err_1($$p$$)$} & \multicolumn{1}{c}{Rate} & 
\multicolumn{1}{c}{$\err_1($$p$$)$} & \multicolumn{1}{c}{Rate} \\
\midrule
 32 &  1.42e+00  &      --  &  1.40e+00  &      --   \\
 64 &  1.06e+00  &     0.42 &  1.02e+00  &     0.46  \\
128 &  6.15e-01  &     0.79 &  6.04e-01  &     0.75  \\
256 &  3.46e-01  &     0.83 &  3.48e-01  &     0.80  \\
512 &  1.90e-01  &     0.86 &  1.90e-01  &     0.87  \\
1024 &  9.64e-02  &     0.98 &  9.66e-02  &     0.98  \\
2048 &  4.46e-02  &     1.11 &  4.47e-02  &     1.11  \\
\bottomrule
\end{tabular}
    \caption{Convergence table of the second-order methods for
      \cref{numex:bondi}, with $M = 0.9$ and $A = 100$.}
    \label{tab:bondi_large}
  \end{center}
\end{table}

\end{document}